\documentclass[reqno]{amsart}

\theoremstyle{plain}
\newtheorem{theorem}{Theorem}[section]

\newtheorem{corollary}[theorem]{Corollary}
\newtheorem{lemma}[theorem]{Lemma}

\newtheorem{principle}{Principle}

\theoremstyle{definition}
\newtheorem{definition}[theorem]{Definition}
\newtheorem{example}[theorem]{Example}

\theoremstyle{remark}

\newtheorem*{notation}{Notation}

\numberwithin{equation}{section}
\numberwithin{figure}{section}
\numberwithin{table}{section}

\newcommand{\texorpdfstring}[2]{#1}

\newcommand{\defequal}{\mathbin{\underset{\text{def}}{=}}}

\renewcommand{\to}{\longrightarrow}
\newcommand{\mapto}{\longmapsto}

\newcommand{\namedto}[1]{\xrightarrow{#1}}

\newcommand{\parens}[1]{{\left({#1}\right)}}
\newcommand{\bracket}[1]{\left\langle{#1}\right\rangle}
\newcommand{\eqclass}[1]{\left[{#1}\right]}

\newcommand{\pair}[2]{{\parens{#1, #2}}}
\newcommand{\triple}[3]{{\parens{#1, #2, #3}}}

\newcommand{\set}[1]{\left\lbrace{#1}\right\rbrace}

\newcommand{\sequence}[2]{\parens{{#1}_{#2}}}

\newcommand{\sequenced}[1]{\parens{#1}}
\newcommand{\family}[2]{\parens{{#1}_{#2}}}

\newcommand{\cover}[2]{\set{{#1}_{#2}}}

\newcommand{\abs}[1]{\left\lvert{#1}\right\rvert}
\newcommand{\norm}[1]{\left\lVert{#1}\right\rVert}
\newcommand{\inner}[2]{\bracket{{#1}, {#2}}}

\newcommand{\inverse}[1]{{{#1}^{-1}}}

\newcommand{\power}[2]{{{#1}^{#2}}}
\newcommand{\exponential}[2]{{{#1}^{#2}}}
\newcommand{\quotient}{\mathbin{/}}

\newcommand{\restriction}[2]{{#1}\lvert_{#2}}

\newcommand{\Natural}{\mathbb{N}}               %Natural numbers.
\newcommand{\Real}{\mathbb{R}}                  %Reals.
\newcommand{\field}{\mathbb{K}}                 %Generic field.

\DeclareMathOperator{\Set}{\mathbf{Set}}            %Sets.
\DeclareMathOperator{\Bool}{\mathbf{Bool}}          %Boolean algebras.
\DeclareMathOperator{\Meas}{\mathbf{Meas}}          %Measures = algebras.
\DeclareMathOperator{\Vect}{\mathbf{Vect}}          %Linear spaces.
\DeclareMathOperator{\CHaus}{\mathbf{CHaus}}        %Compact Hausdorff.
\DeclareMathOperator{\BoolTop}{\mathbf{BoolTop}}    %Boolean spaces.
\DeclareMathOperator{\Cobord}{\mathbf{Cobord}}      %Cobordims.
\DeclareMathOperator{\Ban}{\mathbf{Ban}}            %Banach spaces.
\DeclareMathOperator{\Banc}{\Ban_{c}}               %Linear contractions.
\DeclareMathOperator{\BanAlg}{\mathbf{BAlg}}        %Banach algebras.
\DeclareMathOperator{\Hilb}{\mathbf{Hilb}}          %Hibert spaces.
\DeclareMathOperator{\PShv}{\mathbf{PShv}}          %Presheaves.
\DeclareMathOperator{\PCoShv}{\mathbf{PCoShv}}      %Precosheaves.
\DeclareMathOperator{\Shv}{\mathbf{Shv}}            %Sheaves.
\DeclareMathOperator{\CoShv}{\mathbf{CoShv}}        %Cosheaves.
\DeclareMathOperator{\Bundle}{\mathbf{Bun}}         %Bundles.

\DeclareMathOperator{\Cat}{\mathbf{Cat}}            %Categories.

\newcommand{\homset}[3]{{#1}\parens{#2, #3}}
\newcommand{\objectmatrix}[3]{{\eqclass{{#1}_{#2}^{#3}}}}
\newcommand{\homsetmatrix}[3]{{{#1}_{#2}^{#3}}}
\newcommand{\isomorphic}{\mathbin{\cong}}       %Isomorphic.
\newcommand{\equivalent}{\mathbin{\simeq}}      %Equivalent.

\newcommand{\universal}[1]{\widehat{#1}}        %Universal map.
\newcommand{\zero}{\boldsymbol{0}}              %Zero, initial.
\newcommand{\id}{1}                             %Identity arrow.
\newcommand{\composition}{\mathbin{\circ}}      %Composition map.
\DeclareMathOperator{\evaluation}{ev}           %Evaluation map.
\newcommand{\diagonal}{\Delta}                  %Diagonal.
\newcommand{\characteristic}{\chi}              %Characteristic.

\newcommand{\slice}{\mathbin{/}}

\newcommand{\covariant}[1]{{{#1}_{\ast}}}
\newcommand{\contravariant}[1]{{{#1}^{\ast}}}
\newcommand{\pullback}[1]{{{#1}^{\ast}}}            %Pullback.
\newcommand{\opposite}[1]{{{#1}^{\ast}}}            %Opposite.
\newcommand{\dual}[1]{{{#1}^{\ast}}}                %Dual.
\newcommand{\bidual}[1]{{{#1}^{\ast\ast}}}          %Bidual.
\DeclareMathOperator*{\Lim}{Lim}                    %Limit.
\DeclareMathOperator*{\Colim}{Colim}                %Colimit.
\DeclareMathOperator*{\Lan}{Lan}                    %Let Kan extension.
\DeclareMathOperator{\Yoneda}{Y}                    %%Yoneda.
\DeclareMathOperator{\Free}{F}                      %Free.
\DeclareMathOperator{\LAdjoint}{L}                  %Left adjoint.
\DeclareMathOperator{\RAdjoint}{R}                  %Right adjoint.

\newcommand{\subsets}[1]{{\two^{#1}}}           %Power set.
\newcommand{\dirimage}[1]{{{#1}_{\ast}}}        %Direct image.
\newcommand{\invimage}[1]{{{#1}^{\ast}}}        %Inverse image.
\newcommand{\ordinal}[1]{\boldsymbol{#1}}       %Finite ordinal.

\newcommand{\icomplement}{c}                        %Infix complement
\newcommand{\relcomplement}{\mathbin{\setminus}}    %Relative complement.
\newcommand{\two}{{\ordinal{2}}}                    %Boolean field {0, 1}.

\DeclareMathOperator{\ideal}{\mathcal{I}}           %Generated ideal.
\DeclareMathOperator{\sieve}{S}                     %Generated sieve.

\DeclareMathOperator{\atoms}{\mathbf{atoms}}        %Atoms.
\DeclareMathOperator{\partitions}{\mathbf{parts}}   %Partitions (finite).
\DeclareMathOperator{\Stone}{S}                     %Stone functor.
\DeclareMathOperator{\Baire}{\mathfrak{Ba}}         %Baire sigma-algebra.
\DeclareMathOperator{\Borel}{\mathfrak{Bo}}         %Borel sigma-algebra.
\DeclareMathOperator{\simple}{\mathbf{S}}           %Simple functions.
\DeclareMathOperator{\additive}{\mathbf{A}}         %Additive maps.
\DeclareMathOperator{\badditive}{\mathbf{BA}}       %Bounded additive.
\DeclareMathOperator{\bvadditive}{\mathbf{BVA}}     %Bounded variation.
\DeclareMathOperator{\ladditive}{\mathbf{LA}}       %Additive, Lipschitz.
\DeclareMathOperator{\bsadditive}{\mathbf{BAS}}     %Bounded spectral maps.
\DeclareMathOperator{\rmeasure}{\mathbf{RM}}        %Sigma-additive, regular.

\DeclareMathOperator{\dimension}{dim}               %Dimension.
\DeclareMathOperator{\Null}{\mathcal{N}}            %Null object.
\DeclareMathOperator{\open}{open}               %Opens.
\DeclareMathOperator{\clopen}{clopen}           %Clopens.
\DeclareMathOperator{\supp}{supp}               %Support.

\newcommand{\converges}{\rightarrow}            %Convergence.

\DeclareMathOperator{\ball}{ball}               %Ball.
\DeclareMathOperator{\sphere}{S}                %Sphere.

\newcommand{\aeequal}
    {\mathbin{\underset{\text{a.e.}}{=}}}       %a.e. equality.
\newcommand{\differential}{\mathrm{d}}
\newcommand{\integral}[3]{\int_{#1}{#2}\,\differential{#3}}

\newcommand{\sumnorm}[1]{{\norm{#1}_{1}}}           %L_1 sum norm.
\newcommand{\hilbertnorm}[1]{\norm{#1}_{2}}         %L_2 sum norm.
\newcommand{\supnorm}[1]{\norm{#1}_{\infty}}        %L_\infty sup norm.
\newcommand{\variation}[1]{\norm{#1}}               %Variation.
\newcommand{\semivariation}[1]{\norm{#1}_{\infty}}  %Semivariation.
\newcommand{\lipschitz}[1]{\norm{#1}_{L}}           %Lipschitz norm.

\newcommand{\projnorm}[1]
    {\norm{#1}_{\hat{\phantom{a}}}}                 %Projective norm.
\DeclareMathOperator{\nuclear}{\mathcal{N}}         %Nuclear maps.
\newcommand{\projotimes}{\mathbin{\hat{\otimes}}}   %Projective tensor.
\newcommand{\injotimes}{\mathbin{\breve{\otimes}}}  %Injective tensor.

\DeclareMathOperator{\Integrable}{\mathbf{L}}       %Integrables.
\DeclareMathOperator{\Continuous}{\mathbf{C}}       %Continuous.
\DeclareMathOperator{\Bounded}{\mathbf{B}}          %Bounded maps.
\DeclareMathOperator{\CBounded}{\Continuous_{b}}    %Continuous bounded.
\newcommand{\annihilator}[1]{{{#1}^{\bot}}}         %Annihilator.
\newcommand{\orthogonal}[1]{{\annihilator{#1}}}     %Orthogonal.
\DeclareMathOperator{\wk}{wk}
\DeclareMathOperator{\wkstar}{\wk\!\ast}

\DeclareMathOperator{\wkstarContinuous}{{\Continuous^{\wkstar}}}

\newcommand{\stalk}[2]{{{#1}_{#2}}}                 %Stalk of sheaf.
\newcommand{\etalification}[1]{{\widetilde{#1}}}    %Etalification.
\newcommand{\cosheafification}[1]{{{#1}_{\ast}}}    %Cosheafification.

\DeclareMathOperator{\SBundle}{\simple}             %Simple bundles.
\DeclareMathOperator{\IBundle}{\Integrable}         %Integrables.

\newcommand{\bitensor}{\mathbin{\hat{\otimes}}}     %Bitensor products.

\newcommand{\bicatname}[1]{\ensuremath{2--{#1}}}

\DeclareMathOperator{\BanCat}{\mathbf{BanCat}}      %Banach categories.
\DeclareMathOperator{\BanModCat}{\bicatname{\Ban}}  %Banach 2-spaces.

\usepackage[all]{xy}

\xyoption{2cell}
\UseAllTwocells

\begin{document}

\title{Categorifying measure theory: a roadmap}
\author{G.~Rodrigues}

\address{Centro de An\'{a}lise Matem\'{a}tica, Geometria e Sistemas Din\^{a}micos, Departamento de Matem\'atica, Instituto Superior T\'ecnico (Universidade T\'{e}cnica de Lisboa), Av. Rovisco Pais, 1049-001 Lisboa, Portugal.}
\email{grodrigues.math@gmail.com}

\thanks{This work was supported by the \emph{Programa Operacional Ci\^{e}ncia e Inova\c{c}\~{a}o 2010}, financed by the \emph{Funda\c{c}\~{a}o para a Ci\^{e}ncia e a Tecnologia} (FCT / Portugal) and cofinanced by the European Community fund FEDER,
in part through the research project Quantum Topology POCI / MAT / 60352 / 2004.}

%\date{\today}

\begin{abstract}
A program for \emph{categorifying measure theory} is outlined.
\end{abstract}

\maketitle
\tableofcontents

%Sections.
%Main matter.
\newpage
\section{Introduction}
\label{section:introduction}

\begin{verse}
Begin, ephebe, by perceiving the idea\\
Of this invention, this invented world,\\
The inconceivable idea of the sun.\\

\medskip
You must become an ignorant man again\\
And see the sun again with an ignorant eye\\
And see it clearly in the idea of it.\\

\medskip
Never suppose an inventing mind as source\\
Of this idea nor for that mind compose\\
A voluminous master folded in his fire.\\

\medskip
How clean the sun when seen in its idea,\\
Washed in the remotest cleanliness of a heaven\\
That has expelled us and our images\ldots\\
--- Wallace Stevens, from \emph{Notes Toward a Supreme Fiction}.
\end{verse}
\medskip

%Paragraph:
%- from categorification to categories of representations.

The word ``categorification'' was first coined in \cite{crane-frenkel:4dtqft-hopf-categories-canonical-bases1994}. The context of the paper is the construction of four-dimensional \emph{topological quantum field theories} (TQFT's for short) via state sums. TQFT's were first rigorously defined by M.~Atiyah in \cite{atiyah:tqft1988}, modelled on G.~Segal's definition of conformal field theory (\cite{segal:definition-conformal-field-theory}), with the intention of formalizing E.~Witten's work \cite{witten:tqft} on $4$-manifold invariants coming from the gauge-theoretic Donaldson theory. It was the critical observation of the authors of \cite{crane-frenkel:4dtqft-hopf-categories-canonical-bases1994} that (at least in low dimensions) there is a deep connection between on one side, how the $(n + 1)$-Pachner moves relate to the $n$-Pachner moves and on the other, how the algebraic structures in $n$-categories that provide TQFT's get ``categorified'' in $(n + 1)$-categories (the so-called ``categorical ladder''). Besides \cite{crane-frenkel:4dtqft-hopf-categories-canonical-bases1994}, we refer the reader to the introduction of \cite{mackaay:finite-groups-spherical-2categories-4manifold-invariants2000} for a very lucid account of the story. Here, it suffices to say that in two dimensions, the construction of manifold invariants takes as input a \emph{semisimple algebra} (\cite{fukuma-hosono-kawai:lattice-topological-field-theory-2dimensions1994}). The combinatorics of the Pachner moves relating equivalent triangulations of the same surface exactly match the algebraic laws of associative algebras: the $2$-$2$ Pachner move corresponds to associativity of the multiplication and the $1$-$3$ move (the ``triangulation refinement'' move) corresponds to semi-simplicity, and in essence it is a cutoff, finiteness condition. The construction of $3$-manifold invariants takes as algebraic input a certain kind of monoidal category (\cite{barrett-westbury:invariants-pl-3manifolds1996}) and these arise naturally as categories of representations of quantum groups (\cite{turaev:quantum-invariants-knots-3manifolds1994}). The situation in four dimensions is more mysterious (in more senses than one), but following the categorification cue of \cite{crane-frenkel:4dtqft-hopf-categories-canonical-bases1994}, \cite{mackaay:spherical-2categories-4manifold-invariants1999} constructed $4$-manifold invariants via state sums that take as algebraic input certain kinds of monoidal $2$-categories. We currently lack non-trivial examples of such $2$-categories\footnote{Non-triviality of the $2$-category here can be taken to mean that when fed as algebraic input into the state sum machinery, it is able to detect more than just the homotopy type of the manifold.}, but once again following the categorification hint, it is expected that they arise as monoidal $2$-categories of representations.

%Paragraph:
%- representing objects as linear 2-spaces.

\smallskip
The question we are interested in answering is \emph{representations in what}? Quantum groups are represented in finite-dimensional linear spaces so we want some sort of linear $2$-spaces. Since linear spaces are free, one plausible candidate for the categorified analogues of linear spaces are categories equivalent to $\exponential{\Vect}{n}$ where $\Vect$ is the category of finite-dimensional $\field$-linear spaces with $\field$ a field\footnote{In concrete constructions, $\field$ is usually the complex numbers but at the level of generality we are working, it makes no difference whatsoever what base field one takes.}. This is precisely the definition of linear $2$-space advanced in \cite{kapranov-voevodsky:2categories-Zamolodchikov-equations1994}. These categories can be characterized intrinsically by abelianness plus semi-simplicity (see \cite{yetter:categorical-linear-algebra}). The crucial fact is that just as the linear spaces $\field^{n}$ have a canonical basis, the categories $\exponential{\Vect}{n}$ come equipped with a canonical basis of objects $e_{i}$ defined by:
\begin{equation*}
    e_{i}(j)\defequal
    \begin{cases}
        \field &\text{if $i= j$,} \\
        \zero &\text{otherwise.}
    \end{cases}
\end{equation*}

There is now a canonical decomposition with each $V\in \exponential{\Vect}{n}$ isomorphic to a direct sum of $e_{i}$ tensored with linear spaces $V_{i}\in \Vect$:
\begin{equation} \label{isomorphism:canonical-decomposition-2spaces}
    V\isomorphic \bigoplus_{i}V_{i}\otimes e_{i}
\end{equation}

Another important fact is that the categorification $T\colon \exponential{\Vect}{n}\to \exponential{\Vect}{n}$ of the notion of linear map can be intrinsically defined, but turns out to be equivalent to a square matrix of linear spaces $\objectmatrix{T}{j}{i}$. The action of $T$ on an object $V= \family{V}{i}\in \exponential{\Vect}{n}$ is the vector whose jth component is
\begin{equation*}
    T(V)(j)\isomorphic \bigoplus_{i}T_{j}^{i}\otimes V_{i}
\end{equation*}

Composition is given by the usual matrix composition formula by replacing sums with direct sums $\oplus$ and products by tensor products $\otimes$. Note that with these formulas, composition is only associative up to \emph{unique, canonical linear isomorphism}. This means that we end up with a \emph{bicategory} of $2$-spaces, morphisms $T$ and $2$-morphisms between them. The latter can be identified with matrices of linear maps $\tau_{j}^{i}\colon T_{j}^{i}\to S_{j}^{i}$ with composition being pointwise composition of matrices.

%Paragraph:
%- basis notion subtle => few self-equivalences.

\smallskip
Up to now, these are all features one could reasonably expect from a categorification of linear algebra. But we soon realize that these categorified linear spaces exhibit some unexpected behavior. The gist is that the notion of basis allowing the decomposition \eqref{isomorphism:canonical-decomposition-2spaces} is subtler than appears to be at a first glance. For example, it implies that every self-equivalence $\exponential{\Vect}{n}\to \exponential{\Vect}{n}$ amounts to a permutation of the canonical basis $\family{e}{i}$ of $\exponential{\Vect}{n}$. As first observed in \cite{barrett-mackaay:categorical-representations-categorical-groups2006}, since Lie groups $G$ tend to have few permutation representations, this entails that if the group object of a $2$-group $\mathcal{G}$ is such a Lie group, then there are few representations of $\mathcal{G}$ in $\exponential{\Vect}{n}$.

%Paragraph:
%- larger self-equivalence groups by bundle categories.
%- analogy with need of infinite-dimensional for locally compact groups.

\smallskip
This indicates that to obtain non-trivial representation $2$-categories, one must look around for monoidal categories with larger groupal groupoids of self-equi\-va\-lences. If we think about the construction of $\exponential{\Vect}{n}$ one immediate possibility is to replace $n$ with infinite sets $X$. Since the category $\exponential{\Vect}{X}$ is equivalent to the category of \emph{bundles of linear spaces over $X$}, the next logical step is to impose some structure on the base set $X$ (a topology, a measurable structure, etc.) and then consider the category of corresponding (continuous, measurable, etc.) bundles. At this point, it may even come to the reader's mind the analogous difficulties that one faces when representing locally compact groups: finite-dimensional spaces are not enough and we have to enlarge the representing category to include infinite-dimensional ones. And just as measure theory is the natural place to find such infinite-dimensional spaces large enough to yield non-trivial representations, we expect some sort of categorified measure theory to be a natural place to find infinite-dimensional analogues of $\exponential{\Vect}{n}$ that (may) yield non-trivial representations of $2$-groups. Note the emphasis here: what we are really after are infinite-dimensional categories, and more generally, \emph{categorified functional analysis}. Measure theory is playing the role of the middle man, providing us with a whole class of natural examples.

%Paragraph:
%- categories of measurable bundles and measurable fields of operators.
%- integral decomposition.
%- not clear what is the analogue of bounded linear map.

\smallskip
This is the point of departure for \cite{yetter:measurable-categories2005} and its concomitant suggestion of looking at categories whose objects are measurable fields of Hilbert spaces and with morphisms the measurable fields of bounded linear maps. These are classical analytical objects of study and their theory is explained in several textbooks of which we can cite \cite{mackey:theory-unitary-group-representations1976}, \cite{dixmier:von-neumann-algebras1981} and \cite{takesaki:theory-operator-algebrasI2002}. An important ingredient is that the canonical decomposition \eqref{isomorphism:canonical-decomposition-2spaces} is replaced by a \emph{direct integral decomposition} in the form,
\begin{equation} \label{isomorphism:isomorphism:integral-decomposition-measurable-spaces}
    \xi\isomorphic \int_{X}^{\oplus}H_{x}\differential\mu
\end{equation}
where $H_{x}$ is a measurable field on $X$ and $\mu$ a scalar measure.

\smallskip
A problem with these categories of measurable fields is that while in the previous discrete situation the right notion of categorified linear map is more or less immediate and can even be intrinsically characterized, the case is not so clear-cut for what is the right notion of categorified (bounded) linear map between them. Nevertheless, a solution was found in \cite{yetter:measurable-categories2005}, even if needing somewhat obscure, technical hypotheses on the base measure spaces.

%Paragraph:
%- take integral notation seriously.

\smallskip
The integral notation for the object on the right hand side of \eqref{isomorphism:isomorphism:integral-decomposition-measurable-spaces} is a cute but suggestive notation for a space of square-integrable sections of a bundle. In this paper we intend to take the notation \eqref{isomorphism:isomorphism:integral-decomposition-measurable-spaces} seriously as a \emph{categorified integral}, or in other words, to argue that the constructions of \cite{yetter:measurable-categories2005} (and such followups as \cite{baez-baratin-freidel-wise:infinite-dimensional-representations-2groups2008}) can be interpreted as \emph{categorifying measure theory}. To motivate and explain such an interpretation is a long voyage with a couple of false starts, the inevitable detours and a good deal of backtracking. In the words of W.~Stevens' intense poem, we must become an ignorant man again and perceive anew the inconceivable idea of a bright sun. I direct the reader to the end of this section for a brief description of the contents of the paper and end this introduction with a motley collection of remarks.

%Paragraph:
%- scope and ambitions.

\smallskip
First, the ambitions of the paper are far more modest than what the title may indicate. Although certain reflections of wider significance on the scope and conceptual meaning of categorification will be made in due course, I will not attempt to be systematic or construct any sort of coherent framework. Other categorifications of measure theory exhibiting other properties may very well exist, and although at times I may hint at possible generalizations or further avenues of research, the choices and comments are always guided and framed by the desire to interpret the specific set of constructions in \cite{yetter:measurable-categories2005} as a categorification, not the categorification, of measure theory. Also, since the main concern is with conceptual issues, no applications will be given. The ragtag reflections on other, incidental matters, are relegated to the footnotes, where their small print is not able to distract innocent readers, cause a fuss or otherwise do any harm.

%Paragraph:
%- general ideas, details only sketched.

\smallskip
Throughout, I have concentrated on general ideas, not on the specific details of which there are too many to be crammed in a paper with anything resembling a reasonable size. The basic results will be expounded in the stout handbook definition-theorem-corollary format, but the proofs are only sketched, if at all. Preposterous as it may sound, there is a real possibility that these so-called theorems are not theorems at all and I have botched things horribly somewhere along the way. Any lingering doubts in the reader's mind will only be alleviated in future papers (\cite{rodrigues:banach-2spaces} and \cite{rodrigues:categorified-measure-theory}), where I intend to provide all the missing details.

%Paragraph:
%- readability.

\smallskip
I have also endeavored to make the paper accessible to the widest audience possible, by surveying, even if sketchily, the necessary background material, so that the reader not acquainted with it may at least be able to pick up the basic ideas. There are many mathematical fields that will be touched upon and for what cannot be reviewed within a reasonable space, I have tried to provide references to the relevant literature. Naturally enough, some knowledge of the two central pillars, functional analysis and measure theory on one side and category theory on the other, is required. Despite first appearances, the amount of functional analysis and measure theory needed is actually very modest (it is all that this author knows anyway). For functional analysis, our basic reference is \cite{conway:course-functional-analysis1990} and for measure theory it is \cite{halmos:measure-theory1974}. As far as category theory is concerned, the situation is different and some sophisticated artillery will be deployed. For any undefined terms, the reader is referred to \cite{maclane:categories-working-mathematician1971} or the two volumes \cite{borceux:handbook-categorical-algebra11994} and \cite{borceux:handbook-categorical-algebra21994}. For monoidal categories, consult \cite[chapter VII]{maclane:categories-working-mathematician1971}, while $2$-categories and their weaker siblings, bicategories, are considered in \cite[chapter 7]{borceux:handbook-categorical-algebra11994}. More specific accounts can be found in \cite{kelly-street:elements-2categories1974} or in the more recent \cite{lack:2categories-companion} and references therein.

%Paragraph:
%- originality and borrowing.

\smallskip
It is traditional in mathematical papers to put the contents in perspective, relate them to earlier work and then demonstrate in a positive and assertive way that mathematics has just taken a quantum leap forward. Surveying the paper, the reader will have opportunity to observe that all the major ideas in it (and most of the minor ones) have been borrowed (adapted, stolen, whatever) from previous work, scattered throughout numerous sources. Once the major conceptual hurdles have been overcome and the correct framework identified, it is basically a matter of cranking the categorial machinery and let it operate its magic. Although at times I have been forced to baptize some object with a more descriptive name (e.g.~ \emph{Banach $2$-space} instead of the uninspiring and opaque \emph{cocomplete $\Banc$-enriched category}), there are no essentially new concepts. The train of ideas leading to the present paper was put in motion by an off-hand comment of L.~Crane made during a minicourse on unitary representations of locally compact groups (Lisbon, 1998-1999), to the effect that measurable fields of Hilbert spaces are a ``sort of sheaves.'' They were developed intermittently during a long period, in isolation and at a time of great personal difficulty. At some point, I have lost track of what was originally my own labor and what was borrowed\footnote{Dr.~Johnson, the mortal god of my imaginings, famously retorted to an author: ``Your manuscript is both good and original; but the parts that are good are not original, and the parts that are original are not good.''}. I have tried to provide as many references as possible (which is basically everything connected with this work that I have laid my eyes upon, and then some more), which accounts for the inordinate length of the bibliography, but have swerved away from settling matters of precedency. If the reader feels that I am being too cavalier about an issue of such a paramount importance, I am all too happy to acknowledge that Someone Else is to be blamed for this skein of thorns.

%Paragraph:
%- style.

\smallskip
A final note on style is in order. This paper being just an initial mapping stage, I have allowed myself more freedom in writing it. Some readers may find the occasional quirkiness distasteful, others may dislike the display of literary pedantry. I will not apologize for it, but if it is any consolation I remark that initial versions featured such precious tidbits as a goofy limerick; Biblical passages with accompanying theological commentary; several jokes in poor taste and an even larger number of failed puns. Whatever remained either slipped by the Censor's vigilant watch or is so inextricably tied to the matter that it could not be deleted without injuring the integrity of the whole. Such as the following quotation:
\smallskip
\begin{quote}
THE Necessity of this Digression, will easily excuse the Length; and I have chosen for it as proper a Place as I could readily find. If the judicious Reader can assign a fitter, I do here empower him to remove it into any other Corner he pleases. And so I return with great Alacrity to pursue a more important Concern.\\
--- Jonathan Swift, from the section \emph{A Digression in Praise of Digressions} of \emph{A Tale of a Tub}.
\end{quote}

\subsection*{Contents} In section \ref{section:measurable-bundles-hilbert-spaces}, the fundamental concept of measurable field of Hilbert spaces is discussed. Sections \ref{section:towards-categorified-measure-theory} and \ref{section:from-hilbert-banach-spaces} form the heart and soul of the paper. In them, I try out a preliminary description intended to yield a model of categorified measure theory. The discussion is somewhat long-winded, even repetitious, but I feel it necessary if the certitude and precision of the model is to be established. In section \ref{section:towards-categorified-measure-theory}, measurable fields of Hilbert spaces and the foundational paper \cite{yetter:measurable-categories2005} are revisited in order to extract the most important principles of categorified measure theory, and then in section \ref{section:from-hilbert-banach-spaces} the most important obstructions are faced, namely the need to include Banach spaces and the failure of the Radon-Nikodym theorem in infinite dimensions. Having raised the dust and \emph{then} swept it out, we proceed in the next three sections with whatever unguents of revelation we may possess and flesh out the basic principles expounded in sections \ref{section:towards-categorified-measure-theory} and \ref{section:from-hilbert-banach-spaces}. In section \ref{section:measure-algebras-integrals} we enter in review mode and anchor the required measure theory on measure algebras instead of measure spaces. This will not only entail some technical simplifications\footnote{The reader will find no mention of descriptive set theory and such concepts as analytic set or Luzhin space, that were needed in \cite{yetter:measurable-categories2005} to define the correct notion of categorified linear map.}, but it will open important new vistas on the categorification of measure theory. Subsection \ref{subsection:banach-algebra-linfty} constructs the Banach algebra $\Integrable_{\infty}(\Omega)$, the universal home for all measures on $\Omega$ and subsection \ref{subsection:bochner-integral} constructs the Bochner integral for any vector measure $\nu$ on a measure algebra and discusses some of the basic theorems of measure theory. Section \ref{section:banach-2spaces} starts by discussing very briefly Banach $2$-spaces, the categorified analogue of Banach spaces. The main bulk of the section is subsection \ref{subsection:presheaf-categories} that is devoted to reviewing basic concepts of enriched category theory like Kan extensions and coends. This is then applied in subsection \ref{subsection:free-banach-2spaces-sets} to a treatment of the Banach $2$-spaces of Banach bundles and in subsection \ref{subsection:categorified-measure-theory-discrete-case} to a complete treatment of categorified measure theory in the discrete $\sigma$-finite case, that is, the case of the atomic complete Boolean algebras $\subsets{X}$ for a countable set $X$. In the final section \ref{section:categorified-measures-integrals} we give some of the details of the fundamental constructions of categorified measure theory. The first subsection \ref{subsection:banach-sheaves} starts by defining the three Grothendieck topologies on a Boolean algebra and ends with the fundamental theorem \ref{theorem:banach-sheaves-finite-topology} that identifies sheaves for the finite topology with the internal Banach spaces of the topos $\Shv(\Stone(\Omega))$. The next three subsections are devoted to cosheaves. Subsection \ref{subsection:cosheaves-inverters} shows that $\Shv(\Omega)$ is the Banach $2$-space birepresenting cosheaves, subsection \ref{subsection:cosheafification-functor} constructs the cosheafification functor and subsection \ref{subsection:spectral-measure-cosheaf} constructs the spectral measure of a cosheaf. In this final subsection, the Radon-Nikodym property characterizing direct integrals is finally disclosed.

\subsection*{Acknowledgements} Several participants of the math newsgroup sci.math have kindly answered my questions and helped me sort out some of the technical details. A special thanks to G.~A.~Edgar (measurability in Banach spaces and the Radon-Nikodym property), A.~N.~Niel (measurability) and M.~Olshok (density of functors). Many many thanks are owed to R.~Picken for his generous encouragement, his help in tracking down some important papers and for his comments on initial versions of this paper that helped improving it enormously.

\section{Measurable bundles of Hilbert spaces}
\label{section:measurable-bundles-hilbert-spaces}

%Paragraph:
%- Introduction.

The paper \cite{yetter:measurable-categories2005} is founded upon the classical analytical tool of \emph{measurable fields of Hilbert spaces}, so we begin at the beginning and describe these objects, even if in a sketchy way. For more information, we refer the reader to \cite[chapter II]{dixmier:von-neumann-algebras1981} or \cite[chapter 2]{mackey:theory-unitary-group-representations1976}.

%Paragraph:
%- Preliminaries.

\smallskip
Let $\triple{X}{\Omega}{\mu}$ be a \emph{measure space} (or simply\footnote{As is customary, we denote structured sets by their carriers. Ironically, in this case this is actually a bad choice as will be seen in section \ref{section:measure-algebras-integrals}.} $X$), where $\pair{X}{\Omega}$ is a measurable space (a set $X$ together with a $\sigma$-algebra of subsets of $X$) and $\mu$ is a (positive, $\sigma$-additive) measure on $\Omega$. To dispel any possible confusions, the term $\sigma$-algebra as we employ it \emph{always} carries with it the implicit assumption that the whole space $X$ is measurable. We will also assume that $\mu$ is a \emph{complete probability}. The completeness condition can always be arranged for by adding to $\Omega$ the $\sigma$-ideal of subsets of $\mu$-null sets. Being a probability measure is not as restrictive as it sounds, because in the class of $\sigma$-finite spaces this can also be arranged for by the following device: let $\mathcal{E}$ be a \emph{partition of unity}, that is, a countable set of pairwise disjoint measurable sets of non-zero finite measure such that $X= \bigcup_{E_{n}\in \mathcal{E}}E_{n}$. Define the sequence of functions $\sequence{f}{n}$ by,
\begin{equation*}
    f_{n}\defequal \dfrac{\characteristic(E_{n})}{\mu(E_{n})2^{n}}
\end{equation*}
where $\characteristic(E_{n})$ is the characteristic function of $E_{n}$, and put $g_{n}\defequal \sum_{m= 1}^{n}f_{m}$. Then $g_{n}$ is an integrable positive function with $\sumnorm{,}$-norm equal to
\begin{equation*}
    \sum_{m= 1}^{n}\integral{X}{f_{m}}{\mu}= \sum_{m= 1}^{n} \dfrac{1}{2^{m}}
\end{equation*}

An application of Lebesgue's dominated convergence theorem (\cite[theorem D, page 110]{halmos:measure-theory1974}) tells us that the sequence $\sequence{g}{n}$ converges in $\Integrable_{1}(X)$ to an almost everywhere\footnote{From now on, we will use a.e.~as an abbreviation for \emph{almost everywhere}.} strictly positive function $g$ with $\sumnorm{g}= 1$. Taking the measure,
\begin{equation*}
    \nu(E)\defequal \integral{E}{g}{\mu}
\end{equation*}
we obtain a probability space such that the identity induces an isomorphism with $X$ in the category of measurable spaces and measurable maps.

\smallskip
With these preliminaries out of the way, we can define the notion of \emph{measurable field of Hilbert spaces}. Analogous to the situation in measure theory, measurability of fields is the intermediate concept allowing us to define their \emph{direct integral}.

%Paragraph:
%- bundles of Hilbert spaces.

\smallskip
View the set $X$ as a discrete category. Then a functor $X\to \Hilb$ is the same thing as a family $\family{H}{x}$ of Hilbert spaces parameterized by $x\in X$. These functors can also be described as \emph{bundles}. We use the word \emph{bundle} instead of the more classical terminology of \emph{field} to emphasize their geometric kinship.

\begin{definition} \label{definition:bundles-bundle-maps}
A \emph{bundle} $\xi$ over a set $X$ is a surjective map,
\begin{equation*}
    \pi\colon P\to X
\end{equation*}
such that the inverse image\footnote{We denote the inverse image of a set $F\subseteq Y$ along a map $f\colon X\to Y$ by $\invimage{f}(F)$. This is at odds with the more usual $\inverse{f}(F)$ but it has the convenient advantage of having one and the same notation for both the inverse image and the evaluation at $F$ of the inverse image function $\invimage{f}\colon \subsets{Y}\to \subsets{X}$.} $\invimage{\pi}(x)$ is a Hilbert space for every $x\in X$. The set $P$ is the \emph{total space} of the bundle and the Hilbert spaces $\xi_{x}= \invimage{\pi}(x)$ are the \emph{fibers} over $x\in X$. The map $\pi$ is the \emph{bundle projection map}.

A \emph{map $\tau\colon \xi\to \zeta$ between $X$-bundles $\xi$ and $\zeta$} is a map between the total spaces $\tau\colon P_{\xi}\to P_{\zeta}$ such that diagram \ref{diagram:fiber-preserving-condition} is commutative and for every $x\in X$, the restriction $\restriction{\tau}{\xi_{x}}\colon\xi_{x}\to \zeta_{x}$ to the fiber $\xi_{x}$ is a bounded linear map.
\begin{figure}[htbp]
    \begin{equation*}
    \xymatrix{
        P_{\xi} \ar[dr]_{\pi} \ar[rr]^{\tau} & & P_{\zeta} \ar[dl]^{\pi} \\
            & X &
    }
    \end{equation*}
\caption{Fiber-preserving condition for bundle maps.}
\label{diagram:fiber-preserving-condition}
\end{figure}
\end{definition}

There is an equivalence between the category $\Bundle(X)$ of bundles over $X$ as defined in \ref{definition:bundles-bundle-maps} and the category $\exponential{\Hilb}{X}$ of functors $X\to \Hilb$ given by sending each bundle $\xi$ to the functor $x\mapto \xi_{x}$. From now on, we will routinely identify these two categories via this equivalence.

\begin{definition} \label{definition:section-bundles}
A \emph{section} of an $X$-bundle $\xi$ is a map $s\colon X\to P$ such that $\pi s= \id_{X}$.
\end{definition}

%Paragraph:
%- Lack of measurability on the total space.

If $\xi$ is the constant bundle $x\mapto \field$ then the set of sections of $\xi$ is isomorphic to the set of functions $X\to \field$ and we can define the Hilbert space $\Integrable_{2}(\xi)$ of \emph{square-integrable sections} as the space $\Integrable_{2}(X)$ of square-integrable functions $X\to \field$. The \emph{direct integral} is the extension of this construction to a functor on bundles. But just as in ordinary integration theory we do not (and by the axiom of choice, cannot) expect to integrate an arbitrary function, we should not expect to be able to assign a ``Hilbert space of square-integrable sections'' to every bundle. The crux of the matter is that since the fibers $\xi_{x}$ vary from point to point, there is no notion of measurability available and thus no way to form the \emph{linear subspace of measurable sections}. The naive solution of taking the set of sections $s\colon X\to P$ such that the function $x\mapto \norm{s(x)}$ is measurable does not work, because this set is not closed for addition. The following example of \cite[chapter 2, page 90]{mackey:theory-unitary-group-representations1976} shows this.

\begin{example} \label{example:mackey}
Let $E$ be a non-measurable subset of $X$ and let $\xi$ be the constant bundle $x\mapto \field$. Define the section $s$ of $\xi$ by,
\begin{equation*}
    s(x)\defequal
    \begin{cases}
       1 &\text{if $x\in E$,} \\
       -1 &\text{otherwise.}
    \end{cases}
\end{equation*}
and let $t$ be a second section such that $x\mapto \norm{t(x)}$ is measurable. Then $t + st= (1 + s)t$ and from $\norm{s(x)t(x)}= \norm{t(x)}$ we conclude that $\norm{s(x)t(x)}$ is measurable. But now note that,
\begin{equation*}
    \norm{(t + st)(x)}=
    \begin{cases}
       2\norm{t(x)} &\text{if $x\in E$,} \\
       0 &\text{otherwise.}
    \end{cases}
\end{equation*}
so that $\norm{(t + st)(x)}$ is not measurable.
\end{example}

By suitably modifying the above example, the reader can come up with two natural transformations $\tau$, $\phi$ such that both functions $\norm{\tau_{x}}$ and $\norm{\phi_{x}}$ are measurable and yet the norm of their composition is not. It is clear that whatever measurable condition we impose on natural transformations $\tau$, the measurability of $\norm{\tau_{x}}$ is certainly a \emph{necessary} one. G.~Mackey's example shows that it is not \emph{sufficient}. The failure exposed by it is clear: there is no measurable structure on the total space of the bundle and hence no way to form the linear subspace of measurable sections.

%Paragraph:
%- measurable bundles.

\smallskip
What we need is a criterion that tells us that the fibers $\xi_{x}$ vary ``in a measurable way'' from point to point. The classical solution (see for example \cite[chapter 2]{mackey:theory-unitary-group-representations1976}) is given by the notion of \emph{pervasive sequence}. The idea behind the definition is that instead of fiddling with measurable structures directly, we describe the corresponding modules of measurable sections\footnote{The description of topological objects via modules over appropriate function algebras is at the heart of modern topology and geometry. The Serre-Swan theorem establishing an equivalence between the category of vector bundles over a topological space $X$ and the category of projective modules over the algebra $\Continuous(X)$ of continuous functions is certainly a major example, but the whole field of non-commutative geometry can be seen as a vast elaboration and extension of these ideas.}.

\begin{definition} \label{definition:measurable-bundle}
Let $X$ be a measure space. A \emph{measurable bundle over $X$} is a pair $\pair{\xi}{\sequence{s}{n}}$ where $\xi$ is a bundle over $X$ and $\sequence{s}{n}$ is a sequence of sections $X\to P$ that is \emph{pervasive}:
\begin{enumerate}
  \item \label{def-enum:measurable-bundle1}
  For a.e. $x\in X$, the set $\set{s_{n}(x)}$ is dense in the fiber $\xi_{x}$.

  \item \label{def-enum:measurable-bundle2}
  For every $n, m\in \Natural$ the function $x\mapto \inner{s_{n}(x)}{s_{m}(x)}$ is measurable.
\end{enumerate}
\end{definition}

%Paragraph:
%- consequence of pervasiveness => measurability of norm functions.

Given a measurable bundle $\pair{\xi}{\sequence{s}{n}}$, a section $s\colon X\to P$ is \emph{measurable} if for every $n$, the function $x\mapto \inner{s(x)}{s_{n}(x)}$ is measurable. Now, if $s$ and $t$ are two measurable sections then the function $x\mapto \inner{s(x)}{t(x)}$ is measurable\footnote{This is most easily shown with the aid of theorem \ref{theorem:pervasive-pointwise-orthonormal} below.}. In particular, the norm function $x\mapto \norm{s(x)}$ of any measurable section is measurable.

\smallskip
If $X$ is a measure space, taking the integral of the inner product,
\begin{equation*}
    \inner{s}{t}\defequal \integral{X}{\inner{s(x)}{t(x)}}{\mu}
\end{equation*}
we obtain a semi-inner product on the linear space of measurable sections of square-integrable norm. Quotienting by the linear subspace of sections a.e.~equal to $0$, we obtain the inner product space $\Integrable_{2}(\xi)$ which, by copying the classical proof, is easily shown to be complete. Since the pervasive sequence is a countable dense set, the space $\Integrable_{2}(\xi)$ is \emph{separable}. The space $\Integrable_{2}(\xi)$ is also a Banach module over the Banach algebra $\Integrable_{\infty}(X)$ where the action is scalar multiplication
\begin{equation*}
    \pair{f}{s}\mapto \parens{fs\colon x\mapto f(x)s(x)}
\end{equation*}

The association $\xi\mapto \Integrable_{2}(\xi)$ defines the object part of a functor on measurable bundles. To define it on morphisms we need to impose one more condition on a bundle morphism.

%Paragraph:
%- measurability condition on morphisms.

\begin{definition} \label{definition:measurable-bundle-morphism}
A morphism $\tau\colon \xi\to \zeta$ is \emph{measurable} if it preserves the measurability of sections: for every measurable section $s$ of $\xi$, the section $\tau(s)$ is measurable.
\end{definition}

To illustrate the fundamental character of the measurability condition, let $\tau$ be a measurable bundle morphism $\xi\to \zeta$ and $\sequence{s}{n}$ the pervasive sequence of $\xi$. We start by cutting off each $s_{n}$ away from the unit ball. Define the sets
\begin{equation*}
    E_{n}\defequal \set{x\in X\colon \norm{s_{n}(x)}\leq 1}
\end{equation*}

By measurability of $s_{n}$, each $E_{n}$ is measurable. Define,
\begin{equation*}
    t_{n}\defequal \characteristic(E_{n})s_{n}
\end{equation*}

The sections $t_{n}$ are all measurable and clearly, the sequence $\sequenced{t_{n}(x)}$ is dense in the unit ball of the fiber $\xi_{x}$. Therefore:
\begin{equation} \label{equation:norm-supremum-norm-functions}
    \norm{\tau_{x}}= \sup\set{\norm{\tau_{x}(t_{n}(x))}\colon n\in \Natural}
\end{equation}

By measurability of $\tau$, each $\tau(t_{n})$ is measurable and in particular, the norm function $x\mapto \norm{\tau_{x}(t_{n}(x))}$ is measurable. Since the supremum on the right-hand side of \eqref{equation:norm-supremum-norm-functions} is taken over a \emph{countable} set, the next theorem follows.

\begin{theorem} \label{theorem:measurability-norm-morphism}
If $\tau\colon \xi\to \zeta$ is a measurable bundle morphism, the function $x\mapto \norm{\tau_{x}}$ is measurable.
\end{theorem}

Denote by $\covariant{\tau}$ the linear map $\Integrable_{2}(\xi)\to \Integrable_{2}(\zeta)$ given by $s\mapto \tau(s)$. On the hypothesis that $x\mapto \norm{\tau_{x}}$ is essentially bounded, we have the equality
\begin{equation} \label{equation:norm-equality-measurable-morphisms}
    \norm{\covariant{\tau}}= \supnorm{\tau_{x}}
\end{equation}

For the proof, we refer the reader to \cite[chapter II, section 2, proposition 2]{dixmier:von-neumann-algebras1981}. These constructions yield the morphism part of the $\Integrable_{2}$-functor. Its domain is the category\footnote{The reader should note that we have dropped the subscript $2$ from the category $\IBundle(X)$ of measurable bundles. For now, this has the advantage of avoiding any possible confusion with the Hilbert spaces $\Integrable_{2}(X)$. Later on, we will be deliberate and bend our notation so as to maximize the chances of confusion with the ordinary decategorified situation. The absence of the subscript will then be explained away by the fact that in the world of the categories we will consider such concepts as bounded, summable, etc. are either vacuous or simply make no sense.} $\IBundle(X)$ of measurable bundles and a.e.~equal equivalence classes of essentially bounded, measurable morphisms.

\section{Towards categorified measure theory}
\label{section:towards-categorified-measure-theory}

%Paragraph:
%- introduction.
%- classical integral notation and terminology.

In section \ref{section:measurable-bundles-hilbert-spaces} we introduced the category of measurable bundles $\IBundle(X)$ and the functor assigning to each measurable bundle $\xi$ its Hilbert space of square-integrable sections $\Integrable_{2}(\xi)$. This space is also called the \emph{direct integral} of $\xi$ and is denoted classically by
\begin{equation} \label{notation:direct-integral}
    \int_{X}^{\oplus}\xi\differential\mu
\end{equation}

If we drop the $\oplus$ from \eqref{notation:direct-integral}, we end up with ordinary integral notation, and it is our intention to take this seriously. This entails interpreting the constructions of section \ref{section:measurable-bundles-hilbert-spaces} as a \emph{categorified integral}. If we surmise the evidence from \cite{yetter:measurable-categories2005}, we can gather that the bundles are over a measurable space and that the $\Integrable_{2}$-space of sections is denoted by an integral sign. The paper \cite{baez-baratin-freidel-wise:infinite-dimensional-representations-2groups2008} presents in section 3.3 a table of categorified analogues for the basic ingredients of linear algebra. The last two rows of the table give measurable bundles and the direct integral as the categorified analogues of respectively, measurable functions and the ordinary integral, but the obvious analogy with measure theory is not pursued any further.

%Paragraph:
%- the formal analogy.

\smallskip
To quote almost verbatim from the introduction to \cite{baez-dolan:categorification1998}, at its heart, categorification is based on an analogy between sets and categories as presented in table \ref{table:sets-categorified-analogues}.
\begin{table}[htbp]
  \begin{tabular}{|c|c|}
    \hline
    Sets $X$ & Categories $\mathcal{X}$ \\
    \hline
    \hline
    Elements & Objects \\
    $x\in X$ & $x\in \mathcal{X}$ \\
    \hline
    Equations & Isomorphisms \\
    $x= y$ & $x\isomorphic y$ \\
    \hline
    functions & Functors \\
    $f\colon X\to Y$ & $F\colon \mathcal{X}\to \mathcal{Y}$ \\
    \hline
    Set $\homset{\Set}{X}{Y}$ of & Category $\homset{\Cat}{\mathcal{X}}{\mathcal{Y}}$ of \\
    functions $X\to Y$ & functors $\mathcal{X}\to \mathcal{Y}$ \\
    \hline
  \end{tabular}
  \smallskip
  \caption{Analogy between sets and categories}
  \label{table:sets-categorified-analogues}
\end{table}

Building on table \ref{table:sets-categorified-analogues}, we adduce table \ref{table:categorified-analogues-ordinary-integrals} that presents the categorified analogues of some basic measure-theoretical concepts.
\begin{table}[htbp]
  \begin{tabular}{|c|c|}
    \hline
    Ordinary integrals & Direct integrals \\
    \hline
    \hline
    Measurable functions &  Measurable bundles \\
    $f\colon X\to \field$ & $\xi\colon X\to \Hilb$ \\
    \hline
    Integral & Direct integral \\
    $\integral{X}{f}{\mu}\in \field$ & $\integral{X}{\xi}{\mu}\in \Hilb$ \\
    \hline
    Integral map & Integral functor \\
    $\int_{X}\differential\mu\colon \Integrable_{2}(X)\to \field$ & $\int_{X}\differential\mu\colon \IBundle(X)\to \Hilb$ \\
    \hline
  \end{tabular}
  \smallskip
  \caption{Categorified analogues of ordinary integrals}
  \label{table:categorified-analogues-ordinary-integrals}
\end{table}

%Paragraph:
%- the open questions.

The entries of table \ref{table:categorified-analogues-ordinary-integrals} are imposed on us by the constructions of section \ref{section:measurable-bundles-hilbert-spaces}, but they leave a lot of questions unanswered. For one, while the codomain $\field$ is categorified to $\Hilb$, the domain (a measure space) remains untouched. Similarly, looking at the second row of \ref{table:categorified-analogues-ordinary-integrals} we see that while the integrand of a direct integral is a categorified function, a bundle, the measure is still the same uncategorified scalar measure $\mu$. Some of these questions will be answered once we start clearing up the fog, but for now just take table \ref{table:categorified-analogues-ordinary-integrals} as a starting point (with emphasis on the \emph{starting}).

%Paragraph:
%- summary of rest of section.

\smallskip
In the rest of the section, the two main issues to be tackled are the \emph{definition of measurability} and the notion of \emph{categorified integral}. We will go through the constructions of section \ref{section:measurable-bundles-hilbert-spaces} and ask how and in what sense they can be interpreted as the categorification of the familiar notions of measure theory. In the spirit of an inductive investigation, the findings will be crystalized in ten principles that will then serve as a guide to the construction of a proper, fully-fledged theory of categorified measures and integrals. The quest will be long, but we hope that in the end it will be clear that a comprehensive, conceptual picture of categorified measure theory has emerged. Instead of a mishmash of loose analogies, more or less felicitous guesses, juxtapositions and flights of fancy, we will have at our disposal a conceptual framework in which it is possible to systematically answer the many questions that arise. I am also aware that at every step of the argument there are some genuine problems to solve and gaps to fill, so I am compelled to ask the benign reader to overlook, not in the sense of ignoring but of seeing past, whatever strikes him as inadequate or simply wrong.

\subsection{\texorpdfstring{$\Hilb$}{\textbf{Hilb}} as a categorified ring}
\label{subsection:hilb-categorified-ring}

Per the third row of table \ref{table:categorified-analogues-ordinary-integrals} and the constructions of section \ref{section:measurable-bundles-hilbert-spaces}, the integral is a functor
\begin{equation*}
    \int_{X}\differential\mu\colon \IBundle(X)\to \Hilb
\end{equation*}

Let us leave aside for the moment the categorial structure of $\IBundle(X)$ and in what sense it can be considered a categorified Hilbert space, and instead concentrate on the categorification of the codomain $\field$. Ideally, an integral is not just a functor on some category of bundles, but a functor satisfying some sort of \emph{linearity and continuity properties}, the same as ordinary integrals. In order to formulate these linearity properties, the first task is to identify what are the analogues of the basic operations needed to do integration, sum and multiplication, that is, we need the notion of \emph{categorified ring}. We advance with table \ref{table:categorified-analogues-ring-operations} for the specific case of the category $\Hilb$ of Hilbert spaces.
\begin{table}[htbp]
  \begin{tabular}{|c|c|}
    \hline
    Commutative ring & Categorified commutative ring \\
    \hline
    \hline
    Field $\field$ & Category $\Hilb$ \\
    Sum $+$ & Direct sum $\oplus$ \\
    Multiplication $\times$ & Tensor product $\otimes$ \\
    Additive zero $0$ & Zero object $\zero$ \\
    Multiplicative unit $1$ & Tensor unit $\field$ \\
    \hline
  \end{tabular}
  \smallskip
  \caption{Categorified analogues of the ring operations}
  \label{table:categorified-analogues-ring-operations}
\end{table}

%Paragraph:
%- direct sums of Hilbert spaces.

\smallskip
In the course of this work, we will have opportunity to comment on this and other similar tables; for the moment, we just refer the reader to the first section of \cite{baez:hda2-2-hilbert-spaces1997} (from which the table was borrowed) for a proper motivation and instead move on to recall the construction of \emph{direct sums} in $\Hilb$. If $\family{H}{x}$ is a family of Hilbert spaces, take the linear direct sum $\bigoplus_{x}H_{x}$ with the pointwise-sum inner product
\begin{equation} \label{equation:coproduct-inner-product}
    \inner{\family{h}{x}}{\family{k}{x}}\defequal \sum_{x\in X}\inner{h_{x}}{k_{x}}
\end{equation}

If the indexing set is infinite, the space $\bigoplus_{x}H_{x}$ is not complete. The direct sum $\sum_{x\in X}H_{x}$ is the completion of $\bigoplus_{x}H_{x}$ under the inner product \eqref{equation:coproduct-inner-product}.

\smallskip
Denote by $i_{x}$ the inclusion $H_{x}\to \sum_{x\in X}H_{x}$ assigning to $h\in H_{x}$ the element $i_{x}(h)\in \sum_{x\in X}H_{x}$ given by
\begin{equation*}
    {i_{x}(h)}_{y}\defequal
    \begin{cases}
        h &\text{if $x= y$,} \\
        0 &\text{otherwise.}
    \end{cases}
\end{equation*}

The universal property of the space $\sum_{x\in X}H_{x}$ is described and proved in the next theorem in the case where $X$ is finite.

\begin{theorem} \label{theorem:existence-coproducts-hilb}
Let $X$ be a finite set and $\family{H}{x}$ an $X$-parameterized family of Hilbert spaces or a bundle over $X$. Then for every cone $\family{T}{x}$ of bounded linear maps\footnote{A \emph{cone from a functor $F$ to an object $a$} is simply a natural transformation $F\to \diagonal(a)$ with $\diagonal(a)$ the constant functor with value $a$. In the case at hand, since the domain is a discrete category, a cone is just a family of bounded linear maps all with the same codomain. Cones to a functor are defined similarly; for reasons of euphony we refrain from using such terms as cocone.} $H_{x}\to H$ there is a unique bounded linear map $T\colon \sum_{x\in X}H_{x}\to H$ such that for every $x\in X$ the triangle \ref{diagram:universal-property-direct-sum} is commutative.
\begin{figure}[htbp]
    \begin{equation*}
    \xymatrix{
        \sum_{x\in X}H_{x} \ar@{-->}[r]^-{T} & H \\
        H_{x} \ar[u]^{i_{x}} \ar[ur]_{T_{x}} &
    }
    \end{equation*}
    \caption{Universal property of $\sum_{x\in X}H_{x}$.}
    \label{diagram:universal-property-direct-sum}
\end{figure}

The map $\family{T}{x}\mapto T$ is a contractive linear bijection,
\begin{equation} \label{isomorphism:existence-coproducts-hilb}
    \sum_{x\in X}\homset{\Hilb}{H_{x}}{H}\equivalent \homset{\Hilb}{\sum_{x\in X}H_{x}}{H}
\end{equation}
where the Banach space on the left-hand side is endowed with the Hilbertian $2$-norm.
\end{theorem}
\begin{proof}
The crucial fact is that given a cone $\family{T}{x}$ there is a unique linear map $T\colon \sum_{x\in X}H_{x}\to H$ closing the triangle \ref{diagram:universal-property-direct-sum} defined by:
\begin{equation*}
    b= \family{b}{x}\mapto \sum_{x\in X}T_{x}b_{x}
\end{equation*}

The boundedness of $T$ is now a simple application of the CBS inequality (\cite[chapter I, section 1]{conway:course-functional-analysis1990}):
\begin{equation*}
    \norm{T(b)}= \norm{\sum_{x\in X}T_{x}b_{x}}\leq \sum_{x\in X}\norm{T_{x}}\norm{b_{x}}\leq \hilbertnorm{\family{T}{x}}\hilbertnorm{\family{b}{x}}
\end{equation*}

This implies that $T$ is bounded with $\norm{T}\leq \hilbertnorm{\family{T}{x}}$. It is easy to see that the map is a bijection: just note that if $T$ is a bounded linear map $\sum_{x\in X}H_{x}\to H$ then we have a cone $\parens{Ti_{x}}$ and, since the $i_{x}$ are isometries,
\begin{equation*}
    \sum_{x\in X}\norm{Ti_{x}}^{2}\leq \sum_{x\in X}\norm{T}^{2}= n\norm{T}^{2}
\end{equation*}
with $n$ the cardinality of $X$.
\end{proof}

%Paragraph:
%- first hints of bad behavior on \Hilb.

The reader should note two things in theorem \ref{theorem:existence-coproducts-hilb}. First, that the representability isomorphism \eqref{isomorphism:existence-coproducts-hilb} is \emph{not} an isometry and there is no \emph{natural} norm on the space of cones $\sum_{x\in X}\homset{\Hilb}{H_{x}}{H}$ that makes it so. In this context, the choice of the Hilbertian norm for the space of cones ends up being a tad arbitrary\footnote{We could have endowed $\bigoplus_{x\in X}H_{x}$ with any $p$-norm, with $p\in [1, \infty]$. Since $X$ is finite, these norms are all equivalent and almost all the basic theorems involving the Hilbertian norm would still be true by replacing the Hilbertian norm by any other $p$-norm.}. Second, and more importantly for what follows, is the inevitable appearance of Banach spaces via spaces of bounded linear maps.

\smallskip
Recalling the definition of coproducts (\cite[chapter III, section 3]{maclane:categories-working-mathematician1971}), we have the following immediate corollary.

\begin{corollary} \label{corollary:finite-coproducts-hilb}
The category $\Hilb$ has all finite coproducts.
\end{corollary}

The restriction to finite $X$ is necessary. The proof of theorem \ref{theorem:existence-coproducts-hilb} breaks down if $X$ is infinite and later in section \ref{section:from-hilbert-banach-spaces} we will give the simple argument that proves that $\Hilb$ does \emph{not} have infinite (co)products. This is a very serious drawback of $\Hilb$ as will be seen in the continuation. On the other hand, since $\Hilb$ is additive it follows that it has all products and they are in fact, biproducts (see \cite[chapter VIII, section 2]{maclane:categories-working-mathematician1971} for the definition and basic properties of additive categories).

\smallskip
If coproducts give categorified sums, tensor products give the categorified multiplication. Recall that the tensor product of Hilbert spaces is the representing object for the bifunctor of bounded bilinear maps. This universal property is not as slick as that of theorem \ref{theorem:existence-coproducts-hilb}, because although $\Hilb$ is a symmetric monoidal category, it is \emph{not} closed. This is the second serious drawback of $\Hilb$. We note however that the tensor bifunctor, being linear in each variable, automatically preserves biproducts up to linear homeomorphism which are the isomorphisms in $\Hilb$ by the open mapping theorem (\cite[chapter III, section 12]{conway:course-functional-analysis1990}). The next theorem improves on this a little bit.

\begin{theorem} \label{theorem:commutativity-direct-sums-tensors}
Let $X$ be a finite set, $H_{x}$ and $K$ Hilbert spaces. The canonical map
\begin{equation}\label{isomorphism:commutativity-direct-sums-tensors}
    \sum_{x\in X}H_{x}\otimes K\to \parens{\sum_{x\in X}H_{x}}\otimes K
\end{equation}
is an \emph{isometric isomorphism}.
\end{theorem}

%Paragraph:
%- universal properties vs. coherence laws.

One important detail that table \ref{table:categorified-analogues-ring-operations} manages to hide is that the algebraic laws of rings get promoted in the categorified setting to specified isomorphisms that satisfy equations of their own, the so-called \emph{coherence laws}. These coherence laws are the more mysterious and fascinating part of any categorification enterprise. There is a vast literature on the subject, from which we can cite \cite{laplaza:coherence-natural-distributivity1972} and \cite{kelly-laplaza:coherence-compact-closed-categories1980} that introduced and formalized the notion of \emph{symmetric ring category}. For the most part, it is gobbledygook completely irrelevant to our purposes. In a wonderful flash of insight, \cite{yetter:categorical-linear-algebra} observed that all the operations in table \ref{table:categorified-analogues-ring-operations} are defined by \emph{universal properties} and we can avoid entirely the discussion of their coherence laws.

\smallskip
The notion of universal property is, in its various guises, the central concept of category theory and D.~Yetter's observation strikes us as so important that we elect it as our first guiding principle\footnote{Users of the Python programming language will readily recognize the source for the formulation of principle \ref{principle:universal-properties-coherence-laws}.}.

\begin{principle} \label{principle:universal-properties-coherence-laws}
Universal properties are one honking great idea -- let's do more of those!
\end{principle}

Principle \ref{principle:universal-properties-coherence-laws} is the overarching organizational principle of this paper and its consequences will be followed through doggedly. For example, we mentioned above that we can avoid a discussion of coherence laws because all the entries in table \ref{table:categorified-analogues-ring-operations} are defined by universal properties. This was something of a rhetorical exaggeration made in order to drive a point home. The actual truth is that tensor products are problematic, precisely because $\Hilb$ is not closed. If we restricted ourselves to finite dimensional Hilbert spaces, then we would have a symmetric monoidal closed category, meaning that there is a natural isometric isomorphism
\begin{equation} \label{isomorphism:hilb-monoidal-closed}
    \homset{\Hilb}{H\otimes K}{L}\isomorphic \homset{\Hilb}{H}{\homset{\Hilb}{K}{L}}
\end{equation}

Closedness is an extremely useful property. It not only allows a decent theory of enriched categories but it turns tensor products into \emph{tensors} (\cite[chapter 3, section 7]{kelly:enriched-category-theory2005}). Now, it just happens that tensors are a special example of \emph{weighted colimits} and thus it is obvious what it means for a functor to preserve them -- and again no need to discuss any coherence laws. As things stand we will need to have to recourse to the notion of module functors (\cite[section 1]{yetter:categorical-linear-algebra}). The upshot of these contortions is simple: $\Hilb$ is the ``wrong category''\footnote{One of the more striking teachings of modern algebraic geometry a la A.~Grothendieck is that it is far more preferable to have a well-behaved category of (possibly) badly behaved objects, than well-behaved objects but constituting a badly behaved category.}. There are two cogent reasons why this remark is not erected into a principle; one, that we are not yet quite finished with badmouthing $\Hilb$; the other, that at this point it is not clear what categories would even count as ``right''.

\subsection{Measurability of bundles}
\label{subsection:measurability-bundles}

Let $X$ be a measurable space. Recall that a function $f\colon X\to \field$ is \emph{measurable} if the inverse image $\invimage{f}(E)$ of every (Borel) measurable set $E\subseteq \field$ is measurable. Measurability is a \emph{property}: a function either is or is not measurable. Measurability for bundles as in definition \ref{definition:measurable-bundle} is an \emph{extra structure} on the bundle, the \emph{choice} of a pervasive sequence. This, sure enough, is a common theme in categorification. Potentially, there are many ways to make a category into, say, a symmetric ring category and each choice is an extra structure piled on the categorial structure. But as discussed in the preceding subsection, the categorified algebraic operations we are interested in are given by universal properties (even colimits), and a given category either has them or not. We will argue that definition \ref{definition:measurable-bundle} of measurable bundle is similarly misguided.

%Paragraph:
%- all pervasive sequences equally good.

\smallskip
First, there is the lurking suspicion that the choice of pervasive sequence is fairly arbitrary, almost irrelevant. This is best seen by starting at the end: if $\sequence{h}{n}$ is a countable dense sequence of $\Integrable_{2}(\xi)$, then we obtain a pervasive sequence for $\xi$ that for all purposes is equivalent to the original pervasive sequence, that is, the respective spaces of square-integrable sections are canonically isomorphic\footnote{If this last point is not clear, consult theorem \ref{theorem:measurability-countable-coproduct} below and the comments surrounding it.}. What seems to be happening here is that we do have to make a choice, but then all choices are equally good (or equally bad, depending on one's personal philosophy). Whenever there is a large number of equally good choices to make, it is an hint that there is an action of some symmetry group. The pervasive sequence solution is to factor out this symmetry group by making a choice of an origin, but modern geometry shows us that there are other ways to tackle this problem.

%Paragraph:
%- measurability in constant bundles = Borel measurability.

\smallskip
Second, although the choice of a pervasive sequence allows us to pick a linear space of sections with the required measurability properties, in the case of \emph{characteristic bundles} there \emph{already is} a notion of measurability available. Let $E\subseteq X$ be a measurable subset and $H$ a separable Hilbert space. The characteristic bundle $\characteristic(E)H$ is defined as:
\begin{equation} \label{bundle:characteristic-hilbert}
    \characteristic(E)H\defequal
    \begin{cases}
        H &\text{if $x\in E$,} \\
        \zero &\text{otherwise.}
    \end{cases}
\end{equation}

\smallskip
The measurable structure of $\characteristic(E)H$ is given by picking a countable dense sequence $\sequence{h}{n}$ of $H$ and taking the sequence of constant sections
\begin{equation*}
    s_{n}(x)\defequal
    \begin{cases}
        h_{n} &\text{if $x\in E$,} \\
        0 &\text{otherwise.}
    \end{cases}
\end{equation*}

Note that the particular choice of the dense sequence is irrelevant and any other such would lead to an equivalent measurable structure. On the other hand, since sections $s$ of $\characteristic(E)H$ can be identified with functions $E\to H$ we say that $s$ is \emph{Borel measurable} if the associated function is measurable when $H$ is given the Borel measurable structure. Recall that if $X$ is a topological space, the \emph{Borel structure of $X$} is the $\sigma$-algebra generated by the open sets. We now have:

\begin{theorem} \label{theorem:measurable-pervasive-measurable}
A section $s$ of $\characteristic(E)H$ is Borel measurable iff for every $n$, the map $x\mapto \inner{s(x)}{s_{n}(x)}$ is measurable.
\end{theorem}

A direct proof of the above theorem is not difficult, but since it is a special case of the Pettis measurability theorem to which we will have to return to later in subsection \ref{subsection:failure-radon-nikodym}, we leave it to the interested reader. Note also that a bundle morphism $\tau\colon \characteristic(E)H\to \characteristic(F)K$ can be identified with a map
\begin{equation} \label{map:operators-functions-into-operators}
    \tau\colon E\cap F\to \homset{\Hilb}{H}{K}
\end{equation}

%Paragraph:
%- measurability issues in morphism spaces.

The codomain is a Banach space and has a Borel structure. The question of the equivalence of the measurability of $\tau$ with the Borel measurability of \eqref{map:operators-functions-into-operators} is more delicate\footnote{The difficulty lies partly in the fact that the space $\homset{\Hilb}{H}{K}$ is not separable when $H$ and $K$ are infinite dimensional and thus, Borel measurability is not adequate. To make matters more complicated, a deep result of A.~Szankowski (\cite{szankowski:algebra-bounded-operators-not-approximation-property1981}) shows that this space does not have the approximation property.}, but since the Borel measurability of the latter already implies the measurability of $x\mapto \norm{\tau_{x}}$ we are on the right track.

%Paragraph: classification theorem.
%- pointwise orthonormal sequence.

\smallskip
But the story does not end here. Let $\xi$ be a measurable bundle. Then from a pervasive sequence $\sequence{t}{m}$ we can construct a \emph{pointwise orthonormal sequence} $\sequence{s}{n}$. Since this construction is important, being the foundation of later results, we recall it in the next theorem. For the proof, we refer the reader to \cite[chapter II, section 4]{dixmier:von-neumann-algebras1981}.

\begin{theorem} \label{theorem:pervasive-pointwise-orthonormal}
Let $\xi$ be a bundle of Hilbert spaces and $\sequence{t}{m}$ a pervasive sequence. Then there exists a sequence $\sequence{s}{n}$ of sections of $\xi$ such that:
\begin{enumerate}
  \item \label{theor-enum:pervasive-pointwise-orthonormal1}
  For every $n, m$, the function $x\mapto \inner{s_{n}(x)}{s_{m}(x)}$ is measurable.

  \item \label{theor-enum:pervasive-pointwise-orthonormal2}
  For a.e.~$x\in X$, the set $\set{s_{n}(x)}$ is an orthonormal base of $\xi_{x}$ if $\dimension(\xi_{x})= \omega$. If $\dimension(\xi_{x})= m$ then the set $\set{s_{n}(x)\colon n\leq m}$ is an orthonormal base of $\xi_{x}$ and $s_{n}(x)= 0$ for $n> m$.

  \item \label{theor-enum:pervasive-pointwise-orthonormal3}
  A section $s$ of $\xi$ is measurable iff the function $x\mapto \inner{s(x)}{s_{n}(x)}$ is measurable for every $n\in \Natural$.
\end{enumerate}
\end{theorem}

The converse passage from pointwise orthonormal sequences to pervasive ones is even easier. Given a sequence $\sequence{s}{n}$ satisfying \eqref{theor-enum:pervasive-pointwise-orthonormal1} and \eqref{theor-enum:pervasive-pointwise-orthonormal2} of theorem \ref{theorem:pervasive-pointwise-orthonormal}, then taking rational linear combinations of $s_{n}$ and reindexing, we obtain a pervasive sequence of $\xi$ and the respective $\Integrable_{2}$-spaces of sections are isometrically isomorphic. In other words, as far as bundle measurability is concerned the two concepts are equivalent.

%Paragraph:
%- additivity of measurable bundles.
%- countable direct sums of measurable bundles.

\smallskip
Before putting theorem \ref{theorem:pervasive-pointwise-orthonormal} to work, we must make a small detour and discuss one piece of the categorial structure of $\IBundle(X)$. What we need is the simple fact that direct sums in $\Hilb$ lift to $\IBundle(X)$. If $\pair{\xi}{\sequence{s}{n}}$ and $\pair{\zeta}{\sequence{t}{m}}$ are two measurable bundles, their \emph{pointwise direct sum} is the bundle $x\mapto \xi_{x}\oplus \zeta_{x}$ with pervasive sequence (after suitable reindexing) given by all sums of the form
\begin{equation*}
    s_{n} + t_{m}
\end{equation*}

It is not difficult to see that $\xi\oplus \zeta$ is a coproduct in $\IBundle(X)$. Since the zero bundle $x\mapto \zero$ is a zero in $\IBundle(X)$, we have the following theorem.

\begin{theorem} \label{theorem:measurable-bundles-additive}
The category $\IBundle(X)$ is additive.
\end{theorem}

The construction of pointwise direct sums can be extended to a countable\footnote{But no larger. Indexing sets of larger cardinality lead to non-separable fibers.} sequence $\pair{\xi_{n}}{\sequence{s}{m, n}}$ of measurable bundles. On the fibers, the direct sum is given by
\begin{equation*}
    \sum_{n}\xi_{n}\colon x\mapto \sum_{n}\parens{\xi_{n}}_{x}
\end{equation*}

The measurable structure is the sequence (after reindexing) given by sums of the form,
\begin{equation*}
    \sum_{n\in I}s_{n, m_{n}}
\end{equation*}
where $I\subseteq \Natural$ is a finite subset and $s_{n, m_{n}}$ is a section of $\xi_{n}$ in the respective pervasive sequence. For much the same reasons as in $\Hilb$ (see section \ref{section:from-hilbert-banach-spaces} below), the countable pointwise direct sum $\sum_{n}\xi_{n}$ is \emph{not} a coproduct in $\IBundle(X)$.

\smallskip
Back to measurability issues. If $s$ is a measurable section of a measurable bundle $\xi$, then the \emph{support} of $s$ is the set
\begin{equation} \label{equation:support-measurable-section}
    \supp(s)\defequal \set{x\in X\colon s(x)\neq 0}
\end{equation}

Since $s$ is measurable, $\supp(s)$ is measurable. Thus, the section $s$ generates the characteristic line bundle $\characteristic(\supp(s))\field$ concentrated on $\supp(s)$ and from pointwise orthonormality it follows that:

\begin{theorem} \label{theorem:measurability-countable-coproduct}
Let $\xi$ be a measurable bundle and $\sequence{s}{n}$ a pointwise orthonormal sequence in $\xi$. Put $E_{n}\defequal \supp(s_{n})$. Then there is an isometric bundle isomorphism
\begin{equation}\label{isomorphism:measurability-countable-coproduct}
    \xi\isomorphic \sum_{n}\characteristic(E_{n})\field
\end{equation}
\end{theorem}

Theorem \ref{theorem:measurability-countable-coproduct} describes every measurable bundle as a gluing of characteristic line bundles by a countable collection of measurable bundle isomorphisms $\phi_{n, m}\colon \characteristic(E_{n}\cap F_{m})\to \characteristic(E_{n}\cap F_{m})$, or by a \emph{measurable non-zero function $E_{n}\cap F_{m}\to \field$}. Normalizing, that is, dividing by $\abs{\phi_{n, m}}$, we obtain a function with values in the unit sphere of $\field$: the unit circle $\sphere^{1}$ in the complex scalar case, the zeroth sphere $\sphere^{0}= \set{-1, 1}$ in the real scalar case \footnote{This cocycle description is familiar to anyone conversant with vector bundles. What may not be so familiar is that these gluing functions $\phi_{n, m}$ are precisely the \emph{extreme points} of the unit ball of certain Banach spaces. At this point this may sound like a useless piece of information, but the reader should store it in the back of his mind as extreme point phenomena will resurface later.}.

\smallskip
In a nutshell, we can do away with pervasive sequences by working with countable direct sums of characteristic bundles and countable direct sums of measurable morphisms. In more categorial terms, we expect that something very close to principle \ref{principle:category-measurable-bundles} is true. It can be seen as the categorified analogue of the fact that the space of measurable functions is, in the appropriate topology, the sequentially closed linear span of the characteristic functions.

\begin{principle} \label{principle:category-measurable-bundles}
The category of measurable bundles is the closure of the subcategory of characteristic bundles and measurable morphisms under countable colimits.
\end{principle}

The process of completing categories by classes of (co)limits is covered in \cite{kelly:enriched-category-theory2005} (starting in chapter 5, section 7) and then in a string of papers of which we can mention \cite{albert-kelly:closure-class-colimits1988}, \cite{kelly-schmitt:notes-enriched-categories-colimits-class2005} and \cite{kelly-lack:monadicity-categories-chosen-colimits2000}. In general, such completions involve transfinite constructions, on which the definitive paper is the fairly indigestible \cite{kelly:unified-treatment-transfinite-constructions1980}. Later, we will see that it is both technically simpler and conceptually sounder to simply add \emph{all} (small) colimits.

%Paragraph:
%- shift from measurability to existence of infinitary colimits.

\smallskip
The main point of principle \ref{principle:category-measurable-bundles} is that underlying it, there is at work a deep analogy with geometry where the spaces (schemes, bundles, manifolds, etc.) are built by gluing local models (affine schemes, trivial or constant bundles, open subsets of euclidean space, etc.) along a topology by using an appropriate class of gluing maps. In our case, the local models are the characteristic bundles and the topology is replaced by the Boolean algebra of measurable sets. In this view, a bundle is measurable because it can ``be reached'' by a countable colimit of characteristic bundles and the problem has shifted from the \emph{definition} of measurability to the \emph{existence} of infinitary colimits. An important difference is that while in geometry, the gluing is in general made along an open cover with arbitrary cardinality, with measurable bundles we have to restrict to \emph{countable} covers since anything of higher cardinality is bound to destroy measurability. The next example lifted from \cite{yetter:measurable-categories2005} shows this.

\begin{example} \label{example:higher-cardinality-kills-measurability}
Let $E\subseteq X$ be a non-measurable set and let $\family{E}{i}$ be a, necessarily uncountable, partition of $E$ into measurable sets. For example, if singleton sets are measurable just take $\set{\set{x}\colon x\in E}$. Then:
\begin{equation*}
    \characteristic(E)\field\isomorphic \sum_{i}\characteristic(E_{i})\field
\end{equation*}

But the identity map on $\characteristic(E)\field$ does not have measurable norm.
\end{example}

Even though example \ref{example:higher-cardinality-kills-measurability} shows that there are differences, the analogy with algebraic geometry and schemes is too good to pass up and it is worth elaborating upon it. The first thing to remember is that to define sheaves we do not need a topological space $X$, we only need the \emph{locale} of open sets $\open(X)$. The $\sigma$-Boolean algebras of measure theory are not locales because they are not order complete, only $\sigma$-complete (or countably complete), so we have to resort to the next best thing: \emph{Grothendieck topologies}\footnote{P.~Johnstone's terminology of \emph{coverage} is closer to conveying the underlying idea: the formalization of the basic properties of covers in topological spaces permitting the definition of sheaves.} on categories. In simple terms, a Grothendieck topology is a gadget that allows us to define sheaves over the category. In our case, a Boolean algebra is viewed as a category in the usual way: there is a (unique) arrow $E\to F$ iff $E\subseteq F$. The exact definition of Grothendieck topology can be found in many textbooks (e.g.~\cite[chapter III]{maclane-moerdijk:sheaves-geometry-logic1992}); for now, we content ourselves with conveying the basic ideas. This leads us to our third principle.

\begin{principle} \label{principle:measurable-bundles-sheaves}
The category of measurable bundles is a category of ``locally constant'' sheaves of Hilbert spaces on a suitable site.
\end{principle}

The Grothendieck topology of principle \ref{principle:measurable-bundles-sheaves} we have in mind is the so-called \emph{countable join topology} (or \emph{$\sigma$-topology}). A rigorous definition will be given later, but for the moment it suffices to think of it as restricting the arbitrary cardinality covers by open sets in the topological case to countable covers by measurable sets. The ``locally constant'' buzzword is supposed to remind the reader that a general measurable bundle is obtained from characteristic bundles in a manner reminiscent of how schemes are obtained from affine schemes.

\smallskip
Another connection with sheaves on topological spaces can be established by recalling that Stone duality (for example, see \cite[chapter II]{johnstone:stone-spaces1982}), given by associating to a Boolean algebra $\Omega$ its Stone space $\Stone(\Omega)$, yields an equivalence between the category of Boolean algebras and the dual of the category of compact Hausdorff, totally disconnected spaces. This suggests that it may be possible to transfer sheaves over the site $\Omega$ to sheaves as we know and love them over the topological space $\Stone(\Omega)$. This is indeed possible but it is more delicate\footnote{The first problem is that the unit isomorphism $\eta\colon\Omega\to \clopen(\Stone(\Omega))$ of the Stone equivalence is only \emph{finitely} order continuous so that the transfer of a sheaf over $\Omega$ to $\Stone(\Omega)$ gives a presheaf \emph{over the clopens} and satisfying the patching condition for \emph{finite covers} only. The second problem is with the definition of a Banach space-valued sheaf on topological spaces. The two are intimately related -- see subsection \ref{subsection:banach-sheaves} for the details.}, so we leave it aside for now.

\smallskip
There is yet a third, more analytical, description of measurable bundles as certain modules over the Banach algebra $\Integrable_{\infty}(X)$. This description is equally important but since we have no need of it right now, we pass on immediately to our second concern in this section.

\subsection{Linearity of the direct integral}
\label{subsection:linearity-direct-integral}

%Paragraph:
%- modules and module functors to interpret direct integrals.

As is well-known, even by most calculus students, the integral is linear in the integrands and we have the equality,
\begin{equation} \label{equation:linearity-ordinary-integral}
    \integral{X}{\parens{kf + lg}}{\mu}= k\integral{X}{f}{\mu} + l\integral{X}{g}{\mu}
\end{equation}
for every pair of integrable functions $f, g$ and every pair of scalars $k, l\in \field$. By subsection \ref{subsection:hilb-categorified-ring} we already know how to interpret the right-hand side of \eqref{equation:linearity-ordinary-integral} in the categorified setting. The bit of categorial structure of $\IBundle(X)$ that we need to interpret the left-hand side is that of a \emph{module over a symmetric ring category} as defined for example, in \cite{yetter:categorical-linear-algebra}. The next table contains the relevant information.
\begin{table}[htbp]
  \begin{tabular}{|c|c|}
    \hline
    Module $M$ & Categorified module $\Integrable(X)$ \\
    \hline
    \hline
    Zero $0\in M$ & Zero bundle $x\mapto \zero$ \\
    Sum $v + w\in M$ & Pointwise direct sum $\xi\oplus \zeta$ \\
    Scalar action $kv\in M$ & Pointwise tensor product $H\otimes \xi$ \\
    \hline
  \end{tabular}
  \smallskip
  \caption{Categorified analogues of the module operations}
  \label{table:categorified-analogues-module-operations}
\end{table}

%Paragraph:
%- \sigma-additivity of direct integral.

Pointwise direct sums have already been defined in subsection \ref{subsection:measurability-bundles} and they interact well with the direct integral functor.

\begin{theorem} \label{theorem:direct-integral-commutes-infinite-sums}
Let $\sequence{\xi}{n}$ be a sequence of measurable bundles. Then there is an isometric isomorphism
\begin{equation} \label{isomorphism:direct-integral-commutes-infinite-sums}
    \sum_{n}\integral{X}{\xi_{n}}{\mu}\isomorphic \integral{X}{\parens{\sum_{n}\xi_{n}}}{\mu}
\end{equation}
\end{theorem}
\begin{proof}
To a sequence of sections $s_{n}$ of $\xi_{n}$ living in $\integral{X}{\parens{\sum_{n}\xi_{n}}}{\mu}$ we associate the section,
\begin{equation*}
    x\mapto \sum_{n}s_{n}(x)
\end{equation*}
where we identify a section of $\xi_{n}$ with a section of $\sum_{n}\xi_{n}$ via the obvious isometric embedding. That this map is measurable and bijective is more or less clear and that it is an isometry amounts to the norm equality,
\begin{equation*}
    \sum_{n}\power{\norm{s_{n}}_{2}}{2}= \power{\norm{\sum_{n}s_{n}}_{2}}{2}
\end{equation*}
which is true because the $s_{n}$ are orthogonal in $\integral{X}{\parens{\sum_{n}\xi_{n}}}{\mu}$.
\end{proof}

Theorem \ref{theorem:direct-integral-commutes-infinite-sums} not only encompasses the finite linearity of the integral but it is natural to view it as the expression of the fact that the \emph{direct integral is $\sigma$-additive}.

\smallskip
Next, we turn to the pointwise tensor product. If $\pair{\xi}{\sequence{s}{n}}$ and $\pair{\zeta}{\sequence{t}{m}}$ are two measurable bundles, their pointwise tensor product is the bundle $x\mapto \xi_{x}\otimes \zeta_{x}$ with pervasive sequence (as always, after suitable reindexing) given by rational linear combinations of elements of the form,
\begin{equation*}
    s_{n}\otimes t_{m}
\end{equation*}
where $s_{n}$ and $t_{m}$ are sections in the pervasive sequences of the respective bundles. The bifunctor $\pair{\xi}{\zeta}\mapto \xi\otimes \zeta$ endows $\IBundle(X)$ with a symmetric monoidal structure.

\smallskip
Unfortunately, the formulation of the universal property is even more difficult because of the fact that $\Hilb$ is not closed is now compounded by difficult bundle measurability issues. Nevertheless, if we make use of the notion of \emph{module functor} (see \cite{yetter:categorical-linear-algebra}) the categorified scalar invariance of the integral is easy to formulate.

%Paragraph:
%- equivariance of direct integral.

\smallskip
If $H$ is the scalar field $\field$, then the characteristic line bundle $\characteristic(E)\field$ introduced in subsection \ref{subsection:measurability-bundles} will be denoted simply by $\characteristic(E)$. Identifying an Hilbert space $H$ with the characteristic bundle $\characteristic(X)H$ and making use of the pointwise tensor product, we have the isometric isomorphism
\begin{equation} \label{isomorphism:scalar-tensor-hilb}
    \characteristic(E)H\isomorphic \characteristic(E)\otimes H
\end{equation}

\smallskip
As per the last row of \ref{table:categorified-analogues-module-operations}, the action of $\Hilb$ on $\IBundle(X)$ is given by
\begin{equation*}
    \pair{\xi}{H}\mapto \xi\otimes H
\end{equation*}

The next theorem now shows that the direct integral commutes with the $\Hilb$ action -- the \emph{categorified version of the fact that the integral is equivariant for the action of the scalar field}.

\begin{theorem} \label{theorem:direct-integral-tensors}
Let $\xi$ be a measurable bundle. Then there is an isometric isomorphism
\begin{equation} \label{isomorphism:direct-integral-tensors}
    \integral{X}{\parens{\xi\otimes H}}{\mu}\isomorphic \parens{\integral{X}{\xi}{\mu}}\otimes H
\end{equation}
\end{theorem}
\begin{proof}
See \cite[chapter II, section 8]{dixmier:von-neumann-algebras1981}.
\end{proof}

%Paragraph:
%- direct integral of characteristic line bundles.

Since the direct integral is the extension of the $\Integrable_{2}$-functor on measure spaces to a functor on the category of measurable bundles, the computation of the direct integral of characteristic line bundles is trivial. This isomorphism is so important that we record it in the next theorem.

%ToDo:
%- identify L_2(E, \mu).

\begin{theorem} \label{theorem:integral-characteristic-line-bundles}
If $\xi$ is the characteristic line bundle $\characteristic(E)$ of a measurable subset $E\subseteq X$, then there is an isometric isomorphism:
\begin{equation} \label{isomorphism:integral-characteristic-line-bundles}
    \integral{X}{\characteristic(E)}{\mu}\isomorphic \Integrable_{2}(E, \mu)
\end{equation}
\end{theorem}

%Paragraph:
%- categorified measures as cosheaves.

In ordinary measure theory we have the equality $ \integral{X}{\characteristic(E)}{\mu}= \mu(E)$. If we take our notations seriously, then isomorphism \eqref{isomorphism:integral-characteristic-line-bundles} is telling us that $E\mapto \Integrable_{2}(E, \mu)$ is playing the role of \emph{categorified measure}. The first thing to notice about this functor is that it is \emph{covariant} with domain the $\sigma$-algebra $\Omega$ of measurable sets of $X$ -- it is a \emph{precosheaf}\footnote{Not being a native English speaker, I do not know which, if any, of \emph{copresheaf} or \emph{precosheaf} is more correct. The latter just sounds slightly less horrible to my internal ear.}. And we can even isolate the crucial \emph{$\sigma$-additivity property}:

\begin{theorem} \label{theorem:L2-countably-additive}
Let $\sequence{E}{n}$ be a countable partition of $E$. Denote by $\iota_{E_{n}, E}$ the inclusion map $\Integrable_{2}(E_{n}, \mu)\to \Integrable_{2}(E, \mu)$. Then the induced map
\begin{equation} \label{isomorphism:sigma-additivity-L2}
    \sum_{n}\iota_{E_{n}, E}\colon \sum_{n}\Integrable_{2}(E_{n}, \mu)\to \Integrable_{2}(E, \mu)
\end{equation}
is an isometric isomorphism.
\end{theorem}
\begin{proof}
In essence, this is a special case of theorem \ref{theorem:direct-integral-commutes-infinite-sums}.
\end{proof}

We will not prove it here, but it can be shown that theorem \ref{theorem:L2-countably-additive} implies that the functor $E\mapto \Integrable_{2}(E, \mu)$ satisfies the \emph{dual patching condition} for \emph{finite covers} and it fails the patching condition for \emph{countable covers} simply because the countable direct sums in theorem \ref{theorem:L2-countably-additive} are not coproducts. We have arrived at the fourth principle of categorified measure theory.

\begin{principle} \label{principle:categorified-measures-cosheaves}
A categorified measure is a cosheaf of Hilbert spaces over the $\sigma$-algebra of measurable sets and any measure $\mu$ on $X$ yields a categorified measure via $E\mapto \Integrable_{2}(E, \mu)$.
\end{principle}

\subsection{Continuity of the direct integral}
\label{subsection:continuity-direct-integral}

%Paragraph:
%- integrals as limits of sums.

Integrals are (or ought to be) limits of sums. Recall that a function $f\colon X\to \field$ is \emph{simple} if there is a finite family of pairwise disjoint measurable sets $E_{n}\subseteq X$ and scalars $k_{n}\in \field$ such that,
\begin{equation*}
    f\aeequal \sum_{n}k_{n}\characteristic(E_{n})
\end{equation*}

The connection between the integral map $\int_{X}\differential\mu$ and the measure $\mu$ is the equality
\begin{equation}\label{equation:integral-measure-characteristic}
    \integral{X}{f}{\mu}= \sum_{n}k_{n}\mu(E_{n})
\end{equation}

%Paragraph:
%- infinitary colimits as categorified notion of convergence.

The left-hand side of \eqref{equation:integral-measure-characteristic} is then extended to the whole space of integrable functions by a limiting procedure. In subsections \ref{subsection:measurability-bundles} and \ref{subsection:linearity-direct-integral} the calculations with measurable bundles and direct integrals show that once the direct integral of characteristic line bundles in known, the direct integral of every measurable bundle is known. Pursuing this analogy seriously, we are inevitably led to our fifth principle.

\begin{principle} \label{principle:infinitary-colimits-categorified-convergence}
Infinite direct sums and more generally infinitary colimits, give the categorified notion of convergence.
\end{principle}

The need for infinitary colimits has arisen from a consideration of $\sigma$-additivity, but it has already appeared explicitly before in principle \ref{principle:category-measurable-bundles} when we ventured that the category of measurable bundles is the closure of the category of characteristic bundles under countable colimits. Principle \ref{principle:infinitary-colimits-categorified-convergence} is formulated in a necessarily vague way, not the least because as we have been insistently repeating, the category $\Hilb$ does not have infinite (co)products. Let us for the moment forget such an unfortunate circumstance, and proceed as if all the categories in sight were cocomplete.

%Paragraph:
%- integral as left Kan extension of measure.

\smallskip
Denote by $\SBundle(X)$ the \emph{category of characteristic line bundles}. The objects are the characteristic line bundles $\characteristic(E)$, which we can identify with the measurable subsets $E\subseteq X$, and a morphism $E\to F$ is an essentially bounded measurable morphism $\characteristic(E)\to \characteristic(F)$. The constructions of section \ref{section:measurable-bundles-hilbert-spaces} restricted to this category yield a functor
\begin{equation} \label{map:measure-functor}
    \mu\colon \SBundle(X)\to \Hilb
\end{equation}
given on objects by $E\mapto \mu(E)= \Integrable_{2}(E, \mu)$.

\smallskip
On the other hand, consider the fully-faithful inclusion of $\SBundle(X)$ into the category of measurable bundles $\IBundle(X)$. The integral functor is then a lift of $\mu\colon E\mapto \mu(E)$ as in diagram \ref{diagram:integral-lift}.
\begin{figure}[htbp]
    \begin{equation*}
    \xymatrix{
        \IBundle(X) \ar@{-->}[r]^{\int_{X}\differential{\mu}} \ar@{}[dr]|<<<<{\Uparrow} & \Hilb \\
        \SBundle(X) \ar[ur]_{\mu} \ar[u] &
    }
    \end{equation*}
    \caption{Universal property of the direct integral functor.}
    \label{diagram:integral-lift}
\end{figure}

Diagram \ref{diagram:integral-lift} is not commutative, but only commutative up to a unique isomorphism. This isomorphism, the inner $2$-cell filling the triangle, is the natural isomorphism of theorem \ref{theorem:integral-characteristic-line-bundles}. Thinking in categorial terms, this suggests a sixth principle.

\begin{principle} \label{principle:integral-functor-left-Kan-extension}
The integral functor is the left Kan extension of $\mu$ along the inclusion of the subcategory of characteristic line bundles.
\end{principle}

We expect that the integral functor $\int_{X}\differential{\mu}$ is an extension of the $\mu$-functor. That it is a \emph{Kan extension} is the pavlovian reflex response of anyone trained in category theory (seriously, it is just another instance of principle \ref{principle:universal-properties-coherence-laws} at work). That it is a \emph{left} Kan extension (as opposed to a right one -- no pun intended) is related to the fact that it is the colimits that are important, not limits.

%Paragraph:
%- cocontinuity of direct integral.

\smallskip
If $\Hilb$ were cocomplete, by the coend formula for left Kan extensions\footnote{Left Kan extensions and coends will be revisited later in subsection \ref{subsection:presheaf-categories}. For now, we offer the reader the reference \cite[chapter X, section 4]{maclane:categories-working-mathematician1971}. As usual, we use integral notation for coends. They should not be confused with the direct integrals $\integral{X}{\xi}{\mu}$.}, and dropping the $X$ from the morphism spaces of $\IBundle(X)$, we would have
\begin{equation}\label{isomorphism:direct-integral-coend}
\begin{split}
    \integral{X}{\xi}{\mu} &\isomorphic \int^{E\in \Omega}\homset{\IBundle}{\characteristic(E)}{\xi}\otimes \mu(E) \\
        &\isomorphic \int^{E\in \Omega}\xi(E)\otimes \mu(E)
\end{split}
\end{equation}

Formula \eqref{isomorphism:direct-integral-coend} shows that just as in the ordinary integral case, the direct integral is uniquely determined by its values on the characteristic line bundles. Note that the right-hand side of the first isomorphism does not really make sense, because $\Hilb$ is not closed and thus $\homset{\IBundle}{\characteristic(E)}{\xi}$ is not a Hilbert space. The last term does make sense (forgetting momentarily that $\Hilb$ is not cocomplete) and the second isomorphism is due to the following theorem, whose easy proof is left to the reader.

\begin{theorem} \label{theorem:characteristic-yoneda}
Let $\xi$ be a measurable bundle and $E\subseteq X$ a measurable set. Then there is an isometric isomorphism,
\begin{equation} \label{isomorphism:characteristic-yoneda}
    \homset{\IBundle}{\characteristic(E)}{\xi}\isomorphic \xi(E)
\end{equation}
where $\xi(E)$ denotes the restriction of $\xi$ to $E$.
\end{theorem}

Theorem \ref{theorem:characteristic-yoneda} looks remarkably like Yoneda lemma for the inclusion $\simple(X)\to \IBundle(X)$. Coincidence? There are no coincidences in mathematics. Anyway, the preceding discussion suggests another important principle.

\begin{principle} \label{principle:direct-integral-cocontinuous}
The integral functor has strong \emph{cocontinuity properties}.
\end{principle}

%Paragraph:
%- integral as a two-step construction.

Combining the principles expounded up to now, we see that just as the ordinary integral is a two-stage process so is the categorified integral. Given a cosheaf (a categorified measure), we can define the integral of characteristic bundles. The integral of a general measurable bundle is then obtained by a limiting procedure, in this case, a left Kan extension. The properties of the integral, including its cocontinuity, should follow from this two-step construction. Formulated in these terms, we can extend this categorified integral to a large number of other categories since what we need is basically, the representability of certain functors (to have at our disposal the integral of characteristic bundles) and the existence of certain colimits (to be able to construct the left Kan extension). What we have done was actually the reverse of the historical process, since this categorified integral construction has appeared before in other contexts such as topoi\footnote{In \cite[chapter 4, section 3]{curtius:european-literature-latin-middle-ages1990}, the great German medievalist and literary critic E.~R.~Curtius, informs us that ancient rhetoric had five divisions. Of these, the most important was ``Invention'' (Latin \emph{inventio}) about which:
\begin{quote}
It is divided according to the five parts which make up the judicial oration: 1.~introduction (\emph{exordium} or \emph{prooemium}); 2.~``narrative'' (\emph{narratio}), that is, exposition of the facts in the matter; 3.~evidence (\emph{argumentatio} or \emph{probatio}); 4.~refutation of opposing opinions (\emph{refutatio}); 5.~close (\emph{peroratio} or \emph{epilogus}).
\end{quote}

About the third division, ``evidence'', E.~R.~Curtius writes:
\begin{quote}
As for the ``evidence'', antique theory made this division the field of supersubtle distinctions into which we may not enter. Essentially, every oration (including panegyrics) must make some proposition or thing plausible. It must adduce in its favor arguments which address themselves to the hearer's mind or heart. Now, there is a whole series of such arguments, which can be used on the most diverse occasions. They are intellectual themes, suitable for development and modification at the orator's pleasure. In Greek they are called $\kappa o\iota\nu o\iota$ $\tau o\pi o\iota$; in Latin, \emph{loci communes}; in earlier German, \emph{Gemein\"{o}rter}. Lessing and Kant still use the word. About 1770, \emph{Gemeinplatz} was formed after the English ``commonplace.'' We cannot use the word, since it has lost its original application. We shall therefore retain the Greek \emph{topos}. To elucidate its meaning---a topos of the most general sort is ``emphasis on inability to do justice to the subject''; a topos of panegyric: ``praise of forebears and their deeds.'' In Antiquity collections of such topoi were made. The science of topoi---called ``topics''---was set forth in separate treatises.
\end{quote}

With due deference to those that prefer the plural ``toposes'', there is some inexpressible poesy in using one and the same word for a mathematical universe and a technical term of ancient rhetoric designating a conventional -- in the various senses of the word including \emph{structural} or \emph{informing} -- literary idea.}. As will be seen below, there is a link between categorified measure theory and topos theory. Stay tuned.

%Paragraph:
%- higher-order categorifications; costacks.

\smallskip
Given a measurable subset $E\subseteq X$ we can consider the category of measurable bundles $\IBundle(E)$ on $E$. We expect it is a $2$-Hilbert space, whatever that may be. Varying $E\in \Omega$, we should obtain the categorified version of a categorified measure, that is, a precosheaf of $2$-Hilbert spaces that is a \emph{costack}\footnote{Stacks are categorified versions of sheaves. Very roughly, they are presheaves with values in a $2$-category satisfying the patching condition up to coherent isomorphism.} for the countable Grothendieck topology on $\Omega$. The pattern of higher-order categorifications should now be clearer. We record this as our last principle for this subsection.

\begin{principle} \label{principle:stack-2hilbert-spaces}
The precosheaf $E\mapto \IBundle(E)$ on $\Omega$ is a costack of $2$-Hilbert spaces for the Grothendieck countable join topology.
\end{principle}

\subsection{Universal property of the direct integral}
\label{subsection:universal-property-direct-integral}

In subsection \ref{subsection:linearity-direct-integral} we have introduced the categorified analogues of measures: cosheaves. Given a cosheaf $\mu$ and a measurable bundle $\xi$, we expect to obtain another cosheaf by taking the \emph{indefinite integral}:
\begin{equation} \label{equation:indefinite-direct-integral-cosheaf}
    E\mapto \integral{E}{\xi}{\mu}\defequal \integral{X}{\characteristic(E)\otimes \xi}{\mu}
\end{equation}

This yields an \emph{action} of the category of measurable bundles on the category of cosheaves. It also expresses the integral $\integral{X}{\xi}{\mu}$ as the \emph{evaluation} of a measure at the whole space. The Radon-Nikodym theorem holds if this action yields an equivalence between the categories of measurable bundles and cosheaves. In this case, the measurable bundles appear as \emph{relative densities}, or suitable \emph{quotients}, between cosheaf categorified measures.

\smallskip
As argued in subsection \ref{subsection:continuity-direct-integral}, \eqref{equation:indefinite-direct-integral-cosheaf} is obtained as a left Kan extension of the functor that a cosheaf $\mu$ defines on $\SBundle(X)$ and herein lies the rub. For $\Integrable_{2}$-cosheaves coming from scalar measures it is obvious that this functor is just the functor constructed in section \ref{section:measurable-bundles-hilbert-spaces} restricted to the full subcategory of characteristic line bundles, but what about for a general cosheaf? The conclusions drawn in subsections \ref{subsection:measurability-bundles}, \ref{subsection:linearity-direct-integral} and \ref{subsection:continuity-direct-integral} were based on \emph{formal analogies} between the ordinary integral of measurable functions and the direct integral of measurable bundles. In this subsection we will try to dig deeper; pragmatically this means answering the question if there is some universal property attached to $\integral{X}{\characteristic(E)}{\mu}$.

%Paragraph:
%- universal property of direct integral.

\smallskip
In order to make any kind of progress let us take a step back and direct our attention to the simplest possible case of direct integrals: take the measure space $X$ to be a finite set and for the measure $\mu$ take the counting measure. Then the direct integral $\integral{X}{\xi}{\mu}$ of a bundle $\xi$ over $X$ always exists and is simply the direct sum $\sum_{x\in X}\xi_{x}$ as defined in subsection \ref{subsection:measurability-bundles}. Principle \ref{principle:universal-properties-coherence-laws} to the fore: the most important fact related to direct sums is that they are a coproduct and \emph{defined by a universal property}. Recall that if $\family{a}{x}$ is a finite family of objects in a category (additive or not), their \emph{coproduct} is an object $\sum_{x\in X}a_{x}$ together with a cone $\family{i}{x}$ of maps $a_{x}\to \sum_{x\in X}a_{x}$ such that for every cone $\family{f}{x}$ there is a unique arrow $f$ making the diagram \ref{diagram:universal-property-coproduct} commutative for every $x\in X$.
\begin{figure}[htbp]
    \begin{equation*}
    \xymatrix{
        \sum_{x\in X}a_{x} \ar@{-->}[r]^-{f} & b \\
        a_{x} \ar[u]^{i_{x}} \ar[ur]_{f_{x}} &
    }
    \end{equation*}
    \caption{Universal property of the coproduct.}
    \label{diagram:universal-property-coproduct}
\end{figure}
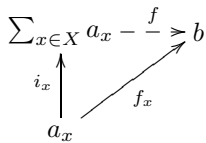

If $\family{T}{x}$ is a finite cone of bounded linear maps $H_{x}\to H$ then the induced map $T$ is the map given by $\family{h}{x}\mapto \sum_{x\in X}T_{x}(h_{x})$. In particular, if $H_{x}= H$ and the maps $T_{x}$ are the identity, then the induced map is the \emph{sum map}:
\begin{equation} \label{map:sum-coproduct-counit}
    \sum_{x\in X}H\to H\colon \family{h}{x}\mapto \sum_{x\in X}h_{x}
\end{equation}

In other words, the direct sum is not just a categorified \emph{analogue} of sums, but is \emph{directly connected} with ordinary sums via the counit map \eqref{map:sum-coproduct-counit}. The direct sum object $\sum_{x\in X}H_{x}$ is the ``object of sums $\sum_{x\in X}h_{x}$ of elements of $H_{x}$''. The sums $\sum_{x\in X}h_{x}$ are sent to a concrete sum via the counit \eqref{map:sum-coproduct-counit} as soon as we have a cone $\family{T}{x}$ of bounded linear maps.

%Paragraph:
%- introduction of RN property.

\smallskip
Another point of view on the universal property of a coproduct is given by observing that a cone $\family{T}{x}$ of maps $H_{x}\to H$ with $H_{x}= \field$ is by the isometric isomorphism $\homset{\Hilb}{\field}{H}\isomorphic H$, the same thing as a section of the bundle with total space $\coprod_{x\in X}H$ or a function $X\to H$. Rewriting the universal property of the coproduct in these terms, we see that every function $f\colon X\to H$ induces a unique bounded linear map $\sum_{x\in X}\field\to H$ as in diagram \ref{diagram:discrete-radon-nikodym}.
\begin{figure}[htbp]
    \begin{equation*}
    \xymatrix{
        \sum_{x\in X}\field \ar@{-->}[r] & H \\
        X \ar[u]^{\diagonal} \ar[ur]_{f}
    }
    \end{equation*}
    \caption{Discrete Radon-Nikodym.}
    \label{diagram:discrete-radon-nikodym}
\end{figure}

The map $\diagonal$ is the \emph{diagonal map} given by $x\mapto \delta_{x}\in \sum_{x\in X}\field $ with,
\begin{equation*}
    \delta_{x}(y)=
    \begin{cases}
        1 &\text{if $x= y$,} \\
        0 &\text{otherwise.}
    \end{cases}
\end{equation*}
and the unique map in the factorization is the sum map applied to $f$:
\begin{equation*}
    \sum_{x\in X}\field\to H\colon \family{k}{x}\mapto \sum_{x\in X}f(k_{x})
\end{equation*}

But this is nothing else than a \emph{finite, Hilbert space, vector version of the Radon-Nikodym theorem}. If this is not clear, just put $H= \field$. Then the universal property states that every bounded linear functional on $\sum_{x\in X}\field\isomorphic \ell_{2}(X)$ arises from a function $X\to \field$ by taking the indefinite integral. Since $X$ is discrete, the integral amounts to a simple finite sum.

\smallskip
This is the case for the constant line bundle $x\mapto \field$. Let us generalize a little bit and take the constant bundle $x\mapto K$ with $K$ an arbitrary Hilbert space. The cones $h_{x}\in H$ are replaced by cones $T_{x}\in \homset{\Hilb}{H_{x}}{H}$ with $H_{x}$ constant and equal to $K$. But since $\Hilb$ is a symmetric monoidal category and the tensor product commutes with direct sums, we have the chain of isomorphisms,
\begin{align*}
    \homset{\Hilb}{\sum_{x\in X}K}{H} &\isomorphic \homset{\Hilb}{\sum_{x\in X}\field\otimes K}{H} \\
        &\isomorphic \homset{\Hilb}{K\otimes \parens{\sum_{x\in X}\field}}{H} \\
        &\isomorphic \homset{\Hilb}{\sum_{x\in X}\field}{\homset{\Hilb}{K}{H}}
\end{align*}

This means that the Radon-Nikodym property for the vector case follows, that is, we have the isomorphism,
\begin{equation*}
    \homset{\Hilb}{\sum_{x\in X}K}{H}\isomorphic \int_{X}H\differential{\mu}
\end{equation*}
if the characteristic line bundles have the Radon-Nikodym property \emph{with respect to Banach spaces}, that is, we have the isomorphism
\begin{equation*}
    \homset{\Hilb}{\sum_{x\in X}\field}{\homset{\Hilb}{K}{H}}\isomorphic \integral{X}{H}{\mu}
\end{equation*}

Once again, Banach spaces are seen to be essential to the proper formulation of the universal properties. Ultimately, this is due to the fact that $\Hilb$ is not closed\footnote{In connection with these matters, it is worth remarking that although in \cite[section 5]{baez-baratin-freidel-wise:infinite-dimensional-representations-2groups2008} the authors suggest $\inner{H}{K}= \dual{H}\otimes K$ for the categorified inner product, the fact is that the space on the right-hand side is the space of Hilbert-Schmidt operators which, in the infinite-dimensional case, is but a tiny subspace of $\homset{\Hilb}{H}{K}$.}.

%Paragraph:
%- testing RN on measurable situation.

\smallskip
Now, let us test these findings on the measurable situation, starting with a characteristic line bundle $\characteristic(E)$ with $E\subseteq X$ a measurable subset. Its direct integral was computed in theorem \ref{theorem:integral-characteristic-line-bundles}. The universal property of diagram \ref{diagram:discrete-radon-nikodym} now takes the form of diagram \ref{diagram:measurable-radon-nikodym}.
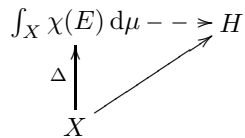
\begin{figure}[htbp]
    \begin{equation*}
    \xymatrix{
        \integral{X}{\characteristic(E)}{\mu} \ar@{-->}[r] & H \\
        X \ar[u]^{\diagonal} \ar[ur]
    }
    \end{equation*}
    \caption{Measurable Radon-Nikodym.}
    \label{diagram:measurable-radon-nikodym}
\end{figure}

But in order to make sense of diagram \ref{diagram:measurable-radon-nikodym} we need to have the integral version of the counit sum map, or we need \emph{integrals of vector-valued functions $X\to H$}. Assuming that all this works, then repeating the argument above, in order to derive the Radon-Nikodym property for more general characteristic bundles than line bundles, not only we need the measurable Radon-Nikodym property with respect to Banach spaces, we need the distributivity of $\otimes$ with respect to direct integrals -- but this is just theorem \ref{theorem:direct-integral-tensors}. Thus, diagram \ref{diagram:measurable-radon-nikodym} is telling us that the \emph{categorified integral of characteristic bundles is characterized by a universal property}. This suggests that just as the existence of direct sums is not an extra structure but a property of a category, the same happens with direct integrals. We elect this as another of our guiding principles.

\begin{principle} \label{principle:direct-integrals-universal-property}
The direct integral of a characteristic bundle is an object characterized by a universal property, the \emph{Radon-Nikodym property}.
\end{principle}

%Paragraph:
%- conceptual importance of RN property.

Principle \ref{principle:direct-integrals-universal-property} is a concrete embodiment of principle \ref{principle:universal-properties-coherence-laws}. We managed to isolate the basic object and characterize it by a universal property. Instead of being guided by more or less vague analogies, we now have a systematic means of investigation. It allows us to bring in the full force of category theory and prompts us to find suitable replacements for certain measure-theoretic and functional-analytic concepts in order to be able to also \emph{categorify the proofs}. Unfortunately, the reader should curb his natural enthusiasm because as will be seen below in subsection \ref{subsection:failure-radon-nikodym}, principle \ref{principle:direct-integrals-universal-property} cannot stand as currently formulated.

\section{From Hilbert to Banach spaces}
\label{section:from-hilbert-banach-spaces}

%Paragraph:
%- need for Banach space bundles.

In section \ref{section:towards-categorified-measure-theory}, time and again we have hinted at the fact that an extension of measurable bundles to Banach spaces is bound to be necessary. The first reason is that $\Hilb$, while symmetric monoidal is not closed, so that enriched category theory with $\Hilb$ as a base is virtually impossible. For example, if $\xi$, $\zeta$ are two $X$-bundles, it is natural to consider the bundle $\exponential{\zeta}{\xi}$ of maps $\xi\to \zeta$ given by:
\begin{equation*}
    x\mapto \homset{\Hilb}{\xi_{x}}{\zeta_{x}}
\end{equation*}

But this is a bundle of Banach spaces. In other words, not only the categories of Hilbert space bundles are bound to lack such a basic construction as exponentials, but the formulation of the critical Radon-Nikodym universal property is right-down impossible.

%Paragraph: difficulties in infinite dimensions.
%- non-abelianness of \Hilb.

\smallskip
The construction of finite coproducts in $\Hilb$ has already been given in theorem \ref{theorem:existence-coproducts-hilb}. Since $\Hilb$ is additive it has all products, and both kernels and cokernels are constructed in the usual manner. In other words, $\Hilb$ is finitely complete and finitely cocomplete. To emphasize the fact that in infinite dimensions substantially new phenomena arise, we recall the following well-known result.

\begin{theorem} \label{theorem:hilb-not-abelian}
The category $\Hilb$ is not abelian.
\end{theorem}
\begin{proof}
This follows from the fact that a monic in $\Hilb$ is an injective map and an epi is a bounded linear map with dense image. It is easy to find epi-monics in $\Hilb$ that are not isomorphisms, thus $\Hilb$ cannot be abelian.
\end{proof}

%Paragraph:
%- lack of infinitary colimits.

But the \emph{really} damaging fact is that $\Hilb$ does \emph{not} have infinite (co)products. Let $\family{H}{x}$ be an infinite family of Hilbert spaces and let us assume that the product $\prod_{x\in X}H_{x}$ existed and let $\family{p}{x}$ be the cone of projections. If $\family{T}{x}$ is a cone of bounded linear maps, then there is unique bounded linear map $T$ closing triangle \ref{diagram:universal-property-projections-product} for every $x\in X$.
\begin{figure}[htbp]
    \begin{equation*}
    \xymatrix{
            & \prod_{y\in X}H_{y} \ar[d]^{p_{x}} \\
        K \ar[r]_{T_{x}} \ar@{-->}[ur]^{T} & H_{x}
    }
    \end{equation*}
    \caption{Universal property of $\family{p}{x}$.}
    \label{diagram:universal-property-projections-product}
\end{figure}

Commutativity of triangle \ref{diagram:universal-property-projections-product} implies the norm inequality
\begin{equation*}
    \norm{T_{x}}= \norm{p_{x}T}\leq \norm{p_{x}}\norm{T}
\end{equation*}

But choosing a cone with a subsequence $\sequence{T}{x_{n}}$ such that $\norm{T_{x_{n}}}\converges \infty$, this is clearly impossible. A similar reasoning shows that infinite coproducts cannot exist.

%Paragraph:
%- space of all cones is not even normed => non-existence products.
%- cannot sum arbitrary families => non-existence of coproducts.

\smallskip
The problem in both non-existence statements lies in the fact that we are trying to parameterize \emph{all} families of operators. In the product case, this is not possible because there is no meaningful inner product in the infinite product linear space $\prod_{x\in X}H_{x}$. In the coproduct case, this is not possible because to define the map,
\begin{equation*}
    \prod_{x\in X}\homset{\Hilb}{H_{x}}{H}\to \homset{\Hilb}{\sum_{x\in X}H_{x}}{H}
\end{equation*}
we need to \emph{sum} families of elements, which is impossible in any reasonable way if the family is arbitrary. In fact, the two problems are related because the sum map gives the representability isomorphism,
\begin{equation*}
    \homset{\Hilb}{\sum_{x\in X}H_{x}}{H}\isomorphic \prod_{x\in X}\homset{\Hilb}{H_{x}}{H}
\end{equation*}
but as we have noticed already the linear space on the right-hand side is not even normed.

%Paragraph:
%- solution is extension to Banach spaces.

\smallskip
Given the crucial role of countable colimits and in particular, countable coproducts, this would seem to deal a death blow to a proper theory of categorial integrals. The solution is to extend the codomain of the measurable bundles from Hilbert to Banach spaces. At first sight, the reader may rightly wonder why this is a solution at all, since the previous example that showed that infinite coproducts of Hilbert spaces do not exist also works for the category $\Ban$ of Banach spaces and bounded linear maps. This is indeed true, but there is a way out.

\smallskip
Let $\family{B}{x}$ be an $X$-family of Banach spaces. We will employ the same functor and bundle terminology introduced in section \ref{section:measurable-bundles-hilbert-spaces} in connection with bundles of Hilbert spaces. The product of the family $\family{B}{x}$ exists in the category $\Vect$ of linear spaces and its elements are $X$-parameterized families $\family{b}{x}$ of elements $b_{x}\in B_{x}$. Define the linear subspace,
\begin{equation*}
    \sum_{x\in X}B_{x}\defequal \set{\family{b}{x}\colon\sum_{x\in X}\norm{b_{x}}< \infty}
\end{equation*}
and equip it with the norm
\begin{equation*}
    \sumnorm{\family{b}{x}}\defequal \sum_{x\in X}\norm{b_{x}}
\end{equation*}

We also define the linear subspace,
\begin{equation*}
    \prod_{x\in X}B_{x}\defequal \set{\family{b}{x}\colon\sup\set{\norm{b_{x}}\colon x\in X}< \infty}
\end{equation*}
equipped with the uniform norm
\begin{equation*}
    \supnorm{\family{b}{x}}\defequal \sup\set{\norm{b_{x}}\colon x\in X}
\end{equation*}

A simple example illustrates these definitions.

\begin{example} \label{example:sum-prod-constant-functor}
Consider the constant functor $x\mapto B$ with $B$ a Banach space. Then $\sum_{x\in X}B$ is the Banach space $\ell_{1}(X, B)$ of absolutely summable families in $B$ indexed by $X$ and $\prod_{x\in X}B$ is the Banach space $\Bounded(X, B)$ of bounded functions $X\to B$ with the supremum norm.
\end{example}

The fact that both $\sum_{x\in X}B_{x}$ and $\prod_{x\in X}B_{x}$ are Banach spaces is proved in exactly the same way as in the aforementioned classical cases. For each $x\in X$ there is an isometric embedding $i_{x}\colon B_{x}\to \sum_{x\in X}B_{x}$ that to $b_{x}\in B_{x}$ associates the absolutely summable family
\begin{equation*}
    \delta_{b_{x}}(y)\defequal
    \begin{cases}
        b_{x} &\text{if $y=x$,} \\
        0 &\text{otherwise.}
    \end{cases}
\end{equation*}

Similarly, for each $x\in X$ there is a contractive linear map $p_{x}\colon \prod_{x\in X}B_{x}\to B_{x}$ given by $(b_{x})\mapto b_{x}$. The next theorem states the universal property of these two spaces and their associated cones. Its proof is left as an exercise to the reader.

\begin{theorem} \label{theorem:existence-co-products-ban}
For any family $\family{T}{x}$ of bounded linear maps $B_{x}\to A$ such that $\sup_{x\in X}\set{\norm{T_{x}}}< \infty$, there is a unique bounded linear map $T\colon\sum_{x\in X}B_{x}\to A$ such that for every $x\in X$, diagram \ref{diagram:universal-property-injections} is commutative.
\begin{figure}[htbp]
    \begin{equation*}
    \xymatrix{
        \sum_{x\in X}B_{x} \ar@{-->}[r]^{T} & A \\
        B_{x} \ar[u]^-{i_{x}} \ar[ur]_{T_{x}} &
    }
    \end{equation*}
    \caption{Universal property of injections $i_{x}$.}
    \label{diagram:universal-property-injections}
\end{figure}

The association $\family{T}{x}\mapto T$ is an isometric isomorphism
\begin{equation} \label{isomorphism:universal-property-injections}
    \prod_{x\in X}\homset{\Ban}{B_{x}}{A}\isomorphic \homset{\Ban}{\sum_{x\in X}B_{x}}{A}
\end{equation}

For any family $\family{T}{x}$ of bounded linear maps $A\to B_{x}$ in $\prod_{x\in X}\homset{\Ban}{A}{B_{x}}$, there is a unique bounded linear map $T\colon A\to \prod_{x\in X}B_{x}$ such that for every $x\in X$, diagram \ref{diagram:universal-property-projections} is commutative.
\begin{figure}[htbp]
    \begin{equation*}
    \xymatrix{
        & \prod_{x\in X}B_{x} \ar[d]^{p_{x}} \\
        A \ar@{-->}[ur]^{T} \ar[r]_{T_{x}} & B_{x}
    }
    \end{equation*}
    \caption{Universal property of the projections $p_{x}$.}
    \label{diagram:universal-property-projections}
\end{figure}

The association $\family{T}{x}\mapto T$ is an isometric isomorphism
\begin{equation} \label{isomorphism:universal-property-projections}
    \prod_{x\in X}\homset{\Ban}{A}{B_{x}}\isomorphic \homset{\Ban}{A}{\prod_{x\in X}B_{x}}
\end{equation}
\end{theorem}

Theorem \ref{theorem:existence-co-products-ban} exhibits some notable improvements over theorem \ref{theorem:existence-coproducts-hilb}. The indexing set $X$ is arbitrary and the representability isomorphisms are isometries. And as mentioned previously, although infinite (co)products do not exist in $\Ban$, it is clear that theorem \ref{theorem:existence-co-products-ban} is expressing a universal property of a (co)product object. It expresses a uniqueness factorization property for \emph{bounded cones}, precisely the elements of the Banach space $\prod_{x\in X}\homset{\Ban}{A}{B_{x}}$. Denote by $\Banc$ the category of Banach spaces and linear contractions, that is, those bounded linear maps $T$ with operator norm $\norm{T}\leq 1$. If we restrict attention to elements in the \emph{unit ball} of $\prod_{x\in X}\homset{\Ban}{A}{B_{x}}$, that is, \emph{contractive cones}, then theorem \ref{theorem:existence-co-products-ban} implies:

\begin{theorem} \label{theorem:banc-small-(co)products}
The category $\Banc$ has all small products and all small coproducts.
\end{theorem}

%Paragraph:
%- \Banc complete, cocomplete, symmetric monoidal closed.

Equalizers and coequalizers are constructed in the usual manner so that $\Banc$ is both complete and cocomplete. For the benefit of the reader, we recall that if $A$ and $B$ are Banach spaces, the \emph{projective tensor product}\footnote{Although the projective tensor product is in a sense the canonical one, there are other tensor products and useful ones at that. For example, later on we will make use of the \emph{injective tensor product}. Just as there is no unique way to norm a direct sum (think of the whole family of $p$-norms) there are also other reasonable norms one can put on the tensor product. The definition of what constitutes a reasonable norm, their connections with spaces of operators, etc.~is an important area in Banach space theory. Good references are \cite{ryan:introduction-tensor-product-banach-spaces2002} and, more demanding but also more exhaustive in its survey, \cite{defant-floret:tensor-norms-operator-ideals1993}.} $A\projotimes B$ is the completion of the linear tensor product $A\otimes B$ under the \emph{projective norm},
\begin{equation} \label{def-equation:projective-norm}
    \projnorm{u}= \inf\set{\sum_{n}\norm{a_{n}}\norm{b_{n}}\colon u= \sum_{n}a_{n}\otimes b_{n}}
\end{equation}
where the infimum is taken over all the representations of $u$ as a sum $\sum_{n}a_{n}\otimes b_{n}$ of elementary tensors.

\smallskip
The projective tensor product makes of $\Ban$ a closed symmetric monoidal category. Since this structure descends to $\Banc$ (all the coherence isomorphisms are isometries), it follows that we have a good theory of \emph{categories enriched in $\Banc$}. The two single biggest problems detected in section \ref{section:towards-categorified-measure-theory}, the lack of closedness and infinitary colimits have just been solved.

%Paragraph:
%- Banach categories.

\smallskip
Categories of Banach spaces have received a good deal of attention in the 1970's. For our purposes, the monograph \cite{cigler-losert-michor:banach-modules-functors-categories-banach-spaces1979} contains most of what we need. While not couched in the language of enriched category theory that is, in practice, what the authors are doing. Our basic reference for enriched category theory is \cite{kelly:enriched-category-theory2005}, but to simplify the discussion it is convenient to introduce some terminology specific to the Banach space case.

\begin{definition} \label{definition:banach-categories}
A $\Banc$-enriched category will be called a \emph{Banach category}\footnote{Deference to tradition and settled custom has led me to this choice of terminology, but I have to confess my envy of a physicist's freedom for scavenging \emph{Finnegans Wake} for accurate and descriptive names.}. A $\Banc$-enriched (or \emph{strong}, or even \emph{contractive}) functor between Banach categories is an ordinary functor $F\colon \mathcal{A}\to \mathcal{B}$ such that the map
\begin{equation*}
    f\mapto F(f)
\end{equation*}
is a linear contraction.

If $F, G\colon \mathcal{A}\to \mathcal{B}$ are two contractive functors, a $\Banc$-enriched natural transformation $F\to G$ is an ordinary natural transformation $\tau\colon F\to G$ that is \emph{uniformly bounded}, that is:
\begin{equation*}
    \sup\set{\norm{\tau_{a}}\colon a\in \mathcal{A}}< \infty
\end{equation*}

It should be noted that the Banach space of natural transformations $F\to G$ is an object of $\Banc$ only if the domain category $\mathcal{A}$ is small. For the moment, we forego such size technicalities but eventually we will have to come to grips with them. Taking unit balls, we obtain the $2$-category $\BanCat$ of Banach categories, contractive functors and contractive natural transformations.
\end{definition}

The remarks preceding definition \ref{definition:banach-categories} suggest the tenth and the last of our list of fundamental principles.

\begin{principle} \label{principle:categorified-measure-theory-banach}
Categorified measure theory lives in the $2$-category $\BanCat$.
\end{principle}

%Paragraph:
%- consequences of opening the door to Banach spaces.

Later, we will sharpen principle \ref{principle:categorified-measure-theory-banach} by taking into account principle \ref{principle:infinitary-colimits-categorified-convergence} and throwing cocompleteness conditions into the mix. What is to be emphasized here is that the important $2$-category is the $2$-category of Banach categories and that all notions are relative to (or enriched in) Banach spaces. This has some noteworthy consequences. For example, since the isomorphisms in $\Banc$ are the linear bijective isometries we will be interested mainly in \emph{isometric Banach space theory}; the distinction between coproducts and products becomes important because even in the finite case they are \emph{not} isometrically isomorphic, etc.

\smallskip
While the passage to $\Banc$-enriched categories gives us the necessary categorial completeness conditions, it also lands us in the very harsh Banach space land. The long litany of wild and intricate counterexamples is not only a testament to the ingenuity of Banach space theorists but also the definitive proof of the unsuitability of Banach spaces for many purposes\footnote{The underlying motivation for undertaking the categorification of measure theory is to build TQFT's and quantum gravity models. In this respect, one should notice that the naive definition of Banach-space valued TQFT's is bound to encounter difficulties. The reason is that every object $M$ of the cobordism category $\Cobord$ is reflexive while there are many non-reflexive Banach spaces. Even worse, the projective tensor product of two reflexive spaces need not be reflexive. For a striking example, we refer the reader to the computation of the tensor diagonal of $\ell_{p}\projotimes \ell_{p}, p\geq 1$ in \cite[section 2.5]{ryan:introduction-tensor-product-banach-spaces2002}.}. Let one example suffice for many. In a Hilbert space, the existence of the orthogonality concept implies that for every closed linear subspace $M$ there is a canonical norm $1$ projection with range $M$ and a canonical topological complement $\orthogonal{M}$. The relatively simple and tame linear geometry of Hilbert spaces allows some deep theorems on bounded linear operators (e.g.~the spectral theorem) which in its turn allows a reasonably satisfying representation theory for many objects (e.g.~locally compact groups). This idyllic situation is blown to smithereens with Banach spaces because their linear geometry can be extremely complicated. While in a Hilbert space every closed subspace is complemented, in a general Banach space the best general result is that finite-dimensional and finite-codimensional subspaces are complemented. In \cite{lindenstrauss-tzafriri:classical-banach-spacesI1977} it is proved that every infinite-dimensional Banach space not isomorphic to a Hilbert space has one non-trivial (that is, neither finite-dimensional nor finite-codimensional) uncomplemented subspace. In \cite{gowers-maurey:unconditional-basic-sequence1993}, the authors constructed the first known example of a Banach space whose every non-trivial closed subspace is uncomplemented. By now there are known several examples of such spaces (commonly called \emph{hereditarily indecomposable}), including examples of $\Continuous(X)$-spaces for cunningly crafted compact Hausdorff $X$ (see \cite{koszmider:banach-spaces-continuous-functions-few-operators2004} and \cite{plebanek:ck-spaces-few-operators2004}).

\smallskip
The passage from Hilbert to Banach spaces has other important consequences to categorified measure theory. Ultimately, it is the very simple isometric classification of separable Hilbert spaces that allows a decomposition such as that of theorem \ref{theorem:measurability-countable-coproduct}. In fact, one can even say that the countability restrictions in such principles as principle \ref{principle:category-measurable-bundles} is something of a red herring, imposed on us by a conspiracy between bad measurability definitions and the simple classification of separable Hilbert spaces. With Banach spaces the situation is completely different. Let $A$ and $B$ be two Banach spaces of linear dimension $n$. The spaces $A$ and $B$ are isomorphic in $\Ban$ but they are not necessarily isometric isomorphic, that is, isomorphic in $\Banc$. Define $d(A, B)$ to be the quantity
\begin{equation}\label{equation-definition:banach-mazur-metric}
    \log\inf\set{\norm{T}\norm{\inverse{T}}\colon \text{$T$ is a linear isomorphism $A\to B$}}
\end{equation}

\smallskip
Definition \eqref{equation-definition:banach-mazur-metric} gives a metric to the set $\mathcal{M}(n)$ of equivalence classes of $n$-dimensional Banach spaces for the equivalence relation of isometric isomorphism. The metric space $\mathcal{M}(n)$ is called the \emph{Banach-Mazur compactum} and is a compact metric space. In \cite{ageev-bogatyi-repovs:banach-mazur-compactum2004} it is proved that it is homeomorphic to the Alexandroff compactification of a Hilbert cube manifold -- a humongous space. The moral is that a measurable bundle $\xi$, even if it has for values only finite dimensional Banach spaces, can be a very complicated beast.

%Paragraph:
%- other embeddings.

\smallskip
Our choice to extend $\Hilb$ to $\Ban$ is a matter of convenience. The category $\Ban$ lies nearest at hand and allows us to remain within the confines of classical functional analysis. More generally, any symmetric monoidal closed category $\mathcal{A}$ that is both complete and cocomplete and in which $\Hilb$ embeds in a sufficiently nice way can serve as a candidate for base category. The category $\mathcal{A}$ should also have the categorial structure of $\Hilb$ that makes it so important for quantum mechanics (basically, a contravariant endofunctor on $\mathcal{A}$ that distills the properties of adjoint operators in $\Hilb$ -- precise definitions can be found in \cite{baez:hda2-2-hilbert-spaces1997} and \cite{selinger:dagger-compact-closed-categories-completely-positive-maps2007}) and the embedding should preserve it\footnote{However, note that a general observable in quantum mechanics is a self-adjoint, \emph{not necessarily bounded} operator.}. Independently of the embedding $\Hilb\to \mathcal{A}$ we choose and assuming it exists in the first place, we believe that the categorified measure theory briefly described in section \ref{section:categorified-measures-integrals} is sufficiently robust and general that it extends to this new setting in a straightforward way.

\subsection{The failure of the Radon-Nikodym property}
\label{subsection:failure-radon-nikodym}

As remarked in the beginning of this section, while passage to the category of Banach spaces gives us the necessary categorial completeness conditions, we pay a heavy price in having to deal with objects far more complicated than Hilbert spaces. This fact will surface on many occasions, the first and most important being the failure of the Radon-Nikodym property, which, as argued in section \ref{section:towards-categorified-measure-theory}, is fundamental for a good categorified measure theory. The plain matter of fact is that the Radon-Nikodym property fails in infinite dimensions and fails badly.

%Paragraph:
%- RN property again.

\smallskip
Let us return to the principles formulated in section \ref{section:towards-categorified-measure-theory}, especially principle \ref{principle:direct-integrals-universal-property}. The universal property implied by it was depicted in diagram \ref{diagram:measurable-radon-nikodym}, which we depict again next.
\begin{figure}[htbp]
    \begin{equation*}
    \xymatrix{
        \integral{X}{\characteristic(E)}{\mu} \ar@{-->}[r]^-{\integral{X}{f}{\mu}} & H \\
        X \ar[u]^{\diagonal} \ar[ur]_{f}
    }
    \end{equation*}
    \caption{Measurable Radon-Nikodym.}
\end{figure}

In terms of representability isomorphisms, it says that the map $f\mapto \integral{X}{f}{\mu}$ establishes a natural isometric isomorphism,
\begin{equation*}
    \Integrable_{\infty}(X, B)\isomorphic \homset{\Ban}{\Integrable_{1}(X)}{B}
\end{equation*}
between a space of measurable functions $X\to B$ and the space of bounded linear operators $\Integrable_{1}(X)\to B$. Using integral notation, we can write,
\begin{equation} \label{isomorphism:rnp-integral-form}
    \prod_{X}B\differential{\mu}\isomorphic \Ban\parens{\integral{X}{\characteristic(E)}{\mu}, B}
\end{equation}
so that $\Integrable_{\infty}(X, B)$ is a \emph{measurable power} while $\Integrable_{1}(X)$ is a \emph{measurable copower}. The Radon-Nikodym property then states that the functor $B\mapto \Integrable_{\infty}(X, B)$ is \emph{representable}.

%Paragraph:
%- failure of RN: definitions.

\smallskip
The space $\Integrable_{\infty}(X, B)$ is the Banach space of a.e.~equivalence classes of ``measurable'' functions $f\colon X\to B$ with the essential supremum norm. The adjective measurable is between quotes because measurability here is \emph{not} what one would naively expect: measurability of $f$ with $B$ endowed with the Borel measurable structure. This latter notion of measurability is \emph{almost} good enough but the reasons why it is not good enough are subtle\footnote{Subtle because finding an example of a Borel measurable, not strongly measurable function is directly connected with questions undecidable in ZFC. We refer the reader to \cite{edgar:measurality-banach-space1977} and \cite{edgar:measurality-banach-spacesII1979} for the details.}.

\smallskip
In subsection \ref{subsection:continuity-direct-integral} we recalled the notion of \emph{simple function}: a function that up to a null-measure set is a linear combination of characteristic functions. We will need a generalization.

\begin{definition} \label{definition:strongly-measurable}
An \emph{elementary function} is a function $f\colon X\to B$ such that there is a countable partition $\family{E}{n}$ of $X$ into measurable sets and a sequence $\sequence{b}{n}$ in $B$ such that:
\begin{equation*}
    f\aeequal \sum_{n}\characteristic(E_{n})b_{n}
\end{equation*}

We say that a function $f\colon X\to B$ is \emph{strongly measurable} if it is the a.e.~uniform limit of a sequence of elementary functions. We now define $\Integrable_{\infty}(X, B)$ as the space of a.e.~equivalence classes of bounded strongly measurable functions with the essential supremum norm.
\end{definition}

The next theorem, which we have already alluded to before, is fundamental. The proof (as well as the proof of the Bochner theorem \ref{theorem:bochner} below) can be found in \cite[chapter 2, section 3]{ryan:introduction-tensor-product-banach-spaces2002}. Chapter 2 of \cite{diestel-uhl:vector-measures1977} is another reference.

\begin{theorem}[Pettis] \label{theorem:pettis}
For a function $f\colon X\to B$ the following are equivalent:
\begin{enumerate}
  \item \label{theor-enum:pettis1}
  $f$ is strongly measurable.

  \item \label{theor-enum:pettis2}
  $f$ is Borel measurable and there is a conegligible\footnote{A measurable subset $E\subseteq X$ is \emph{conegligible} if its complement has null measure. We remind the reader of our standing assumption that the measure $\mu$ is complete.} set $E\subseteq X$ such that $\dirimage{f}(E)$ is separable.

  \item \label{theor-enum:pettis3}
  $f$ is \emph{weakly measurable}, that is, for every bounded linear functional $b^{\ast}$, the function $b^{\ast}f$ is measurable, and there is a conegligible set $E\subseteq X$ such that $\dirimage{f}(E)$ is separable.
\end{enumerate}
\end{theorem}

With the strong measurability notion, the (strong or Bochner) integral of functions is developed just like the Lebesgue integral. If $f$ is an elementary function, we define its integral to be,
\begin{equation} \label{def-equation:integral-elementary-functions}
    \integral{X}{f}{\mu}\defequal \sum_{n}\mu(E_{n})b_{n}
\end{equation}
and say that $f$ is \emph{integrable} if the series on the right-hand side of \eqref{def-equation:integral-elementary-functions} converges. The integral is then extended in the obvious way to strongly measurable functions. The next criterion for integrability is very useful.

\begin{theorem}[Bochner] \label{theorem:bochner}
A strongly measurable function $f\colon X\to B$ is integrable iff the function $x\mapto \norm{f(x)}$ is integrable.
\end{theorem}

Basically, all the properties of scalar integrals lift unscathed to the vector case. One that we will need below, and that follows from the Pettis and Bochner theorems, is the following:

\begin{theorem} \label{theorem:simples-dense-integrables}
The linear space of simple functions $X\to B$ is dense in the space $\Integrable_{1}(X, B)$ of integrable functions.
\end{theorem}

The integral of a measurable vector function $g$ yields a natural isometric embedding, the \emph{Radon-Nikodym map},
\begin{equation*}
    \Integrable_{\infty}(X, B)\to \homset{\Ban}{\Integrable_{1}(X)}{B}
\end{equation*}
given by
\begin{equation*}
    g\mapto \parens{f\mapto \integral{X}{gf}{\mu}}
\end{equation*}

\smallskip
Bounded linear maps $T$ that are in the image of the Radon-Nikodym map are called \emph{representable} and the function $g\in\Integrable_{\infty}(X, B)$ representing $T$ is called its \emph{Radon-Nikodym derivative}. A Banach space $B$ such that every bounded linear map $\Integrable_{1}(X)\to B$ is representable for every probability space $X$ is said to have the \emph{Radon-Nikodym property}, or RNP for short.

%Paragraph:
%- failure of RN.

\smallskip
The scalar Radon-Nikodym theorem (for example, see \cite[section 31]{halmos:measure-theory1974}) implies that the map is indeed surjective for $B= \field$. Since $B\mapto \homset{\Ban}{\Integrable_{1}(X)}{B}$ is a representable functor, it is continuous in the Banach space variable and from the fact that $B\mapto \Integrable_{\infty}(X, B)$ is \emph{finitely continuous} it follows that every finite limit of Banach spaces with RNP has RNP. Note also that since what is in question is the \emph{surjectivity} of the Radon-Nikodym map, it follows that the Radon-Nikodym property is an invariant of the linear homeomorphism type. Many other classes of spaces are known to have the Radon-Nikodym property, such as all reflexive spaces, all separable dual spaces, countable coproducts of RNP spaces, etc.~but alas, it is not universally true. The problem is readily apparent if we note that the unit $X\to \Integrable_{1}(X)$ is the inverse image of the identity under the would-be-isomorphism Radon-Nikodym map. The obvious candidate for such a universal map is $x\mapto \delta_{x}$ with $\delta_{x}$ given by:
\begin{equation*}
    \delta_{x}(y)\defequal
    \begin{cases}
        1 &\text{if $y= x$,} \\
        0 &\text{otherwise.}
    \end{cases}
\end{equation*}

%ToDo:
%- why is L_\infty not a counter-example?

But if the measure space is atomless,
\begin{equation*}
    \delta_{x}\aeequal 0
\end{equation*}
so that this cannot work. This is no proof that $\Integrable_{1}(X)$ does not have the Radon-Nikodym property\footnote{This line of argument can be pursued and justified somewhat. The presence of atoms in the measure space guarantees the existence of extreme points in the unit ball of $\Integrable_{1}(X)$ while if $X$ is atomless it can be shown that the unit ball has no extreme points whatsoever. This suggests that there is an intimate relation between a geometric property, the existence of extreme points, and an analytic one, the existence of Radon-Nikodym derivatives. The single most important unsolved problem in the area of vector measures on Banach spaces is precisely the conjectured equivalence between the Krein-Milman property (every closed bounded convex set is the closed convex hull of its extreme points) and the Radon-Nikodym property. That RNP implies KMP is known and is due to J.~Lindenstrauss (\cite{lindenstrauss:extreme-points-l11966}). As for the converse, in \cite[section 5.4, page 118]{albiac-kalton:topics-banach-space-theory2006} the authors remark that ``It is probably fair to say that the subject has received relatively little attention since the 1980's and some really new ideas seem to be necessary to make further progress.'' We refer the reader to \cite[chapter VII]{diestel-uhl:vector-measures1977} for more information on this subject.}, but it sure is a strong indication that in infinite dimensions the Radon-Nikodym theorem fails.

\begin{theorem} \label{theorem:representable-nuclear-Linfty}
Denote by $\iota$ the inclusion $\Integrable_{\infty}(X)\to \Integrable_{1}(X)$. If the bounded linear map $T\colon \Integrable_{1}(X)\to B$ is representable then the composite $T\iota$ is compact.
\end{theorem}
\begin{proof}
Suppose $T$ is representable, that is, there is $g\in \Integrable_{\infty}(X, B)$ such that:
\begin{equation*}
    T(f)= \integral{X}{gf}{\mu}
\end{equation*}

By Bochner's theorem, the function $g$ is Bochner integrable. Density of simple functions implies the existence of a sequence $\sequence{g}{n}$ such that $g_{n}\converges g$ in $\Integrable_{1}(X, B)$. The linear maps
\begin{equation*}
    T_{n}(f)\defequal \integral{X}{g_{n}f}{\mu}
\end{equation*}
on $\Integrable_{\infty}(X)$ have finite rank. From,
\begin{equation*}
    \norm{\integral{X}{(g - g_{n})f}{\mu}}\leq \norm{f}_{\infty}\integral{X}{\norm{g - g_{n}}}{\mu}
\end{equation*}
we can conclude that the sequence of operators $T_{n}$ converges to $T\iota$, which implies that $T\iota$ is compact.
\end{proof}

Theorem \ref{theorem:representable-nuclear-Linfty} can be improved. A bounded linear map $T\colon\Integrable_{1}(X)\to B$ is representable iff the composite map $T\iota$ is \emph{nuclear}. Recall that a linear map $A\to B$ is nuclear if it is in the image of the natural map $\dual{A}\projotimes B\to \homset{\Ban}{A}{B}$. The quotient of $\dual{A}\projotimes B$ by the kernel of this map is the space $\nuclear(A, B)$ of nuclear operators. This result, originally due to A.~Grothendieck, will not be needed since the theorem as stated is more than enough for our purposes. If $X$ is an infinite atomless probability space such as the unit interval with the Lebesgue measure, consider the identity $\Integrable_{1}(X)\to \Integrable_{1}(X)$. By theorem \ref{theorem:representable-nuclear-Linfty}, if the identity were representable, the unit ball of $\Integrable_{\infty}(X)$ would be relatively compact in $\Integrable_{1}(X)$, which is clearly and obviously false\footnote{On second thought, this may not be entirely obvious. To show that the unit ball of $\Integrable_{\infty}(X)$ is not relatively compact it suffices to construct a uniformly bounded sequence $\sequence{f}{n}$ in $\Integrable_{\infty}(X)$ with no $\Integrable_{1}$-convergent subsequence. The fact that $X$ is atomless allows us to employ exhaustion (see \cite[section 41, exercises 2 and 3]{halmos:measure-theory1974} or \cite[chapter 1, section 5]{fremlin:measure-theory22001}) to, starting with the singleton partition $\set{X}$ of $X$, recursively build a partition $\mathcal{E}_{n + 1}$ of $X$ with $\power{2}{n}$ sets by dividing each set of $\mathcal{E}_{n}$ in two sets of equal measure. With this setup, we will now simulate \emph{Rademacher functions} on $X$. Define the sequence $\sequence{r}{n}$ by,
\begin{equation*}
    r_{n}\defequal \sum_{E_{m, n}\in \mathcal{E}_{n}}\epsilon_{m}\characteristic(E_{m, n})
\end{equation*}
with $\epsilon_{m}\in \set{-1, 1}$ (the unit sphere of $\Real$) defined by
\begin{equation*}
    \epsilon_{m}\defequal
    \begin{cases}
        1 &\text{if $m$ is even,} \\
        -1 &\text{otherwise.}
    \end{cases}
\end{equation*}

It is easy to see that $\sumnorm{r_{n} - r_{m}}= \mu(X)$ for $n\neq m$. Leaping ahead of ourselves, we remark that this construction extends almost word for word to $\sigma$-complete atomless measure algebras.}.

%Paragraph:
%- difficulties in fixing RN property.

\smallskip
Pondering over theorem \ref{theorem:representable-nuclear-Linfty}, the reader may suspect that the failure of RNP is due to our definition of the space $\Integrable_{\infty}(X, B)$, a space too small to represent all bounded linear maps, and if only we were ingenious enough a suitable generalized definition of a.e.~bounded measurable functions could be found to make the Radon-Nikodym map surjective. Unfortunately, even a cursory perusal of \cite[chapter VI]{dunford-schwartz:linear-operatorsI1958} will show that to make the Radon-Nikodym map surjective, we have to change the space $\Integrable_{\infty}(X, B)$ and then in all cases one ends up having one or more of the following defects:
\begin{enumerate}
  \item \label{enum:rn-defects1}
  The Radon-Nikodym map is not surjective for some Banach spaces $B$.

  \item \label{enum:rn-defects2}
  The representing function $g$ takes values in a space other than $B$.

  \item \label{enum:rn-defects3}
  The representing function $g$ is not a function at all but a measure, whose Radon-Nikodym derivative, if it existed, would be the representing function.
\end{enumerate}

In section \ref{section:measure-algebras-integrals} we will provide a Radon-Nikodym theorem in the form \eqref{enum:rn-defects3}, which will be important in our quest to categorify measure theory. In subsection \ref{subsection:stone-space-ba-weak-integrals} we will give an inkling of how Radon-Nikodym theorems of the form \eqref{enum:rn-defects2} can be obtained when discussing the fundamental Stone space construction of a Boolean algebra $\Omega$.

%Paragraph:
%- how to proceed.

\smallskip
The failure of the Radon-Nikodym property in infinite dimensions is a blow in our dreams of building a proper categorified measure theory. But not all is lost. A possible course of action is to salvage what can be salvaged: we already know that RNP holds uniformly for all finite dimensional spaces so this could be exploited to find weaker forms of the RNP universal property. This course of action could be pursued, but not only it would lead us far into the so-called \emph{local Banach space theory} (see \cite{pietsch:what-is-local-banach-space-theory1999} and references therein), but at this point it is not even clear that it brings any significant dividends, so we leave it aside for some better opportunity. The more obvious course of action however, is to simply forget (half of) principle \ref{principle:direct-integrals-universal-property} and proceed with what we have and this is precisely what we will do. As section \ref{section:categorified-measures-integrals} unfolds, the importance of the Radon-Nikodym property will recede into nearly microscopic proportions, until finally our perseverance is rewarded and in the end, almost out of the blue, the universal property for $\integral{X}{\characteristic(E)}{\mu}$ will fall on our laps and the final truth revealed. But this universal property depends on some of the deeper properties of the category of cosheaves and to get there, a significant amount of ground will have to be covered.

\section{Foregrounding: measure algebras and integrals}
\label{section:measure-algebras-integrals}

If the reader looks closely on the work done up to now, he will notice that the underlying set $X$ of a measure space was hardly ever needed and was dragged along like an appendage whose original function has been long forgotten. In this section we will excise it, throw it in the trash bin, and concentrate on the real core of measure theory: a Boolean algebra $\Omega$ and a measure $\mu$ defined on it, or a \emph{measure algebra} as will be defined below.

%Paragraph:
%- importance of measure algebras.

\smallskip
At this juncture, some readers will raise their hands in protest and remark that the underlying set $X$ does fulfill a function: it allows us to define the $\Integrable_{p}$-spaces of functions. The answer to this objection is that the $\Integrable_{p}$-spaces \emph{can} be constructed for a general Boolean algebra and even throw a considerable light on the whole categorification process. The reason for this is actually quite simple: the abstract construction of the $\Integrable_{p}$-spaces brings to the foreground the \emph{universal properties of these spaces} (principle \ref{principle:universal-properties-coherence-laws} again!) and these depend not on the underlying set $X$ but on the measure algebra $\pair{\Omega}{\mu}$. In general, we believe that a good deal of conceptual clarification is achieved and, most importantly for us, they are an essential stepping stone for categorifying measure theory. Since most readers will probably not have seen these universal properties\footnote{Understandably, textbooks on measure theory do not mention these universal properties as they call for concepts such as vector measures, Banach spaces, semivariation, etc., which have no place in introductory texts.} we will devote this section to them\footnote{As said elsewhere, the main motivation for this work arises from physical considerations related with TQFT's and quantum gravity models. One of the main objectives of quantum gravity is to \emph{derive} space-time from the physical theory and one obvious place to start from is the structure of the (state independent) propositions that can be made about the system. Classically, these are Boolean algebras (e.g.~see \cite[chapter 5]{jauch:foundations-qm1968}). See recent work connecting topos theory with quantum mechanics \cite{doring-isham:what-is-a-thing2008} and \cite{heunen-landsman-spitters:topos-algebraic-quantum-theory2007}. The connection of this paper with these works will become clearer in subsection \ref{subsection:banach-sheaves}.}.

\smallskip
For measure algebras, the obvious reference is the encyclopedic \cite{fremlin:measure-theory32002}. The reader can also find there appropriate references for the theory of Boolean algebras, including Stone duality. For the latter, \cite{johnstone:stone-spaces1982} is eminently suitable. For vector measures, there is nothing better than to turn to \cite{diestel-uhl:vector-measures1977}. As it happens frequently, many of the results we find ourselves quoting do not come in the exact versions we require. Many are but trifle modifications of known results, while others require one or two additional ideas, usually of a category-theoretic nature. Since their number is Legion, we will herd them all together in a future survey paper (\cite{rodrigues:vector-measures-boolean-algebras}).

%Paragraph:
%- basic measure algebra definitions.

\smallskip
We start by introducing a couple of definitions.

\begin{definition} \label{definition:measure-algebra}
A \emph{measure algebra} is a pair $\pair{\Omega}{\mu}$ where $\Omega$ is a Boolean algebra and $\mu$ is a bounded, positive, finitely additive map $\Omega\to \Real$.
\end{definition}

Comparing with the corresponding definition in \cite[chapter 2, section 1]{fremlin:measure-theory32002}, there are a few differences. First, we work only with bounded (or \emph{totally finite}) measures. This is necessary if we are to exploit the strong link between the measure-theoretic and the functional-analytic sides. Second, we do not require the measure $\mu$ to be \emph{non-degenerate}, that is:
\begin{equation*}
    \mu(E)= 0\implies E= 0
\end{equation*}

This is a minor technicality since we can always quotient $\Omega$ by the ideal of null-measure sets (see below). As will be seen in the continuation, virtually every functor that we define on the category of measure algebras will factor through the subcategory of non-degenerate measure algebras. The measure algebra definition \ref{definition:measure-algebra} just happens to be slightly more convenient for our categorially-minded purposes.

%Paragraph:
%- didactic reasons for finite additivity.

\smallskip
The more substantial difference is however, that we do not require either $\sigma$-completeness of $\Omega$ or $\sigma$-additivity of $\mu$. There are many reasons for this laxness, starting with the fact that the more important measures that we will meet are spectral measures and these are never $\sigma$-additive except in trivial cases. Restricting ourselves to finitely additive measures also has a didactic purpose. All the deep theorems of measure theory need some sort of $\sigma$-completeness hypothesis (Lebesgue's dominated theorem, Radon-Nikodym theorem, etc.). Contrapositively, all the proofs sketched below for finitely additive measures are not only easy but they are \emph{easy to categorify}. This is not to say that $\sigma$-completeness hypothesis are not needed or important; they are and very much so. But by restricting ourselves to finitely additive measures, we can more easily underline the categorial underpinnings essential for categorifying measure theory.

%Paragraph:
%- \mu-supported maps and the null ideal.

\smallskip
It is easy to see that $\mu$ is \emph{monotone}, that is, if $E\subseteq F$ then $\mu(E)\leq \mu(F)$. Because of this, $\mu$ is bounded iff $\mu(1)< \infty$. We repeat that unless specified otherwise, we will \emph{only} consider bounded measures. We denote by $\Null(\mu)$ the ideal of $\mu$-null elements. In symbols:
\begin{equation} \label{equation-definition:measure-null-ideal}
    \Null(\mu)\defequal \set{E\in \Omega\colon \mu(E)= 0}
\end{equation}

We can get rid of these $\mu$-null sets by taking the quotient $\pi\colon\Omega\to \Omega\quotient \Null(\mu)$. The universal property of this quotient with respect to measures is easy to state. A finitely additive $\nu$ is \emph{$\mu$-supported} (or \emph{has $\mu$-support})\footnote{Classically, for \emph{$\sigma$-additive measures on $\sigma$-complete Boolean algebras}, having $\mu$-support is equivalent to a ``uniform continuity'' property and is called \emph{absolute $\mu$-continuity}. Reusing the classical terminology would be a mistake however, because $\mu$-support is a very weak requirement when the $\sigma$-completeness conditions are dropped. The paper \cite{wendt:category-disintegration1994} uses the term \emph{measure reflecting}, but the direction is reversed since the author is working with measure spaces.} if for every $E\in \Omega$ with $\mu(E)= 0$ then $\nu(E)= 0$. We now have that a measure $\nu$ on $\Omega$ is $\mu$-supported iff there is a (unique) finitely additive $\universal{\nu}$ on the quotient $\Omega\quotient \Null(\mu)$ such that diagram \ref{diagram:universal-property-quotient-null-ideal} is commutative.
\begin{figure}[htbp]
    \begin{equation*}
    \xymatrix{
        \Omega\quotient\Null(\mu) \ar@{-->}[r]^-{\universal{\nu}} & V \\
        \Omega \ar[u]^{\pi} \ar[ur]_{\nu} &
    }
    \end{equation*}
\caption{Universal property of quotient $\pi\colon\Omega\to \Omega\quotient \Null(\mu)$.}
\label{diagram:universal-property-quotient-null-ideal}
\end{figure}

This universal property says that for most purposes we can forget about the ideal of $\mu$-null sets and simply assume that $\mu$ is \emph{non-degenerate}, that is $\Null(\mu)= \zero$.

%Paragraph:
%- measure-algebra and Lipschitz maps.

\smallskip
A morphism $\pair{\Omega}{\mu}\to \pair{\Sigma}{\nu}$ of measure algebras is a Boolean algebra morphism $\phi\colon \Omega\to \Sigma$ such that $\nu\phi$ has $\mu$-support. Note that formally, a measure algebra $\pair{\Omega}{\mu}$ can be identified with the measure since the Boolean algebra can always be recovered as the domain of $\mu$. Thus, the category of measure algebras and measure algebra maps will be denoted by $\Meas$. Without any sort of order completeness properties (and e.g.~Radon-Nikodym's theorem) the morphisms are too weak for some purposes and at times it will be necessary to require more of $\phi$. The morphism $\phi$ is \emph{Lipschitz} if there is a constant $C$ such that:
\begin{equation} \label{inequality:lipschitz}
    \nu\phi(E)\leq C\mu(E), \forall E\in \Omega
\end{equation}

If $\phi$ is a Lipschitz morphism we denote by $\lipschitz{\phi}$, its \emph{Lipschitz norm}, the quantity $\inf C_{\phi}$ where $C_{\phi}$ is the set of constants $C$ in the conditions of \eqref{inequality:lipschitz}. Lipschitz morphisms are needed for example, to establish the functoriality of the $\Integrable_{1}$-functor on measure algebras.

\subsection{The Banach algebra \texorpdfstring{$\Integrable_{\infty}(\Omega)$}{of measurable bounded functions}}
\label{subsection:banach-algebra-linfty}

%Paragraph:
%- spectral measures; the Banach algebra \Integrable_{\infty}.

One of the more striking theorems on linear operators is the spectral theorem, which in one form (\cite[chapter IX, section 2]{conway:course-functional-analysis1990}) represents every bounded normal operator $T$ on a separable Hilbert $H$ as an integral over a canonical spectral measure on the Borel sets of the spectrum of $T$. The extension of this cycle of ideas to Banach spaces was done by W.~G.~Bade in the 1950's. The original papers are \cite{bade:boolean-algebras-projections-algebras-operators1955} and \cite{bade:mulplicity-theory-boolean-algebras-projections-banach-spaces1959}, but see also \cite{dunford-schwartz:linear-operators31971} or the more recent monograph \cite{ricker:operator-algebras-commuting-projections1999}, whose outlook is closer to our needs. As we noticed in section \ref{section:from-hilbert-banach-spaces} there are known by now a handful of Banach spaces with very few projections so it is matter of some importance to know conditions that guarantee the existence of enough projections. We offer to the reader \cite{orihuela-valdivia:projective-generators-resolutions-identity-banach-spaces1989} as an entry point into this vast subject. In this sketchy subsection we are interested only in some very simple aspects of these objects which will be enough to construct the spectral measure associated to a cosheaf.

%Notation:
%- inclusion map \iota.

\smallskip
The theory of $\Integrable_{\infty}(\Omega)$ depends only on the fact that $\Omega$ is a Boolean algebra and necessitates no order completeness properties or the choice of a measure on $\Omega$. The first step is the construction of the space of ``simple functions on $\Omega$''. This is simply the quotient of the free linear space $\Free(\Omega)$ by the subspace generated by,
\begin{equation*}
    \iota(E\cup F) - \iota(E) - \iota(F)
\end{equation*}
for all pairs of disjoint $E, F\in \Omega$. Denote the quotient space, the space of \emph{simple elements}, by $\simple(\Omega)$, and by $\pi$ the quotient linear map $\Free(\Omega)\to \simple(\Omega)$. The composite $\Omega\namedto{\iota}\Free(\Omega)\namedto{\pi}\simple(\Omega)$, the \emph{characteristic map}, is denoted by $\characteristic$. By construction it is finitely additive and the universal such among finitely additive maps with values in linear spaces -- see diagram \ref{diagram:universal-property-simple-elements}.
\begin{figure}[htbp]
    \begin{equation*}
    \xymatrix{
        \simple(\Omega) \ar@{-->}[r] & V \\
        \Omega \ar[u]^{\characteristic} \ar[ur]_{\nu}
    }
    \end{equation*}
    \caption{Universal property of $\characteristic\colon \Omega\to \simple(\Omega)$.}
    \label{diagram:universal-property-simple-elements}
\end{figure}

\smallskip
To describe the universal property inside the category of Banach spaces, we start by introducing a norm on $\simple(\Omega)$. First, we need a lemma to the effect that every simple element has a representation as a linear combination of characteristic elements.

\begin{lemma} \label{lemma:canonical-representation-simples}
For every non-zero $f\in \simple(\Omega)$ there is a partition $\family{E}{n}$ and non-zero scalars $k_{n}\in \field$, such that:
\begin{equation} \label{equation:canonical-representation-simples}
    f= \sum_{n}k_{n}\characteristic(E_{n})
\end{equation}
\end{lemma}
\begin{proof}
For each $f\in \Free(\Omega)$ there are elements $E_{n}\in \Omega$ and scalars $k_{n}\in \field$ such that $f= \sum_{n}k_{n}\characteristic(E_{n})$. The lemma is now proved by induction on the number of terms and using some basic Boolean algebra manipulations.
\end{proof}

We now introduce a norm on $\simple(\Omega)$ by picking a representation of $f$ as in lemma \ref{lemma:canonical-representation-simples} and putting:
\begin{equation} \label{equation-definition:Linfty-norm}
    \supnorm{f}\defequal \max\set{\abs{k_{n}}}
\end{equation}

It is not difficult to see that \eqref{equation-definition:Linfty-norm} does not depend on the chosen representation. The completion of $\simple(\Omega)$ under the norm \eqref{equation-definition:Linfty-norm} is denoted by $\Integrable_{\infty}(\Omega)$.

\smallskip
For the additive maps side, denote by $\partitions(E)$ the set of finite partitions of $E\in \Omega$. If $E= 1$ then this set is also denoted by\footnote{Or the other way around, the unit or top element $1$ of a Boolean algebra $\Omega$ will also be denoted by $\Omega$.} $\partitions(\Omega)$. This set is ordered by \emph{refinement}: a partition $\mathcal{E}$ \emph{refines} $\mathcal{F}$ if for every $F\in \mathcal{F}$ there is $E\in \mathcal{E}$ such that $E\subseteq F$. Under this partial order, $\partitions(E)$ is \emph{filtered}.

\begin{definition} \label{definition:semivariation}
Let $B$ be a Banach space and $\nu\colon\Omega\to B$ an additive map. The \emph{semivariation} of $\nu$ on $E$ is the (possibly infinite) quantity:
\begin{equation} \label{equation:semivariation}
    \semivariation{\nu}(E) = \sup\set{\sum_{F\in \mathcal{E}}\abs{\inner{b^{\ast}}{\nu(F)}}\colon b^{\ast}\in \ball(\dual{B}), \mathcal{E}\in \partitions(E)}
\end{equation}

The map $\nu$ has \emph{bounded semivariation} if for every $E\in \Omega$, $\semivariation{\nu}(E)< \infty$.
\end{definition}

It is hard to grasp what exactly is the semivariation measuring and to be totally honest, I do not have any illuminating thoughts to offer on this matter.

\smallskip
The semivariation defines a map $E\mapto \semivariation{\nu}(E)$ on $\Omega$ that is a positive, monotone and finitely subadditive, thus $\nu$ has bounded semivariation iff $\semivariation{\nu}(1)< \infty$. Only slightly more difficult than this triviality is the fact that $\nu$ has bounded semivariation iff it has bounded range. We denote by $\additive(\Omega, V)$ the linear space of additive maps $\Omega\to V$ and by $\badditive(\Omega, B)$ the normed space of bounded finitely additive maps $\Omega\to B$ with the semivariation norm.

\begin{theorem} \label{theorem:universal-property-Linfty}
For every bounded additive map $\nu\colon\Omega\to B$ with $B$ a Banach space, there is a unique bounded linear map $\int\differential\nu\colon\Integrable_{\infty}(\Omega)\to B$ such that the triangle \ref{diagram:universal-property-Linfty} is commutative.
\begin{figure}[htbp]
    \begin{equation*}
    \xymatrix{
        \Integrable_{\infty}(\Omega) \ar@{-->}[r]^{\int\differential\nu} & B \\
        \Omega \ar[u]^{\characteristic} \ar[ur]_{\nu} &
    }
    \end{equation*}
    \caption{Universal property of $\Integrable_{\infty}(\Omega)$.}
    \label{diagram:universal-property-Linfty}
\end{figure}

The map $\nu\mapto \int\differential\nu$ is a natural linear isometric isomorphism
\begin{equation} \label{isomorphism:universal-property-Linfty}
    \badditive(\Omega, B)\isomorphic \homset{\Ban}{\Integrable_{\infty}(\Omega)}{B}
\end{equation}
\end{theorem}
\begin{proof}
The proof is mostly a matter of checking details. First, if $T$ is a bounded linear map then $T\characteristic$ is an additive map with $\semivariation{T\characteristic}\leq \norm{T}\semivariation{\characteristic}$. The trickiest thing may now be to prove that the semivariation of $\characteristic$ is $\leq 1$ -- if the reader is having difficulties we refer him to \cite[page 3, example 7]{diestel-uhl:vector-measures1977}. On the converse direction, the universal property \ref{diagram:universal-property-simple-elements} gives the lift of $\nu$ to $\simple(\Omega)$ and a simple computation shows it to be bounded for the norm \eqref{equation-definition:Linfty-norm}, thus it lifts uniquely to $\Integrable_{\infty}(\Omega)$.
\end{proof}

Theorem \ref{theorem:universal-property-Linfty} can seem quite innocent, but when coupled with other high-powered tools such as Stone duality, it can yield some significant results. For example, we refer the reader to \cite{garling:short-proof-riesz-representation-theorem1973} for a very elegant proof of the Riesz representation theorem. An exposition of this proof can also be found in \cite[chapter 16]{carothers:short-course-banach-space-theory}.

\smallskip
We emphasize the conceptual role of theorem \ref{theorem:universal-property-Linfty}. The integral map is introduced as a byproduct of the universal property of $\Integrable_{\infty}(\Omega)$, exactly the type of result we need to guide us in categorifying measure theory. Operator algebraists will readily recognize it. Note also that some of the formal properties of the integral are implicit in \ref{theorem:universal-property-Linfty}. For example, the naturality of the isomorphism \eqref{isomorphism:universal-property-Linfty} in both variables amounts to the equalities,
\begin{equation} \label{equation:naturality-integral-map}
\begin{split}
    T\parens{\int f\differential\nu} &= \int f\differential T\nu \\
    \int \covariant{\phi}(f)\differential{\nu} &= \int f\differential \pullback{\phi}\nu
\end{split}
\end{equation}
for every bounded linear map $T$ and every Boolean algebra morphism $\phi$. The second equality is a cheap form of the change of variables formula. The measure $\pullback{\phi}\nu$, the \emph{pullback of $\nu$ under $\phi$}, is simply
\begin{equation} \label{equation-definition:pullback-measure}
    \pullback{\phi}\nu\colon E\mapto \nu(\phi(E))
\end{equation}

%Paragraph:
%- spectral measures.

\smallskip
A Banach algebra structure\footnote{In fact, by theorem \ref{theorem:stone-measurable-continuous-map} below $\Integrable_{\infty}(\Omega)$ is a $C^{\ast}$-algebra and with suitable order completeness hypothesis on $\Omega$, even a Von-Neumann algebra. This operator algebra viewpoint is very fruitful, probably essential, but will not be pursued here.} can be introduced on $\Integrable_{\infty}(\Omega)$ such that $\characteristic(E\cap F)= \characteristic(E)\characteristic(F)$, that is, the additive map $\characteristic\colon \Omega\to \Integrable_{\infty}(\Omega)$ is \emph{spectral} or \emph{multiplicative}. The universal property of theorem \ref{theorem:universal-property-Linfty} can be extended to yield an isometric isomorphism between the spaces of Banach algebra morphisms and the bounded spectral finitely additive measures\footnote{Spectral measures are hardly ever $\sigma$-additive, although in many cases they are $\sigma$-additive for certain weaker, locally convex linear topologies. For example, the spectral measure of the spectral theorem is $\sigma$-additive for the strong operator topology (the analysts name for the topology of pointwise convergence).}. Spectral measures are essential in the categorified integral theory as will be seen later and because of this we insert the next theorem, which is a more or less trivial extension of \ref{theorem:universal-property-Linfty} to spectral maps.

\begin{theorem} \label{theorem:universal-property-Linfty-algebra}
Let $A$ be a Banach algebra and $\nu\colon\Omega\to A$ a bounded spectral map. Then there is a unique bounded algebra morphism $\int\differential\nu\colon \Integrable_{\infty}(\Omega)\to A$ such that the diagram in figure \ref{diagram:universal-property-Linfty-algebra} is commutative.
\begin{figure}[htbp]
    \begin{equation*}
    \xymatrix{
        \Integrable_{\infty}(\Omega) \ar@{-->}[r]^{\int\differential\nu} & A \\
        \Omega \ar[u]^{\characteristic} \ar[ur]_{\nu} &
    }
    \end{equation*}
    \caption{Universal property of the Banach algebra $\Integrable_{\infty}(\Omega)$.}
    \label{diagram:universal-property-Linfty-algebra}
\end{figure}

The association $\nu\mapto \int\differential\nu$ establishes a natural isometric isomorphism between the set of bounded spectral measures $\Omega\to A$ and the set of bounded algebra morphisms $\Integrable_{\infty}(\Omega)\to A$,
\begin{equation} \label{isomorphism:universal-property-Linfty-algebra}
    \bsadditive(\Omega, A)\isomorphic \homset{\BanAlg}{\Integrable_{\infty}(\Omega)}{A}
\end{equation}
\end{theorem}

%ToDo:
%- measures and holonomies on the path category of \Omega.

\subsection{The Bochner integral}
\label{subsection:bochner-integral}

%Paragraph:
%- the Bochner integral.

In this subsection we revisit the strong (or Bochner) integral and show how all the basic properties follow easily with \emph{easy} proofs. Since our objective is to support our introductory injunction that a good deal of conceptual clarification is achieved if we consider Boolean algebras and measures on it, we consider only finitely additive measures.

\smallskip
Let $\pair{\Omega}{\mu}$ be a measure algebra. In subsection \ref{subsection:banach-algebra-linfty} we have defined the space of simple elements $\simple(\Omega)$. We extend the definition to measure algebras by defining:
\begin{equation} \label{equation-definition:simples-measure-algebra}
    \simple(\Omega, \mu)\defequal \simple(\Omega\quotient \Null(\mu))
\end{equation}

The notation $\simple(\Omega, \mu)$ is slightly misleading, because this space only depends on $\Null(\mu)$ not on the specific values of $\mu$. On $\simple(\Omega, \mu)$ we have the norm:
\begin{equation}\label{equation-definition:L1-Bochner-norm}
    \sumnorm{\sum_{n}k_{n}\characteristic(E_{n})}\defequal \sum_{n}\abs{k_{n}}\mu(E_{n})
\end{equation}

It is not difficult to see that the right-hand side of \eqref{equation-definition:L1-Bochner-norm} is independent of the chosen representation as a linear combination of characteristic elements. Finally, we define the \emph{space of simple elements with values in $B$} to be the linear tensor product,
\begin{equation} \label{equation-definition:simples-value-Banach-space}
    \simple(\Omega, \mu, B)\defequal \simple(\Omega, \mu,)\otimes B
\end{equation}
and define a norm on it by
\begin{equation}\label{equation-definition:L1-projective-norm}
    \sumnorm{\sum_{n}f_{n}\otimes b_{n}}\defequal \sum_{n}\sumnorm{f_{n}}\norm{b_{n}}
\end{equation}

Obviously, the space $\simple(\Omega, \mu, B)$ reduces to $\simple(\Omega, \mu)$ when $B$ is the real field. The space $\Integrable_{1}(\Omega, \mu, B)$ is the completion of $\Integrable_{1}(\Omega, \mu, B)$ under the norm \eqref{equation-definition:L1-projective-norm}. Before defining the integral map, we need a lemma giving a canonical form for the elements of $\simple(\Omega, \mu, B)$.

\begin{lemma} \label{lemma:combination-tensor-disjoint-charateristic}
For every non-zero $f\in \simple(\Omega, \mu, B)$, there exist non-zero pairwise disjoint elements $E_{n}\in \Omega$ and elements $b_{n}\in B$ such that
\begin{equation} \label{equation:combination-tensor-disjoint-charateristic}
    f = \sum_{n}\characteristic(E_{n})\otimes b_{n}
\end{equation}
\end{lemma}
\begin{proof}
This is a combination of lemma \ref{lemma:canonical-representation-simples} with the properties of linear tensor products.
\end{proof}

%Paragraph:
%- vector integrals obtained by tensoring scalar integrals.

There is a map $\int\differential\mu\colon \simple(\Omega, \mu, B)\to B$ given by
\begin{equation*}
    \sum_{n}\characteristic(E_{n})\otimes b_{n}\mapto \sum_{n}\mu(E_{n})b_{n}
\end{equation*}

This map is linear and contractive for the $\sumnorm{,}$-norm, therefore it extends uniquely to a linear contraction $\Integrable_{1}(\Omega, \mu, B)\to B$. The next theorem identifies the space $\Integrable_{1}(\Omega, \mu, B)$ with a projective tensor product and, in particular, it will show that $\Integrable_{1}(\Omega, \mu, B)$ is cocontinuous in the Banach space variable. It will also prove that not only the vector integral $\int\differential\mu$ commutes with every bounded linear map, but also that $\int\differential\mu$ is the scalar integral tensored with the identity $\id_{B}$.

\begin{theorem} \label{theorem:l1-tensor-product}
The bilinear map $\simple(\Omega, \mu)\times B\to \Integrable_{1}(\Omega, \mu, B)$ given by
\begin{equation*}
    \pair{f}{b}\mapto f\otimes b
\end{equation*}
establishes a natural isometric isomorphism,
\begin{equation} \label{isomorphism:l1-tensor-product}
    \Integrable_{1}(\Omega, \mu, B)\isomorphic \Integrable_{1}(\Omega, \mu)\projotimes B
\end{equation}
where $\projotimes$ denotes the projective tensor product of Banach spaces. Furthermore, diagram \ref{diagram:naturality-integral} is commutative.
\begin{figure}[htbp]
    \begin{equation*}
    \xymatrix{
        \Integrable_{1}(\Omega, \mu)\projotimes A \ar[rrr]^{\id_{\Integrable_{1}(\Omega, \mu)}\projotimes T} \ar[dd]_{\isomorphic} \ar[dr]_{\int\differential\mu\projotimes \id_{A}} &  &  & \Integrable_{1}(\Omega, \mu)\projotimes
        B \ar[dd]^{\isomorphic} \ar[dl]^{\int\differential\mu\projotimes \id_{B}} \\
          & A \ar[r]^{T} & B & \\
        \Integrable_{1}(\Omega, \mu, A) \ar[rrr]_{\covariant{T}} \ar[ur]^{\int\differential\mu} &  &  &
        \Integrable_{1}(\Omega, \mu, B) \ar[ul]_{\int\differential\mu}
    }
    \end{equation*}
    \caption{Naturality of $\int\differential\mu$.}
    \label{diagram:naturality-integral}
\end{figure}
\end{theorem}
\begin{proof}
Given our definitions, the proof consists in a simple check that the projective norm is equal to the norm \eqref{equation-definition:L1-projective-norm} and a few trivial computations with simple functions.
\end{proof}

Next, we treat the universal property of $\Integrable_{1}(\Omega, \mu, A)$. By theorem \ref{theorem:l1-tensor-product}, we have the chain of isomorphisms,
\begin{equation} \label{isomorphism:universal-property-general-L1}
\begin{split}
    \homset{\Ban}{\Integrable_{1}(\Omega, \mu, A)}{B} &\isomorphic \homset{\Ban}{\Integrable_{1}(\Omega, \mu)\projotimes A}{B} \\
        &\isomorphic \homset{\Ban}{\Integrable_{1}(\Omega, \mu)}{\Ban(A, B)}
\end{split}
\end{equation}
so that it is sufficient to discuss the universal property of $\Integrable_{1}(\Omega, \mu)$. In order to do that, we introduce another class of measures, the \emph{Lipschitz measures}.

\begin{definition} \label{definition:dominated-measure}
Let $B$ be a Banach space. An additive map $\nu\colon\Omega\to B$ is \emph{$\mu$-Lipschitz} (or simply \emph{Lipschitz} if the scalar measure $\mu$ is fixed throughout) if there is a positive constant $C$ such that:
\begin{equation} \label{equation:measure-dominated}
    \norm{\nu(E)}\leq C\mu(E), \forall E\in \Omega
\end{equation}
\end{definition}

It is easy to see that the measure $\nu$ is $\mu$-Lipschitz iff the quotient map $\Omega\to \Omega\quotient \Null(\nu)$ is Lipschitz according to \eqref{inequality:lipschitz}. If $\nu$ is $\mu$-Lipschitz, then we also say that \emph{$\nu$ is $\mu$-dominated} and $\mu$ is a \emph{dominating measure}. This terminology parallels the terminology for \emph{control measures}. If $\nu$ is $\mu$-Lipschitz, we define the \emph{$\mu$-Lipschitz norm} by:
\begin{equation} \label{equation-definition:measure-lipschitz-norm}
    \lipschitz{\nu}\defequal \inf C_{\nu}
\end{equation}
where $C_{\nu}$ is the set of constants $C$ in the conditions of \eqref{equation:measure-dominated}. The Lipschitz norm is an upper bound for the \emph{distortion ratio} $\norm{\nu(E)}/\mu(E)$. We define $\ladditive(\Omega, \mu, B)$ to be the space of $\mu$-Lipschitz additive maps $\Omega\to B$ and with norm the $\mu$-Lipschitz norm \eqref{equation-definition:measure-lipschitz-norm}.

\begin{theorem} \label{theorem:universal-property-L1}
If $\nu\colon\Omega\to B$ is a $\mu$-Lipschitz, finitely additive map, then there is a unique bounded linear map $\int\differential{\nu}\colon\Integrable_{1}(\Omega, \mu)\to B$ such that triangle \ref{diagram:universal-property-L1} is commutative.
\begin{figure}[htbp]
    \begin{equation*}
    \xymatrix{
        \Integrable_{1}(\Omega, \mu) \ar@{-->}[r]^-{\int\differential\nu} & B \\
        \Omega \ar[u]^{\characteristic} \ar[ur]_{\nu} &
    }
    \end{equation*}
    \caption{Universal property of $\Integrable_{1}(\Omega, \mu)$.}
    \label{diagram:universal-property-L1}
\end{figure}

The map $\nu\mapto \int\differential\nu$ establishes a natural isometric isomorphism
\begin{equation} \label{isomorphism:universal-property-L1}
    \ladditive(\Omega, \mu, B)\isomorphic \homset{\Ban}{\Integrable_{1}(\Omega, \mu)}{B}
\end{equation}
\end{theorem}

Theorem \ref{theorem:universal-property-L1} is something of a blind alley and I know of no serious analytical uses of it as it simply trades one complication for another. Its positive aspects however, are far more important for us. For one, it settles the cocontinuity properties of the bifunctor $\ladditive(\Omega, \mu, B)$, thereby giving us another family of Banach sheaves. More important is its conceptual meaning. The universal property expressed in theorem \ref{theorem:universal-property-L1}, completely analogous to the universal property of $\Integrable_{\infty}(\Omega)$ in theorem \ref{theorem:universal-property-Linfty}, \emph{defines} the space $\Integrable_{1}(\Omega, \mu)$, while the constructions preceding it \emph{prove} that it does indeed exist, while also giving a particularly convenient presentation for it.

\smallskip
Other basic theorems follow equally, with relatively easy proofs. For example, the proof of Fubini's theorem starts with the construction of the coproduct of Boolean algebras $\Omega\otimes \Sigma$ and the identification of the measures on $\Omega\otimes \Sigma$ as certain biadditive functionals. In particular, if we have finitely additive measures $\mu$ and $\nu$ on $\Omega$ and $\Sigma$ respectively, then the map
\begin{equation*}
    \mu\otimes \nu\colon E\otimes F\mapto \mu(E)\nu(F)
\end{equation*}
yields a well-defined measure on $\Omega\otimes \Sigma$. The Fubini theorem\footnote{For $\sigma$-additive measures on $\sigma$-complete Boolean algebras the construction of this coproduct is more involved, because the coproduct $\Omega\otimes \Sigma$ in the category $\Bool$ of Boolean algebras is $\sigma$-complete only in trivial cases.} is now the assertion that there is a natural isometric isomorphism
\begin{equation}\label{isomorphism:fubini}
    \Integrable_{1}(\Omega\otimes \Sigma, \mu\otimes \nu)\isomorphic \Integrable_{1}(\Omega, \mu)\projotimes \Integrable_{1}(\Sigma, \nu)
\end{equation}

Since by theorem \ref{theorem:l1-tensor-product} we have the natural isomorphism $\Integrable_{1}(\Omega, \mu)\projotimes \Integrable_{1}(\Sigma, \nu)\isomorphic \Integrable_{1}(\Omega, \mu, \Integrable_{1}(\Sigma, \nu))$, the \emph{equality of the iterated integrals} follows from the Fubini isomorphism \eqref{isomorphism:fubini}.

\smallskip
At this point, the reader may be justly impatient with our dabbling with formal trivialities\footnote{Judging by the times it is repeated, ``mere formality'' may very well be the favorite expression of \cite{diestel-uhl:vector-measures1977}. The pragmatic distrust of ``mere formalities'' is a natural, even healthy attitude on the part of an analyst. The publishing date of \cite{diestel-uhl:vector-measures1977} is 1977, but the prejudices it echoes are still around and one can still occasionally hear that category theory is ``just a language,'' a ``useful organizational tool'' or some such, usually said in a scornful tone betraying ignorance and misunderstanding. This attitude derives much from the individual perceptions of what counts as deep in mathematics, but even on the face of it, describing category theory as ``just a language,'' as if it were derogatory, is surely bizarre, for isn't language one of the highest functions of the human brain? Language enhances our perceptions and more than expands the reach of our minds. Similarly, even if category theory were just an ``organizational tool,'' it would already serve a very useful purpose much in the same way as a hammer is a mighty fine tool for organizing nails along walls and other similar hard surfaces like blockhead skulls. But category theory is first and foremost, a division of mathematics with its own specific, autonomous field of activity and its own specific set of problems and concerns. A mathematician can spend his whole career without once using any deep result from say, group theory; and yet, were he to decry group theory as a ``mere organizational tool for symmetry,'' such a comment would be met by many of us with an ironical silence in the (probably misguided) notion that there are some points of view about which the less said the better. In a given mathematical field, category theory may not contribute much more than a few trivial observations; in others it may be essential -- witness algebraic geometry and algebraic topology. The evocation of these examples is not without reason, since by the very essence of their disciplines, algebraic geometers and topologists, less afflicted by the narrow myopia of specialization, have to cross the boundaries and smuggle the goods between algebra and geometry, and category theory is just the right language to organize that sort of trafficking.}. In a sense it is an unjust criticism, because a string of trivialities may yield, and does yield, non trivial results. In another sense, the criticism is entirely justified. But it is precisely the trivial quality of these proofs that not only justifies the scaffolding, but it allows to categorify it, which after all, is our purpose all along.

\subsection{The Stone space of a Boolean algebra and weak integrals.}
\label{subsection:stone-space-ba-weak-integrals}

This subsection has two distinct purposes. The first is to elucidate the relation of a Boolean algebra $\Omega$ with its Stone space $\Stone(\Omega)$ and the important consequences this has for measure theory. The second is to fulfill our promise of subsection \ref{subsection:failure-radon-nikodym} of basing Radon-Nikodym type theorems on weak integrals. Strictly speaking, it will not be needed for our subsequent work but it is such a beautiful piece of mathematics that we have not resisted the temptation to include it here. There are two weak integrals, the Pettis integral involving the weak topology and the Gelfand integral involving the $\wkstar$-topology in dual spaces. The Gelfand integral will be the object of interest in this subsection. Both integrals can be constructed in a purely analytical manner with no measure theory involved. For the general construction, we refer the reader to \cite[chapter 3]{rudin:functional-analysis1991}, and for its construction in the context of vector measure theory we refer the reader to \cite[chapter II, section 3]{diestel-uhl:vector-measures1977}. Contrary to the previous subsections, where only the most basic functional analysis was needed, in this subsection sophistication is upped by a few notches and some knowledge of weak topologies in Banach spaces is required. Our basic reference is \cite[chapter V]{conway:course-functional-analysis1990}. In our introduction to the $\wkstar$-integral, we follow a somewhat devious route. There is method to our madness though, as it leads to the description of the dual category of $\Banc$ as the category of Walbr\"{o}eck spaces. We refer the reader to \cite[chapter I, section 2]{cigler-losert-michor:banach-modules-functors-categories-banach-spaces1979} for more information.

%Paragraph:
%- Stone space of a Boolean algebra.

\smallskip
Recall that if $\Omega$ is a Boolean algebra, its Stone space $\Stone(\Omega)$ is the set of (proper) ultrafilters on $\Omega$. These can be identified with Boolean algebra morphisms $\Omega\to \two$ where $\two$ is the Boolean field $\set{0, 1}$. A topology can be introduced in $\Stone(\Omega)$ in the following way: to each $E\in \Omega$ we associate the set of ultrafilters containing $E$, that is:
\begin{equation*}
    \eta(E)\defequal \set{x\in \Stone(\Omega)\colon x(E)= 1}
\end{equation*}

Now, take the collection $\set{\eta(E)\colon E\in \Omega}$ as a base for the topology of $\Stone(\Omega)$. It can then be proved that each $\eta(E)$ is both closed and open (a \emph{clopen}) and that $\Stone(\Omega)$ is a compact Hausdorff, totally disconnected space. If we denote by $\Bool$ the category of Boolean algebras and by $\BoolTop$ the full subcategory of compact Hausdorff, totally disconnected spaces (or \emph{Boolean spaces}), then it is a fundamental result that there is an adjoint equivalence (called \emph{Stone duality}),
\begin{equation} \label{equivalence:stone-duality}
    \homset{\Bool}{\Omega}{\clopen(Y)}\isomorphic \homset{\opposite{\BoolTop}}{\Stone(\Omega)}{Y}
\end{equation}
with unit the Boolean algebra isomorphism $\eta\colon \Omega\to \clopen(\Stone(\Omega))$ given by $E\mapto \eta(E)$.

%Paragraph:
%- consequences of Stone duality.

\smallskip
Stone duality implies that every measure-theoretical concept has a natural topological mirror image in the world of Boolean spaces. The remarkable aspect is that these topological counterparts are usually better behaved than their Boolean algebra siblings. A first indication of this is the identification of the elements of $\Integrable_{\infty}(\Omega)$ with actual \emph{continuous functions}. The crucial theorem in this direction is:

\begin{theorem} \label{theorem:stone-measurable-continuous-map}
The map $\Omega\to \Continuous(\Stone(\Omega))$ given by
\begin{equation*}
    E\mapto \characteristic(\eta(E))
\end{equation*}
is a bounded spectral map. The unique lift to $\Integrable_{\infty}(\Omega)$ as in diagram \ref{diagram:stone-measurable-continuous-map} is an isometric Banach algebra isomorphism.
\begin{figure}[htbp]
    \begin{equation*}
    \xymatrix{
        \Integrable_{\infty}(\Omega) \ar@{-->}[r] & \Continuous(\Stone(\Omega)) \\
        \Omega \ar[u]^{\characteristic} \ar[ur]
    }
    \end{equation*}
\caption{The Banach algebra isomorphism $\Integrable_{\infty}(\Omega)\isomorphic \Continuous(\Stone(\Omega))$.}
\label{diagram:stone-measurable-continuous-map}
\end{figure}
\end{theorem}
\begin{proof}
The map $E\mapto \characteristic(\eta(E))$ is clearly a spectral measure $\Omega\to \Continuous(\Stone(\Omega))$. By theorem \ref{theorem:universal-property-Linfty-algebra}, there is a unique Banach algebra morphism closing triangle \ref{diagram:stone-measurable-continuous-map}. A simple computation proves that this morphism is an isometry. Since $\Stone(\Omega)$ is totally disconnected, the characteristic functions of clopen sets are continuous and separate points. The Stone-Weierstrass theorem (\cite[chapter V, section 8]{conway:course-functional-analysis1990}) now implies that the morphism is bijective.
\end{proof}

In particular, the spectral map $\Omega\to \Continuous(\Stone(\Omega))$ has the same universal property of theorem \ref{theorem:stone-measurable-continuous-map}. This theorem allows us to embed measure theory in topological measure theory and entails a sleuth of important results, starting with the following corollary:

\begin{corollary} \label{corollary:Linfty-continuous}
The functor $\Bool\to \Ban$ given on objects by $\Omega\mapto \Integrable_{\infty}(\Omega)$ is continuous.
\end{corollary}

\smallskip
Let $\Omega$ be a Boolean algebra and $\mu\colon \Omega\to \field$ a bounded additive map. Using the isomorphism unit map $\eta\colon \Omega\to \clopen(\Stone(\Omega))$ we can transfer $\mu$ to the Boolean algebra $\clopen(\Stone(\Omega))$ of clopens of $\Stone(\Omega)$ by:
\begin{equation}\label{equation-definition:transfer-measure-stone-space}
    U\mapto \mu(\inverse{\eta}(U))
\end{equation}

The finitely additive measure \eqref{equation-definition:transfer-measure-stone-space} will also be denoted by $\mu$. Now a miracle happens, as $\mu$ is \emph{$\sigma$-additive on $\clopen(\Stone(\Omega))$}: let $\sequence{U}{n}$ be a countable disjoint sequence of clopens with $U_{n}= \eta(E_{n})$ and
\begin{equation*}
    U= \bigcup_{n}U_{n}
\end{equation*}
with $U$ clopen. Since $U$ is closed, it is compact, and from the countable cover $\sequence{U}{n}$ we can extract a finite subcover $\cover{U}{i}$ with $U= \bigcup_{i}U_{i}$. But since the $U_{n}$ are pairwise disjoint, it follows that only the $U_{i}$ are non-empty and thus
\begin{equation*}
    \mu(U)= \sum_{i}\mu(U_{i})= \sum_{n}\mu(U_{n})
\end{equation*}

Also note that since the Boolean algebra domain of $\mu$ is $\clopen(\Stone(\Omega))$, $\mu$ is automatically \emph{regular}. A peek at the Hahn-Caratheodory extension theorem (see \cite[chapter III, section 13]{halmos:measure-theory1974}) tells us that $\mu$ extends uniquely to a $\sigma$-additive, regular measure on the $\sigma$-algebra generated by $\clopen(\Stone(\Omega))$. This $\sigma$-algebra is the \emph{Baire $\sigma$-algebra} $\Baire(\Stone(\Omega))$ of $\Stone(\Omega)$ and is equal to the $\sigma$-algebra generated by the continuous functions\footnote{The Baire $\sigma$-algebra is the natural $\sigma$-algebra for discussing integration in topological spaces. In general, it is strictly contained in the Borel $\sigma$-algebra $\Borel(X)$ generated by the open sets with equality happening iff $X$ is metrizable. Nevertheless, for $X$ compact Hausdorff but not necessarily metrizable, there is a well-known technique to extend measures from $\Baire(X)$ to $\Borel(X)$ (see \cite[chapter VII]{folland:real-analysis1984}). We will have no need of such an extension. For an exhaustive treatment of the interactions between measures and topologies see \cite{fremlin:measure-theory42003}.}.

\smallskip
If $X$ is a compact Hausdorff space, denote by $\rmeasure(X)$ the space of regular, $\sigma$-additive scalar measures on the Baire $\sigma$-algebra of $X$ with the total variation norm. The previous constructions yield a map
\begin{equation} \label{map:ba-mr-stone-space}
    \badditive(\Omega)\to \rmeasure(\Stone(\Omega))
\end{equation}
and the following theorem\footnote{The theorem readily extends to vector measures with values in Banach spaces. Instead of the Hahn-Caratheodory extension theorem, we simply note that $\mu$ is uniformly continuous on the Boolean algebra $\clopen(\Stone(\Omega))$ with the measure (pseudo)metric and that $\clopen(\Stone(\Omega))$ is dense in $\Baire(\Stone(\Omega))$ for this metric.}.

\begin{theorem} \label{theorem:ba-mr-stone-space}
The map \eqref{map:ba-mr-stone-space} is a natural isometric isomorphism.
\end{theorem}

Theorem \ref{theorem:ba-mr-stone-space} tells us that if we grind measures through the Stone space machinery, they come out improved on the other side. The generalization of $\sigma$-additive measures to finitely additive ones is not much of a generalization. We ask the reader to have these facts in mind as they will make understandable analogous phenomena in the categorified setting to be uncovered in subsection \ref{subsection:banach-sheaves}.

\smallskip
Next, we discuss an important adjunction that gives an inkling on how to represent every bounded linear map $\Integrable_{1}(\Omega, \mu)\to B$ by a \emph{weak integral}. The basic idea is that theorem \ref{theorem:stone-measurable-continuous-map} encourages to view ``bounded measurable functions on $\Omega$'' as \emph{continuous} functions on the Stone space $\Stone(\Omega)$. But given a Banach space $B$, there are weaker topologies at our disposal giving us weaker notions of measurability. These notions were already used in the Pettis theorem \ref{theorem:pettis} where weak measurability was found to be as good as Borel measurability in characterizing the strongly measurable functions. In what follows, we will consider $\wkstar$-measurability of functions $f$ with values in dual Banach spaces. For domain of $f$ we will take a general compact Hausdorff space $X$, but the reader can mentally replace it by $\Stone(\Omega)$ and consider them as $\wkstar$-measurable functions.

\smallskip
Let $X$ be a compact Hausdorff space. Then a point $x\in X$ defines a linear functional $\delta_{x}\colon \Continuous(X)\to \field$ by evaluation:
\begin{equation*}
    \delta_{x}\colon f\mapto f(x)
\end{equation*}

The $\delta_{x}$ are the Dirac measures with singleton support. For example, if $X$ is the Stone space $\Stone(\Omega)$ of a Boolean algebra $\Omega$, then by the isomorphism \eqref{map:ba-mr-stone-space}, the $\delta_{x}$ can be identified with the ultrafilters $x$ of $\Omega$. In fact, they are even \emph{spectral measures}. The next theorem contains two easily proved facts about $\delta_{x}$.

\begin{theorem} \label{theorem:properties-delta-wkstar-map}
Let $X$ be a compact Hausdorff space and $\Delta\colon X\to \dual{\Continuous(X)}$ the map $x\mapto \delta_{x}$. Then:
\begin{enumerate}
  \item \label{prop-enum:properties-delta-wkstar-map1}
  The set $\set{\delta_{x}\colon x\in X}$ is linearly independent in $\dual{\Continuous(X)}$.

  \item \label{prop-enum:properties-delta-wkstar-map2}
  The map $\Delta$ is $\wkstar$-continuous.
\end{enumerate}
\end{theorem}

It is not difficult to see that each $\delta_{x}$ is an extreme point of $\ball(\dual{\Continuous(X)})$. With a little more work, we can even characterize these extreme points.

\begin{theorem} \label{theorem:extreme-points-unit-ball-measures}
A point $\mu\in \ball(\rmeasure(X))$ is an extreme point iff $\mu = \epsilon\delta_{x}$ with $\epsilon$ an element of the unit sphere of $\field$.
\end{theorem}
\begin{proof}
It is more or less clear that measures of the form $\epsilon\delta_{x}$ are extreme points. That all extreme points are of this form, follows from a clever application of Urysohn's lemma to prove that every extreme point must have singleton support.
\end{proof}

There is a universal property associated to the map $\Delta\colon X\to \dual{\Continuous(X)}$ with various nice consequences such as giving a representation of bounded linear maps $B\to \Continuous(X)$ as certain \emph{$\wkstar$-integrals}.

%Paragraph:
%- the space \wkstarcontinuous.

\smallskip
Let $X$ be a compact Hausdorff space and $\dual{B}$ a dual Banach space. Since the $\wkstar$-topology of $\dual{B}$ is the topology generated by the evaluation linear functionals $b^{\ast}\mapto \inner{b^{\ast}}{b}$ for $b\in B$, it follows that $f\colon X\to \dual{B}$ is $\wkstar$-continuous iff $x\mapto \inner{f(x)}{b}$ is continuous for every $b\in B$. Let us denote by $\wkstarContinuous(X, \dual{B})$ the linear space of such functions with the norm
\begin{equation} \label{def-equation:wkstar-supnorm}
    \supnorm{f}\defequal \sup\set{\supnorm{\inner{f(x)}{b}}\colon b\in \ball(B)}
\end{equation}

Since a subset of $\dual{B}$ is bounded iff it is $\wkstar$-bounded, an interchange of $\sup$'s in \eqref{def-equation:wkstar-supnorm} yields that
\begin{equation*}
    \supnorm{f}= \sup\set{\norm{f(x)}\colon x\in X}
\end{equation*}

The following theorem now follows easily.

\begin{theorem} \label{theorem:wkstarcontinuous-complete}
The normed space $\wkstarContinuous(X, \dual{B})$ is complete.
\end{theorem}

Next we show that the map $\Delta\colon X\to \dual{\Continuous(X)}$ is universal among $\wkstar$-continuous maps in the sense that if $f\colon X\to \dual{B}$ is a $\wkstar$-continuous map then there is a unique bounded linear map $\universal{f}\colon \dual{\Continuous(X)}\to \dual{B}$ such that the triangle \ref{diagram:induced-map-baire-regular-measures} is commutative.
\begin{figure}[htbp]
    \begin{equation*}
    \xymatrix{
        \rmeasure(X) \ar@{-->}[r]^{\universal{f}} & \dual{B} \\
        X \ar[u]^{\Delta} \ar[ur]_{f} &
    }
    \end{equation*}
\caption{The induced map $\rmeasure(X)\to \dual{B}$.}
\label{diagram:induced-map-baire-regular-measures}
\end{figure}

Denote by $V$ the linear span of $\set{\delta_{x}}$ in $\dual{\Continuous(X)}$. The function $f$ lifts to $\universal{f}\colon V\to \dual{B}$ by \eqref{prop-enum:properties-delta-wkstar-map1} of theorem \ref{theorem:properties-delta-wkstar-map} and putting for a linear combination $\sum_{n}k_{n}\delta_{x_{n}}$ of $\delta_{x}$,
\begin{equation} \label{equation:action-wkstar-integral-linear-span-deltas}
    \universal{f}(\mu)\colon b\mapto \integral{X}{\inner{f(x)}{b}}{\mu}= \inner{\sum_{n}k_{n}f(x_{n})}{b}
\end{equation}

Combining the Banach-Alaoglu theorem (\cite[chapter V, section 3]{conway:course-functional-analysis1990}) with the Krein-Milman theorem (\cite[chapter V, section 7]{conway:course-functional-analysis1990}), it follows that the linear span of $\set{\delta_{x}}$ is $\wkstar$-dense in $\rmeasure(X)$, therefore there is at most one $\wkstar$-continuous linear extension of $f$ closing the triangle \ref{diagram:induced-map-baire-regular-measures}. But taking the hint from \eqref{equation:action-wkstar-integral-linear-span-deltas}, we define $\universal{f}$ on the whole space $\rmeasure(X)$ as
\begin{equation} \label{def-equation:wkstar-integral}
    \universal{f}(\mu) \colon b\mapto \integral{X}{\inner{f(x)}{b}}{\mu}
\end{equation}

It is natural to denote the element $\universal{f}(\mu)$ of $\dual{B}$ by $\integral{X}{f}{\mu}$ and view it as the \emph{$\wkstar$-integral of $f$}. In inner product notation, definition \eqref{def-equation:wkstar-integral} can be formulated as
\begin{equation*}
    \inner{\integral{X}{f}{\mu}}{b}= \integral{X}{\inner{f(x)}{b}}{\mu}, \forall b\in B
\end{equation*}
and the map $\universal{f}$ is simply $\mu\mapto \integral{X}{\inner{f(x)}{b}}{\mu}$.

\smallskip
As noted above, this map is $\wkstar$-continuous. By the Banach-Alaoglu theorem, the image under this map of the unit ball of $\dual{B}$ is $\wkstar$-compact and therefore \eqref{def-equation:wkstar-integral} is actually a bounded linear map on $\rmeasure(X)$. In fact it can be proved by a combination of \eqref{def-equation:wkstar-integral} with the Hahn-Banach theorem that we have,
\begin{equation*}
    \sup\set{\norm{\integral{X}{f}{\mu}}\colon \mu\in \ball(\rmeasure(X))}= \supnorm{f}
\end{equation*}

This means that the map $f\mapto \universal{f}$ is a natural isometric embedding
\begin{equation*}
    \wkstarContinuous(X, \dual{B})\to \homset{\Ban}{\rmeasure(X)}{\dual{B}}
\end{equation*}

The image of this map is the space of $\wkstar$-continuous operators and it is well known that these are precisely the adjoints of bounded linear maps $B\to \Continuous(X)$. In this case, the preadjoint can be computed with the aid of formulas \eqref{equation:naturality-integral-map}. Start by noting that the function $x\mapto \inner{f(x)}{b}$ is the function $\eta(b)f$ with $\eta$ the embedding $B\to \bidual{B}$ in the bidual. Now, we have
\begin{equation*}
\begin{split}
    \inner{\integral{X}{f}{\mu}}{b} &= \integral{X}{\eta(b)f}{\mu} \\
        &= \inner{\eta(b)f}{\mu}
\end{split}
\end{equation*}

Thus, the preadjoint of $\mu\mapto \integral{X}{f}{\mu}$ is $b\mapto \eta(b)f$. Combining the results up to now, we can state the following theorem.

\begin{theorem} \label{theorem:continuous-dual-unit-ball}
The map $f\mapto \parens{b\mapto \eta(b)f}$ establishes a natural isometric isomorphism
\begin{equation} \label{isomorphism:continuous-dual-unit-ball}
    \wkstarContinuous(X, \dual{B})\isomorphic \homset{\Ban}{B}{\Continuous(X)}
\end{equation}
\end{theorem}

Theorem \ref{theorem:continuous-dual-unit-ball} can be turned into an adjunction by taking unit balls and passing to the dual category. If we denote by $\CHaus$ the category of compact Hausdorff spaces, then isomorphism \eqref{isomorphism:continuous-dual-unit-ball} descends to an adjunction
\begin{equation} \label{adjunction:continuous-dual-unit-ball}
    \homset{\CHaus}{X}{\ball(\dual{B})}\isomorphic \homset{\opposite{\Banc}}{\Continuous(X)}{B}
\end{equation}
where $\ball(\dual{B})$ is the unit ball of $\dual{B}$ endowed with the $\wkstar$-topology.

\smallskip
Theorem \ref{theorem:continuous-dual-unit-ball} gives a complete description of the bounded linear maps $B\to \Continuous(X)$. We refer the reader to \cite[chapter VI, section 7]{dunford-schwartz:linear-operatorsI1958} for these results. It can also serve as a first stepping stone to prove some Radon-Nikodym theorems.

\section{Banach \texorpdfstring{$2$}{2}-spaces}
\label{section:banach-2spaces}

Since categorified measure theory needs a categorified analogue of a Banach space, we return to the ideas adumbrated in subsection \ref{subsection:hilb-categorified-ring}. We will be very brief, since this will be treated more fully in a future paper \cite{rodrigues:banach-2spaces}.

\begin{notation}
From now on we denote the projective tensor product of Banach spaces by $A\otimes B$ instead of $A\projotimes B$. This will guarantee more symmetric formulas.
\end{notation}

\renewcommand{\projotimes}{\otimes}

We first expand table \ref{table:categorified-analogues-ring-operations} in order to include the basic Banach space categorified analogues. This is done in table \ref{table:categorified-analogues-banach-space}.
\begin{table}[htbp]
  \begin{tabular}{|c|c|}
    \hline
    Ordinary Banach spaces & Categorified Banach spaces \\
    \hline
    Base field $\field$ & Base ring $\Ban$ \\
    Sum $+$ & Coproduct $\oplus$ \\
    Difference $-$ & Cokernels of maps \\
    Additive zero $0$ & Zero object $\zero$ \\
    Scalar multiplication $\times$ & Tensors $\otimes$ \\
    Cauchy completeness & Existence of filtered colimits \\
    \hline
  \end{tabular}
  \smallskip
  \caption{Categorified analogues of Banach space theory}
  \label{table:categorified-analogues-banach-space}
\end{table}

The only addition to the categorified analogues of linear algebra presented in \cite{baez:hda2-2-hilbert-spaces1997} is the last row. The need for infinitary colimits was expressed in principles \ref{principle:category-measurable-bundles} and \ref{principle:infinitary-colimits-categorified-convergence}, and given the similarities between the formal properties of filtered colimits and those of convergence of nets in topological spaces, the last row of \ref{table:categorified-analogues-banach-space} is at least plausible. All the categorified analogues in table \ref{table:categorified-analogues-banach-space} are \emph{colimits} and defined by suitable universal properties. In fact, all the conditions put together amount to require cocompleteness of a Banach category, so we register the following definition\footnote{There is a weaker notion of completeness for enriched categories called \emph{Cauchy completeness}. The work of section \ref{section:towards-categorified-measure-theory} has shown that we need infinitary colimits but Cauchy completeness will make its appearance later on.}.

\begin{definition} \label{definition:banach-2spaces}
A \emph{Banach $2$-space} is a cocomplete Banach category\footnote{The ``Banach'' qualifier is being used in a different sense in the two nouns. In \emph{Banach category} is used to mean relative or enriched in Banach spaces, while in \emph{Banach $2$-space} it is used in the sense of cocompleteness. It is an unfortunate consequence of our terminology that the term \emph{complete Banach $2$-space} is \emph{not} a redundancy.}.
\end{definition}

Given definition \ref{definition:banach-2spaces} for the categorified analogue of a Banach space, it is clear what is the right notion of morphism between Banach $2$-spaces: a cocontinuous functor $T\colon \mathcal{A}\to \mathcal{B}$. If $S, T\colon \mathcal{A}\to \mathcal{B}$ are cocontinuous functors, a morphism $S\to T$ is simply a contractive natural transformation. In this way, we obtain the $2$-category $\BanModCat$ of Banach $2$-spaces.

%Paragraph:
%- comments on definition of Banach 2-space.

\smallskip
Some comments on these definitions are in order. A first question that can arise in the reader's mind is what is the categorified analogue of the metric. The answer, which we happily embrace, is given on \cite{lawvere:metric-spaces-generalized-logic-closed-categories2002}. The metric, or the object that ``measures the distance between two objects $a$ and $b$'', is the Banach space $\homset{\mathcal{A}}{a}{b}$ of morphisms $a\to b$. Perceptive readers will also notice the conspicuous absence of any completeness conditions in definition \ref{definition:banach-2spaces}. Eventually, we will need them, e.g.~when we discuss sheaves in subsection \ref{subsection:banach-sheaves}. Their presence also insures that the dual category $\opposite{\mathcal{A}}$ is also a Banach $2$-space. In many instances one can prove directly that a given Banach $2$-space is complete; this is the case with all the examples of this paper. One should view definition \ref{definition:banach-2spaces} as a minimal set of properties a Banach $2$-space must have; further developments may dictate stronger requirements. Requiring limits also makes some constructions with Banach $2$-spaces more difficult (for a hint of this, see the remarks at the end of subsection \ref{subsection:presheaf-categories}). A related question is if cocontinuous functor is the right notion of morphism between Banach $2$-spaces. The adjoint functor theorem guarantees that in most cases cocontinuous functors will have right adjoints\footnote{A particularly striking and easy to remember form of the adjoint functor that applies to many, if not all of our examples of cocontinuous functors, can be found in \cite[chapter 5, section 6, theorem 33]{kelly:enriched-category-theory2005}: every cocontinuous functor defined on a cocomplete category with a small dense subcategory has a right adjoint.}, so taking the hint from topos theory and its geometric morphisms, we are inclined to take adjunctions as the right notion of morphism between Banach $2$-spaces.

\smallskip
In fact as hinted earlier, there is an (uneasy) relationship between categorified measure theory and topos theory. This relation will be more and more agitated as we proceed, until by the end of the next section it will have mounted to an almost obsessive pattern. There is of course, a lesson to be learned here and one cannot help but wonder if there is not a notion of ``Banach topos'' lurking behind all the analogies. I do not know the answer to this question, but I refer the reader to \cite{street:cauchy-characterization-enriched-categories2004}, especially its introductory remarks.

\smallskip
As mentioned in the beginning we will not develop the theory of Banach $2$-spaces. But we will need two simple, related constructions on $\BanModCat$. If $\mathcal{A}$ and $\mathcal{B}$ are two Banach $2$-spaces then the Banach category $\homset{\BanModCat}{\mathcal{A}}{\mathcal{B}}$ is cocomplete with colimits computed pointwise. The problem is that without size restrictions on $\mathcal{A}$, $\homset{\BanModCat}{\mathcal{A}}{\mathcal{B}}$ does not have to be locally small and thus not a Banach category according to our definitions. For much the same reasons, the \emph{bitensor product} $\mathcal{A}\bitensor \mathcal{B}$ of Banach $2$-spaces may not exist. Recall that a bifunctor $\mathcal{A}\otimes \mathcal{B}\to \mathcal{C}$ is \emph{bicocontinuous} if it is cocontinuous in each variable. The bitensor product is then a bicocontinuous bifunctor $\mathcal{A}\otimes \mathcal{B}\to \mathcal{A}\bitensor \mathcal{B}$ universal among all such bifunctors, that is, each bicocontinuous $\mathcal{A}\otimes \mathcal{B}\to \mathcal{C}$ factors through it via a cocontinuous functor as in diagram \ref{diagram:universal-property-bitensor}, and this factorization is unique up to unique canonical isomorphism.
\begin{figure}[htbp]
    \begin{equation*}
    \xymatrix{
        \mathcal{A}\bitensor \mathcal{B} \ar@{-->}[r] \ar@{}[dr]|<<<<<{\isomorphic} & \mathcal{C} \\
        \mathcal{A}\otimes \mathcal{B} \ar[u] \ar[ur]_{\varphi} &
    }
    \end{equation*}
\caption{Universal property of $\mathcal{A}\otimes \mathcal{B}\to \mathcal{A}\bitensor \mathcal{B}$.}
\label{diagram:universal-property-bitensor}
\end{figure}

It is more or less clear that if the bitensor product exists, then we have an equivalence
\begin{equation}\label{equivalence:banmodcat-biclosed}
    \homset{\BanModCat}{\mathcal{A}\bitensor \mathcal{B}}{\mathcal{C}}\equivalent \homset{\BanModCat}{\mathcal{A}}{\homset{\BanModCat}{\mathcal{B}}{\mathcal{C}}}
\end{equation}

\subsection{Presheaf categories}
\label{subsection:presheaf-categories}

In principle \ref{principle:integral-functor-left-Kan-extension} it was stated that the integral functor is a left Kan extension. Since left Kan extensions are essential for everything that follows, we devote this subsection to review them. This will lead to the fundamental concepts of \emph{dense functor}, which can be seen as a categorified analogue of topological density, and \emph{(small) projective object}. We will use \cite{kelly:enriched-category-theory2005} as our main reference for enriched category theory. Chapter 6 of \cite{borceux:handbook-categorical-algebra21994} is another profitable source.

%Paragraph:
%- the free 2-adjunction.

\smallskip
Before the definition of left Kan extensions we take a step back and recall the construction of \emph{free Banach categories}. An important special case of the isometric isomorphism \eqref{isomorphism:universal-property-injections} is obtained by letting $B_{x}= \field$ for every $x\in X$. Denoting $\ell_{1}(X, \field)$ simply by $\ell_{1}(X)$ then \eqref{isomorphism:universal-property-injections} specializes to
\begin{equation} \label{isomorphism:free-banach-space}
    \homset{\Ban}{\ell_{1}(X)}{A}\isomorphic \Bounded(X, A)
\end{equation}

The isometric isomorphism \eqref{isomorphism:free-banach-space} is easily seen to be natural in $X\in\Set$, so that we can view $\ell_{1}(X)$ as a sort of free Banach space. In fact, $X\mapto \ell_{1}(X)$ can be made into a left adjoint. There is a functor $\ball\colon \Banc\to \Set$, the \emph{unit ball functor}, that to a Banach space $B$ associates its unit ball $\ball(B)$. Isomorphism \eqref{isomorphism:free-banach-space} now reads as
\begin{equation} \label{isomorphism:unit-ball-adjunction}
    \homset{\Banc}{\ell_{1}(X)}{A}\isomorphic \homset{\Set}{X}{\ball(A)}
\end{equation}
which is true because a function $X\to \ball(A)$ induces via the isomorphism \eqref{isomorphism:free-banach-space} a unique linear contraction $\ell_{1}(X)\to A$. All told:

\begin{theorem} \label{theorem:free-banach-space}
The unit ball functor $\ball\colon\Banc\to \Set$ has a left adjoint\footnote{The unit ball functor is a right adjoint, faithful and reflects isomorphisms, but it is not monadic. For the proof, see \cite[chapter 4, section 3]{barr-wells:toposes-triples-theories1983}.} that on objects is given by $X\mapto \ell_{1}(X)$.
\end{theorem}

Coupling the projective tensor product with the $\ell_{1}$ free functor yields,
\begin{align*}
\homset{\Banc}{\ell_{1}(X)\projotimes \ell_{1}(Y)}{A}
    &\isomorphic \homset{\Banc}{\ell_{1}(X)}{\homset{\Banc}{\ell_{1}(Y)}{A}} \\
    &\isomorphic \homset{\Banc}{\ell_{1}(X)}{\homset{\Set}{Y}{\ball(A)}} \\
    &\isomorphic \homset{\Set}{X}{\homset{\Set}{Y}{\ball(A)}} \\
    &\isomorphic \homset{\Set}{X\times Y}{\ball(A)} \\
    &\isomorphic \homset{\Banc}{\ell_{1}(X\times Y)}{A}
\end{align*}
and it follows from Yoneda's lemma that it determines a natural isometric isomorphism,
\begin{equation} \label{isomorphism:l1-product}
    \ell_{1}(X\times Y)\isomorphic \ell_{1}(X)\projotimes \ell_{1}(Y)
\end{equation}
that can be seen as \emph{Fubini's theorem for absolutely summable families}. It also means that $\ell_{1}$ is a symmetric monoidal functor when $\Set$ is given the cartesian closed symmetric monoidal structure $\pair{X}{Y}\mapto X\times Y$.

\smallskip
The fact that $X\mapto \ell_{1}(X)$ is symmetric monoidal implies that the adjunction \eqref{isomorphism:unit-ball-adjunction} lifts to a $2$-adjunction between the $2$-categories $\BanCat$ and $\Cat$. The right adjoint part is the forgetful $2$-functor $\BanCat\to \Cat$ that associates to a Banach category $\mathcal{A}$ its \emph{unit ball category} $\ball(\mathcal{A})$, which is the subcategory of contractions of $\mathcal{A}$. In symbols:
\begin{equation}\label{definition:unit-ball-category}
    \ball(\mathcal{A})\defequal \set{f\in \mathcal{A}\colon \norm{f}\leq 1}
\end{equation}

To construct the left adjoint of the unit ball functor, let $\mathcal{X}$ be a category and denote in matrix form $\homsetmatrix{\mathcal{X}}{x}{y}$ the set of morphisms $x\to y$. The Banach category $\ell_{1}(\mathcal{X})$ has for objects the objects of $\mathcal{X}$. The Banach space $\homset{\ell_{1}(\mathcal{X})}{x}{y}$ of morphisms $x\to y$ is the free Banach space $\ell_{1}(\homsetmatrix{\mathcal{X}}{x}{y})$. The unit map is the map
\begin{equation*}
    \field\to \ell_{1}(\homsetmatrix{\mathcal{X}}{x}{x})
\end{equation*}
given by $\id_{x}\in \homsetmatrix{\mathcal{X}}{x}{x}$ and the composition map
\begin{equation*}
    \ell_{1}(\homsetmatrix{\mathcal{X}}{x}{y})\otimes \ell_{1}(\homsetmatrix{\mathcal{X}}{y}{z})\to \ell_{1}(\homsetmatrix{\mathcal{X}}{x}{z})
\end{equation*}
is the composite,
\begin{equation*}
    \xymatrix{
        \ell_{1}(\homsetmatrix{\mathcal{X}}{x}{y})\otimes \ell_{1}(\homsetmatrix{\mathcal{X}}{y}{z}) \ar[r]^{\isomorphic} & \ell_{1}(\homsetmatrix{\mathcal{X}}{x}{y}\times \homsetmatrix{\mathcal{X}}{y}{z}) \ar[rr]^-{\ell_{1}(\composition_{x, y, z})} & & \ell_{1}(\homsetmatrix{\mathcal{X}}{x}{z})
    }
\end{equation*}
where $\composition_{x, y, z}$ is the composition map $\homsetmatrix{\mathcal{X}}{x}{y}\times \homsetmatrix{\mathcal{X}}{y}{z}\to \homsetmatrix{\mathcal{X}}{x}{z}$. Clearly, both the unit and composition maps are contractions and routine diagrammatic computations which we leave to the reader would prove the associative and unital laws. Putting it all together, we have:

\begin{theorem} \label{theorem:2adjunction-ball-ell}
The $2$-functor $\mathcal{X}\mapto \ell_{1}(\mathcal{X})$ establishes a natural isomorphism
\begin{equation} \label{isomorphism:2adjunction-ball-ell}
    \homset{\BanCat}{\ell_{1}(\mathcal{X})}{\mathcal{A}}\isomorphic \homset{\Cat}{\mathcal{X}}{\ball(\mathcal{A})}
\end{equation}
\end{theorem}

%Paragraph:
%- notational simplification.

\begin{notation}
In what follows we will make \emph{no notational distinction} between a category $\mathcal{X}$ and the free Banach category $\ell_{1}(\mathcal{X})$. In other words, in the enriched context an ordinary category $\mathcal{X}$ is automatically promoted to the free Banach category $\ell_{1}(\mathcal{X})$. Likewise, a contractive functor $\mathcal{X}\to \mathcal{A}$ into a Banach category will be identified with the induced contractive functor $\ell_{1}(\mathcal{X})\to \mathcal{A}$.
\end{notation}

%Paragraph:
%- Kan extensions.

\smallskip
A functor $F$ defined on a subcategory $\mathcal{M}\subseteq \mathcal{A}$ can have from none to many extensions to the whole category. Among all the possible extensions of $F$ to the whole category $\mathcal{A}$ we can single out a maximal and a minimal one and these can be characterized by simple universal properties.

\begin{definition} \label{definition:left-kan-extension}
Let $I\colon \mathcal{M}\to \mathcal{A}$ and $F\colon \mathcal{M}\to \mathcal{B}$ be two functors, with $\mathcal{B}$ a Banach $2$-space. The \emph{left Kan extension of $F$ along $I$} is a pair $\pair{L}{\eta}$ with $L\colon \mathcal{A}\to \mathcal{B}$ and $\eta\colon F\to LI$ universal as an arrow from $\contravariant{I}\colon \homset{\BanCat}{\mathcal{A}}{\mathcal{B}}\to \homset{\BanCat}{\mathcal{M}}{\mathcal{B}}$ to $F$ -- see diagram \ref{2diagram:left-Kan-extension}.
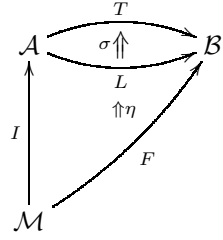
\begin{figure}[htbp]
    \begin{equation*}
    \xymatrix{
        \mathcal{A} \rrtwocell_{L}^{T}{^{\sigma}} &  & \mathcal{B} \\
            & \ar@{}[u]|<<<{\Uparrow\eta} & \\
        \mathcal{M} \ar[uu]^{I} \ar@/_/[uurr]_{F} & &
    }
    \end{equation*}
\caption{$\pair{L}{\eta}$ as a left Kan extension.}
\label{2diagram:left-Kan-extension}
\end{figure}
\end{definition}

Definition \ref{definition:left-kan-extension} of a left Kan extension is the one that can be found in \cite[chapter X, section 3]{maclane:categories-working-mathematician1971} for ordinary categories, but for reasons that need not concern us here it is not entirely satisfactory in the general enriched context. The cocompleteness hypothesis on the codomain category $\mathcal{B}$ tacked on the definition forces (small) colimits to be \emph{pointwise colimits} and possibilitates the simpler definition \ref{definition:left-kan-extension}.

\smallskip
We will denote the pair $\pair{L}{\eta}$ simply by $\Lan_{I}F$, leaving the universal cone $\eta$ implicit. By universality, the assignment $\sigma\mapto \sigma_{I}\eta$ establishes an isometric isomorphism
\begin{equation} \label{isomorphism:left-kan-extension}
    \homset{\BanCat}{\Lan_{I}F}{T}\isomorphic \homset{\BanCat}{F}{TI}
\end{equation}

In particular, if $\Lan_{I}F$ exists for every $F$ then the functor $F\mapto \Lan_{I}F$ is left adjoint to $T\mapto TI$ and therefore, cocontinuous.

\smallskip
The calculus of coends yields a very beautiful formula for $\Lan_{I}F$ and simple sufficient conditions for its existence. Coends are colimits over bifunctors that exhibit behavior analogous to integrals -- we refer the reader to \cite[chapter IX]{maclane:categories-working-mathematician1971} besides the already cited references for enriched category theory where (co)ends play no less than an essential role.

\begin{theorem} \label{theorem:left-kan-extension-coend}
Let $I\colon \mathcal{M}\to \mathcal{A}$ and $F\colon \mathcal{M}\to \mathcal{B}$ be two functors, with $\mathcal{B}$ a Banach $2$-space. Then the left Kan extension $\Lan_{I}F$ exists and is given by,
\begin{equation}\label{isomorphism:left-kan-extension-coend}
    \Lan_{I}F(a)\isomorphic \int^{m\in\mathcal{M}}\homset{\mathcal{A}}{Im}{a}\otimes Fm
\end{equation}
on the hypothesis that for every $a\in \mathcal{A}$ the coend on the right hand side exists. Furthermore, the universal map $\eta\colon F\to {\Lan_{I}F}I$ is given by the composite in diagram \ref{diagram:universal-map-left-Kan-extension}.
\begin{figure}[htbp]
    \begin{equation*}
    \xymatrix{
        Fn \ar[r] \ar[dr]_{\eta_{n}} & \homset{\mathcal{A}}{In}{In}\otimes Fn \ar[d]^{w_{n, In}} \\
            & \int^{m\in \mathcal{M}}\homset{\mathcal{A}}{Im}{In}\otimes Fm
    }
    \end{equation*}
\caption{The universal map $\eta$.}
\label{diagram:universal-map-left-Kan-extension}
\end{figure}

The vertical arrow $w_{n, In}$ in \ref{diagram:universal-map-left-Kan-extension} is the colimiting wedge for the coend and the top arrow is the composite
\begin{equation*}
    \xymatrix{
        Fn \ar[r]^-{\isomorphic} & \field\otimes Fn \ar[rr]^-{i_{\id_{Fn}}\otimes \id_{Fn}} & & \homset{\mathcal{A}}{In}{In}\otimes Fn
    }
\end{equation*}
\end{theorem}
\begin{proof}
The proof in \cite[chapter X, section 4]{maclane:categories-working-mathematician1971} can be adapted in a straightforward fashion to the $\Banc$-enriched case.
\end{proof}

As an immediate corollary, if $\mathcal{M}$ is small the coends in \eqref{isomorphism:left-kan-extension-coend} exist and therefore so does every left Kan extension $\Lan_{I}F$. Another corollary is that if $I$ is fully-faithful then
\begin{equation*}
    (\Lan_{I}F)In\isomorphic \int^{m\in\mathcal{M}}\homset{\mathcal{A}}{Im}{In}\otimes Fm\isomorphic
    \int^{m\in\mathcal{M}}\homset{\mathcal{M}}{m}{n}\otimes Fm
\end{equation*}
but the second coend is just $Fn$ by Yoneda. Thus, $\eta$ induces an isometric isomorphism $\parens{\Lan_{I}F}I\isomorphic F$ so that $\Lan_{I}F$ is indeed an extension of $F$.

\smallskip
Two questions present themselves naturally. First, when is the extension essentially unique and second, when is the extension $\Lan_{I}F$ cocontinuous as a functor on $\mathcal{A}$. The second question can be answered almost immediately. Looking at the coend formula
\eqref{isomorphism:left-kan-extension-coend} for a left Kan extension, we see that if the functor $a\mapto \homset{\mathcal{A}}{Im}{a}$ is cocontinuous, then the interchange of colimits theorem (\cite[chapter 3, section 3]{kelly:enriched-category-theory2005}) guarantees that $\Lan_{I}F$ is cocontinuous. Since the codomain $\mathcal{B}$ is cocomplete, colimits of functors are computed pointwise and thus the same interchange of colimits theorem yields that this condition is not only sufficient but actually necessary. This leads us to the next important concept.

%Paragraph:
%- projective object.

\begin{definition} \label{definition:projective-objects}
Let $\mathcal{M}$ be a small Banach category, $\mathcal{A}$ a Banach $2$-space and $I\colon \mathcal{M}\to \mathcal{A}$ a functor. An object $m\in \mathcal{M}$ is \emph{small $I$-projective} if the functor
\begin{equation*}
    a\mapto \homset{\mathcal{A}}{Im}{a}
\end{equation*}
is cocontinuous. If $I$ is the inclusion functor we just say that $m$ is \emph{small projective}.
\end{definition}

The notion of projective object comes from homological algebra (see \cite[chapter 2, section 2]{weibel:introduction-homological-algebra1994}). The definition makes sense in any category: an object $a\in \mathcal{A}$ is \emph{projective} if every arrow $f\colon a\to b$ factors through every epi $e\colon c\to b$ as in diagram \ref{diagram:projective-object}.
\begin{figure}[htbp]
    \begin{equation*}
    \xymatrix{
          & c \ar[d]^{e} \\
        a \ar[r]_{f} \ar@{-->}[ur] & b
    }
    \end{equation*}
\caption{Lifting property of projective objects.}
\label{diagram:projective-object}
\end{figure}

This is equivalent to require that if $e$ is epi then $\covariant{e}$ is epi. In the context of abelian categories (= homological algebra) this is equivalent to the fact that the functor $b\mapto \homset{\mathcal{A}}{a}{b}$ is \emph{finitely cocontinuous}. Definition \ref{definition:projective-objects} is a considerable strengthening in that we require full cocontinuity of the representable, but since we will consider no other notion of projective object, we drop the ``small'' qualifier.

\begin{example} \label{example:representables-projective-presheaves}
Let $\mathcal{A}$ be a small Banach category. By Yoneda's lemma, we have the isometric isomorphism,
\begin{equation*}
    \homset{\BanCat}{\Yoneda^{a}}{F}\isomorphic F(a)
\end{equation*}
which implies that the functor $F\mapto \homset{\BanCat}{\Yoneda^{a}}{F}$ is isomorphic to the evaluation functor $\evaluation_{a}\colon F\mapto F(a)$. Since $\Ban$ is cocomplete and colimits in functor categories are computed pointwise, the evaluation functor is cocontinuous and thus every representable is projective in $\homset{\BanCat}{\opposite{\mathcal{A}}}{\Ban}$.
\end{example}

%Paragraph:
%- density concept.

In the context of Banach spaces, the extension of a bounded linear map $M\to B$ with $M\subseteq A$ is unique when $M$ is a dense linear subspace of $A$. In the context of enriched categories, this leads us to the concept of \emph{dense functor}.

\begin{definition} \label{definition:dense-functor}
A functor $I\colon \mathcal{M}\to \mathcal{A}$ is \emph{dense} if the identity $\id_{I}\colon I\to \id_{\mathcal{A}}I$ exhibits $\id_{\mathcal{A}}$ as the left Kan extension of $I$ along $I$. If $I$ is the inclusion of a subcategory $\mathcal{M}$ we just say that $\mathcal{M}$ is dense in $\mathcal{A}$.
\end{definition}

Recall that if $T\colon \mathcal{A}\to \mathcal{B}$ is a functor, then we have natural transformations $T_{a, b}\colon \homset{\mathcal{A}}{a}{b}\to \homset{\mathcal{B}}{Ta}{Tb}$ given by $f\mapto Tf$. The concreteness of the base category $\Banc$ makes the proof of naturality a triviality, for the general enriched case follow the instructions in \cite[chapter 1, section 8]{kelly:enriched-category-theory2005}. The existence of tensors in $\mathcal{B}$ turns these natural transformations into wedges,
\begin{equation*}
    T_{a, b}\colon \homset{\mathcal{A}}{a}{b}\otimes Ta\to Tb
\end{equation*}
and the strong form of Yoneda's lemma (\cite[chapter 2, section 4]{kelly:enriched-category-theory2005}) implies that these wedges are universal, that is, they induce an isomorphism
\begin{equation}\label{isomorphism:generalized-yoneda-coend-form}
    Tb\isomorphic \int^{a\in \mathcal{A}}\homset{\mathcal{A}}{a}{b}\otimes Ta
\end{equation}

The case of $T$ the identity functor was used above in deriving the isomorphism $\parens{\Lan_{I}F}I\isomorphic F$. In this case also, the universal wedge gives us a canonical map,
\begin{equation*}
    \evaluation_{a,b}\colon \homset{\mathcal{A}}{a}{b}\otimes a\to b
\end{equation*}
which can be seen as a sort of internal version of the evaluation map (and thus the notation). The coend formula for a left Kan extension now means that $I$ is dense iff the canonical map $\evaluation_{Im, b}$ is a universal wedge with,
\begin{equation} \label{isomorphism:density-presentation-identity}
    a\isomorphic \int^{m\in \mathcal{M}}\homset{\mathcal{A}}{Im}{a}\otimes
    Im
\end{equation}
that is, every object $b$ is a colimit over $I\colon \mathcal{M}\to \mathcal{A}$ \emph{in a canonical way}. Furthermore, if $T\colon \mathcal{A}\to \mathcal{B}$ is cocontinuous then by \eqref{isomorphism:density-presentation-identity}, we have,
\begin{equation} \label{isomorphism:cocontinuous-after-dense}
    Ta\isomorphic T\parens{\int^{m\in \mathcal{M}}\homset{\mathcal{A}}{Im}{a}\otimes
    Im}\isomorphic \int^{m\in \mathcal{M}}\homset{\mathcal{A}}{Im}{a}\otimes
    TIm
\end{equation}
which implies that if $TI\isomorphic SI$ then $T\isomorphic S$, that is, cocontinuous extensions over dense functors, if they exist, are unique (as always, up to unique isomorphism).

\smallskip
We now apply these results to the special case of presheaf categories. If $\mathcal{A}$ is a small Banach category, the category of contractive functors $\opposite{\mathcal{A}}\to \Ban$ will be denoted by $\PShv(\mathcal{A})$ and its objects are called \emph{Banach space-valued presheaves}, or simply presheaves if there is no ambiguity about the codomain category being $\Ban$. Following the bundle terminology of section \ref{section:measurable-bundles-hilbert-spaces}, a general presheaf will be denoted by $\xi$, $\zeta$, etc.~and morphisms by $\tau$, $\phi$, etc. By the usual pointwise computation of limits and colimits, $\PShv(\mathcal{A})$ is a complete Banach $2$-space. The Yoneda functor $\Yoneda\colon \mathcal{A}\to \PShv(\mathcal{A})$ associates to each $a\in \mathcal{A}$, the representable $\opposite{\mathcal{A}}\to \Ban$ given on objects by
\begin{equation*}
    b\mapto \Yoneda(a)b= \homset{\mathcal{A}}{b}{a}
\end{equation*}

Computing the left Kan extension of $\Yoneda$ along $\Yoneda$ we have,
\begin{align*}
    \Lan_{\Yoneda}{\Yoneda}(\xi) &\isomorphic \int^{a\in \mathcal{A}}\homset{\mathcal{A}}{\Yoneda(a)}{\xi}\otimes \Yoneda(a) \\
        &\isomorphic \int^{a\in \mathcal{A}}\xi(a)\otimes \Yoneda(a) \\
        &\isomorphic \int^{a\in \mathcal{A}}\Yoneda(a)\otimes \xi(a) \\
        &\isomorphic \xi
\end{align*}
where we have used Yoneda's lemma and isomorphism \eqref{isomorphism:generalized-yoneda-coend-form}. By \cite[chapter 5, theorem 5.1]{kelly:enriched-category-theory2005} a functor $F\colon \mathcal{M}\to \mathcal{A}$ is dense iff there is \emph{some} natural isomorphism $\id_{\mathcal{A}}\isomorphic \Lan_{F}F$. Therefore:

\begin{theorem} \label{theorem:yoneda-dense}
The Yoneda embedding $\Yoneda\colon \mathcal{A}\to \PShv(\mathcal{A})$ is dense.
\end{theorem}

If $F\colon\mathcal{A}\to \mathcal{B}$ is a contractive functor with $\mathcal{B}$ a Banach $2$-space, we can lift $F$ to $\PShv(\mathcal{A})$ by taking the left Kan extension $\Lan_{\Yoneda}F$. Theorem \ref{theorem:left-kan-extension-coend} coupled with Yoneda's lemma yields:
\begin{equation} \label{isomorphism:coend-formula-left-kan-extension}
\begin{split}
    \Lan_{\Yoneda}F(\xi) &\isomorphic \int^{a\in \mathcal{A}} \homset{\PShv}{\Yoneda(a)}{\xi}\otimes Fa \\
        &\isomorphic \int^{a\in \mathcal{A}}\xi(a)\otimes F(a)
\end{split}
\end{equation}

%Paragraph:
%- free cocompletion.

This leads to the next very important theorem that characterizes the Yoneda embedding $\Yoneda\colon \mathcal{A}\to \PShv(\mathcal{A})$ as the \emph{free cocompletion} of $\mathcal{A}$.

\begin{theorem} \label{theorem:free-cocompletion}
Let $\mathcal{A}$ be a small Banach category and $F\colon \mathcal{A}\to \mathcal{B}$ a functor into a Banach $2$-space. Then there is a cocontinuous functor
\begin{equation*}
    \universal{F}\colon \PShv(\mathcal{A})\to \mathcal{B}
\end{equation*}
such that $\universal{F}\Yoneda\isomorphic F$. Furthermore, if $T\colon \PShv(\mathcal{A})\to \mathcal{B}$ is another such cocontinuous functor, there is a unique isometric isomorphism closing the $2$-diagram \ref{2diagram:universal-property-yoneda-embedding}.
\begin{figure}[htbp]
    \begin{equation*}
    \xymatrix{
        \PShv(\mathcal{A}) \rrtwocell_{\universal{F}}^{T}{^{\isomorphic}} & & \mathcal{B} \\
            & \ar@{}[u]|<<<{\Uparrow\isomorphic} & \\
        \mathcal{A} \ar[uu]^{\Yoneda} \ar@/_1pc/[uurr]_{F} &
    }
    \end{equation*}
    \caption{Universal property of the Yoneda embedding.}
    \label{2diagram:universal-property-yoneda-embedding}
\end{figure}
\end{theorem}

Theorem \ref{theorem:free-cocompletion} implies that the functor $F\mapto \universal{F}$ establishes an equivalence
\begin{equation} \label{equivalence:free-cocompletion-biadjunction}
    \homset{\BanModCat}{\PShv(\mathcal{A})}{\mathcal{B}}\equivalent \homset{\BanCat}{\mathcal{A}}{\mathcal{B}}
\end{equation}

Note that in \eqref{equivalence:free-cocompletion-biadjunction} we have an equivalence, not an isomorphism. In other words, we have a \emph{biadjunction}, not a $2$-adjunction. This follows from the fact that the factorization of $1$-cells is only defined up to (unique) isometric isomorphism, the universal $2$-cell $\eta$ filling the triangle in diagram \ref{2diagram:left-Kan-extension}. As a consequence, it is \emph{not} true that the free Banach $2$-space construction $\PShv$ extends to a $2$-functor such that the equivalence \eqref{equivalence:free-cocompletion-biadjunction} is natural. Instead, since the $1$-cells are only defined up to (unique) isometric $2$-cells, $\PShv$ extends to a \emph{pseudo $2$-functor}, roughly, a $2$-functor respecting composition of $1$-cells up to coherent isomorphism. For the exact definition of pseudo $2$-functor, biadjunction and related notions and results we refer the reader to \cite{kelly-street:elements-2categories1974} or the more recent \cite{fiore:pseudo-limits-biadjoints-pseudo-algebras2006}, \cite{lack:2categories-companion} and references therein. Although we will not expatiate on these inherently bicategorial notions, we do note that constructions with Banach $2$-spaces tend to yield biadjunctions and not $2$-adjunctions.

\smallskip
The smallness of $\mathcal{A}$ is necessary for the arguments of this section. If $\mathcal{A}$ is not small then its cocompletion can still be constructed and is the category of \emph{accessible presheaves} (see \cite[chapter 4, section 8]{kelly:enriched-category-theory2005}). The completeness of the cocompletion is then a more delicate matter. We refer the reader to \cite{day-lack:limits-small-functors2006} for more information.

\subsection{Free Banach \texorpdfstring{$2$}{2}-spaces over sets}
\label{subsection:free-banach-2spaces-sets}

A very special, and especially important, class of Banach $2$-spaces is the class of free Banach $2$-spaces over a set $X$.  Their study in this and the next subsection has two aims: first to show how some of the results in \cite{yetter:categorical-linear-algebra} on categorified linear algebra can be recovered and second, as a toy model of categorified measure theory.

\smallskip
By virtue of $X$ being discrete, two things happen. First, $X$ is isomorphic to its dual $\opposite{X}$ so that there is no difference between presheaves and their duals, precosheaves. Second, every functor $F\colon X\to \Ban$ is automatically contractive. This leads us to identify the category $\PShv(X)$ with the category of \emph{bundles of Banach spaces over $X$}. The definitions simply repeat those of section \ref{section:measurable-bundles-hilbert-spaces} by replacing Hilbert with Banach spaces. If we denote by $\Bundle(X)$ the Banach category of Banach bundles and bounded natural transformations, the next theorem is almost self-evident.

\begin{theorem} \label{theorem:bundles-presheaves}
The functor that to a bundle $\xi$ assigns the functor $x\mapto \xi_{x}$ is an isomorphism
\begin{equation} \label{isomorphism:bundles-presheaves}
    \Bundle(X)\isomorphic \PShv(X)
\end{equation}
\end{theorem}

The $\hom$-spaces of $\Bundle(X)$ are easily described. If $\mathcal{A}$ is small, the end formula (\cite[chapter 2, section 2]{kelly:enriched-category-theory2005}) for the space of natural transformations $F\to G$ says that we have an isometric isomorphism
\begin{equation}\label{isomorphism:end-formula-hom-functor-categories}
    \homset{\exponential{\mathcal{B}}{\mathcal{A}}}{F}{G}\isomorphic \int_{a\in \mathcal{A}}\homset{\mathcal{B}}{Fa}{Ga}
\end{equation}

If $\xi$ and $\zeta$ are bundles, then the end formula gives (dropping the $X$ from the $\hom$-spaces):
\begin{equation} \label{isomorphism:hom-spaces-bundle}
    \homset{\Bundle}{\xi}{\zeta}\isomorphic \int_{x\in
    X}\homset{\Ban}{\xi_{x}}{\zeta_{x}}\isomorphic \prod_{x\in X}\homset{\Ban}{\xi_{x}}{\zeta_{x}}
\end{equation}

Since $X$ is discrete, $X\isomorphic \opposite{X}$ and, in terms of bundles, a representable $\Yoneda^{x}\colon X\to \Ban$ is just a Dirac $\delta$-like bundle supported on a point (or in more measure theoretic terms, the \emph{characteristic bundle of a singleton set}):
\begin{equation*}
    \delta_{x}(y)\defequal
    \begin{cases}
        \field &\text{if $y= x$,} \\
        \zero &\text{otherwise.}
    \end{cases}
\end{equation*}

Inside $\Bundle(X)$, Yoneda's lemma implies \emph{Schur's lemma}, that is:
\begin{equation} \label{isomorphism:schur-lemma}
    \homset{\Bundle(X)}{\delta_{x}}{\delta_{y}}\isomorphic
    \begin{cases}
        \field &\text{if $x= y$,} \\
        \zero &\text{otherwise.}
    \end{cases}
\end{equation}

More generally, the evaluation or fiber functors $\evaluation_{x}\colon \xi\mapto \xi_{x}$ are representable with $\homset{\Bundle}{\delta_{x}}{\xi}\isomorphic \xi_{x}$. Since the Yoneda embedding $X\to \Bundle(X)$ is dense, the coend formula gives the canonical decomposition of a bundle $\xi$ as a sum of $\delta$-likes:
\begin{equation} \label{isomorphism:canonical-decomposition-bundle}
    \xi \isomorphic \int^{x\in X}\xi_{x}\otimes \delta_{x} \isomorphic \sum_{x\in X}\xi_{x}\otimes \delta_{x}
\end{equation}

Formula \eqref{isomorphism:canonical-decomposition-bundle} has important consequences for the categories $\Bundle(X)$. First, we can introduce a symmetric monoidal structure on $\Bundle(X)$ by taking the pointwise monoidal product:
\begin{equation*}
    \xi\otimes \zeta\isomorphic \sum_{x\in X}\parens{\xi_{x}\otimes \zeta_{x}}\delta_{x}
\end{equation*}

But denoting in exponential notation $\exponential{\varrho}{\zeta}$ the bundle $x\mapto \homset{\Ban}{\zeta_{x}}{\varrho_{x}}$,
\begin{equation*}
\begin{split}
    \homset{\Bundle}{\xi\otimes \zeta}{\varrho} &\isomorphic \prod_{x\in X}\homset{\Ban}{\xi_{x}\otimes \zeta_{x}}{\varrho_{x}} \\
        &\isomorphic \prod_{x\in X}\homset{\Ban}{\xi_{x}}{\homset{\Ban}{\zeta_{x}}{\varrho_{x}}} \\
        &\isomorphic \homset{\Bundle}{\xi}{\exponential{\varrho}{\zeta}}
\end{split}
\end{equation*}
and $\Bundle(X)$ is closed.

\begin{theorem} \label{theorem:bundle-symmetric-monoidal-closed}
The Banach category $\Bundle(X)$ with the pointwise tensor product $\xi\otimes \zeta\colon x\mapto \xi_{x}\otimes \zeta_{x}$ is a symmetric monoidal closed category.
\end{theorem}

Now, let $T\colon \Bundle(X)\to \Bundle(Y)$ be a cocontinuous functor. By the decomposition formula \eqref{isomorphism:canonical-decomposition-bundle} applied to both categories we have
\begin{align*}
    T\xi &\isomorphic T\parens{\sum_{x\in X}\xi_{x}\otimes \delta_{x}} \\
        &\isomorphic \sum_{x\in X}\xi_{x}\otimes T(\delta_{x}) \\
        &\isomorphic \sum_{x\in X}\xi_{x}\otimes \parens{\sum_{y\in Y}T_{y}^{x}\otimes \delta_{y}} \\
        &\isomorphic \sum_{y\in Y}\parens{\sum_{x\in X}\xi_{x}\otimes T_{y}^{x}}\otimes \delta_{y}
\end{align*}

This computation shows that we can identify a cocontinuous $T\colon \Bundle(X)\to \Bundle(Y)$ with a \emph{matrix of Banach spaces $T_{x}^{y}$} parameterized by $\pair{x}{y}\in X\times Y$. In other words, we have the following theorem.

\begin{theorem} \label{theorem:cocontinuous-functors-free-product}
The functor $T\mapto T_{x}^{y}$ establishes an equivalence
\begin{equation} \label{equivalence:cocontinuous-functors-free-product}
    \homset{\BanModCat}{\Bundle(X)}{\Bundle(Y)}\equivalent \Bundle(X\times Y)
\end{equation}
\end{theorem}

Using the equivalence \eqref{equivalence:cocontinuous-functors-free-product}, we can identify the composite of cocontinuous functors with \emph{matrix multiplication}. If $S\colon \Bundle(Y)\to \Bundle(Z)$ is a cocontinuous functor with matrix representation $S_{y}^{z}$, then $ST$ is a cocontinuous functor with matrix representation given by:
\begin{equation}\label{isomorphism:composition-matrix-composition}
    \parens{ST}_{x}^{z}\isomorphic \sum_{y\in Y}S_{y}^{z}\otimes T_{x}^{y}
\end{equation}

%Paragraph:
%- distributors.

Note that while composition of cocontinuous functors is associative on the nose and yields a $2$-category, matrix composition given by \eqref{isomorphism:composition-matrix-composition} only yields a \emph{bicategory}. The identification of cocontinuous functors $\Bundle(X)\to \Bundle(Y)$ with bifunctors on $X\times Y$ is but a special case of a much more general phenomenon and leads to the important concept of \emph{distributor} in J.~B\'{e}nabou's terminology. Alternative names are \emph{profunctor}, \emph{bimodule} and even \emph{module} in case the reader lives upside down in the southern hemisphere and speaks Australian\footnote{In the definitive catalogue of the vagaries of human madness that is \emph{The Anatomy of Melancholy}, Robert Burton in one of his endless lists of invectives notes of a lover's blindness:
\begin{quote}
The major part of lovers are carried headlong like so many brute beasts; reason counsels one way, thy friends, fortunes, shame, disgrace, danger, and an ocean of cares that will certainly follow; yet this furious lust precipitates, counterpoiseth, weighs down on the other; though it be their utter undoing, perpetual infamy, loss, yet they will do it, and become at last \emph{insensati}, void of sense; degenerate into dogs, hogs, asses, brutes; as Jupiter into a bull, Apuleius an ass, Lycaon a wolf, Tereus a lapwing, Callisto a bear, Elpenor and Gryllus into swine by Circe.
\end{quote}

This methinks, is the surest scientifick explanation for why Australia features such a strong school in category theory: by walking upside down, the blood that in an average healthy male concentrates in the fiery region of the loins, is diverted to the lower parts of the brain, oxygenating it, opening the pores and allowing a better ventilation of the lustful humors, unobstipated vision, clear thinking, etc.}. We will get to these matters later in a more general context, so we proceed with other concerns.

\smallskip
Another important consequence of these results is that the bitensor product $\PShv(X)\bitensor \PShv(Y)$ of free Banach $2$-modules exists and is $\PShv(X\times Y)$. This comes from the biadjunction \eqref{equivalence:free-cocompletion-biadjunction} and using the line of reasoning used to establish isomorphism \eqref{isomorphism:l1-product}.

\begin{theorem}[Fubini] \label{theorem:fubini-free-banach-2modules}
The bitensor product $\PShv(X)\bitensor \PShv(Y)$ exists and there is an equivalence
\begin{equation} \label{equivalence:fubini-free-banach-2modules}
    \PShv(X)\bitensor \PShv(Y)\equivalent \PShv(X\times Y)
\end{equation}
\end{theorem}

Coupling theorems \ref{theorem:cocontinuous-functors-free-product} and \ref{theorem:fubini-free-banach-2modules}, we have that the $2$-category of free Banach $2$-modules over sets, cocontinuous functors and contractive natural transformations with bitensor product \eqref{equivalence:fubini-free-banach-2modules} is \emph{biclosed}\footnote{There is a potential terminology clash here, as the term \emph{biclosed} as been used in the literature in a different sense -- basically, as the generalization of closedness to monoidal, possibly non-symmetric categories. Since we will make no use of the latter, there should be no confusion.}.

\subsection{Categorified measure theory: the discrete case}
\label{subsection:categorified-measure-theory-discrete-case}

In this subsection, we tackle categorified measure theory in the simplest case of discrete Boolean algebras and show that all the basic results of section \ref{section:measure-algebras-integrals} for ordinary integrals have straightforward categorified analogues proved by suitably categorifying the proofs -- the essential work was already done in subsections \ref{subsection:presheaf-categories} and \ref{subsection:free-banach-2spaces-sets}. But first, let us review the basic definitions.

\smallskip
Let $\Omega$ be a Boolean algebra and $\mathcal{A}$ a Banach $2$-space. A \emph{precosheaf} $\Omega\to \mathcal{A}$ is a simply a contractive functor $\mu\colon\Omega\to \mathcal{A}$. Let $\mathcal{E}$ be a partition of $E= \sup\mathcal{E}$. Then there is a unique map closing the triangle \ref{diagram:induced-precosheaf-map-finite-partitions1} where $\mu_{\mathcal{E}}$ denotes the coproduct $\sum_{F\in \mathcal{E}}\mu(F)$.
\begin{figure}[htbp]
    \begin{equation*}
    \xymatrix{
        \mu_{\mathcal{E}} \ar@{-->}[r]^{\varepsilon_{\mathcal{E}}} & \mu(E) \\
        \mu(E_{n}) \ar[u]^{i_{\mu(E_{n})}} \ar[ur]_{\mu_{E_{n}, E}} &
    }
    \end{equation*}
\caption{The map $\varepsilon_{\mathcal{E}}\colon \mu_{\mathcal{E}}\to \mu(E)$.}
\label{diagram:induced-precosheaf-map-finite-partitions1}
\end{figure}

A cosheaf is a precosheaf for which $\varepsilon_{\mathcal{E}}$ is an isomorphism for every partition $\mathcal{E}$, or equivalently, for which the cone $\parens{\mu_{E_{n}, E}}$ is a coproduct. More precisely:

\begin{definition} \label{definition:cosheaf}
A precosheaf $\mu$ is \emph{finitely additive} if $\varepsilon_{\mathcal{E}}$ is an isometric isomorphism for every finite partition. It is \emph{$\sigma$-additive} if $\varepsilon_{\mathcal{E}}$ is an isometric isomorphism for every countable partition and \emph{completely additive} if $\varepsilon_{\mathcal{E}}$ is an isometric isomorphism for every partition.
\end{definition}

In this subsection only, the unadorned term \emph{cosheaf} will mean \emph{$\sigma$-additive precosheaf}. The fact that the Boolean algebra $\subsets{X}$ is atomic will allow us to bypass the more delicate analytical issues that arise when we consider $\sigma$-additive cosheaves. As remarked in the introduction, we view cosheaves as \emph{categorified measures}. The category of $\mathcal{A}$-valued cosheaves will be denoted by $\CoShv(\Omega, \mathcal{A})$. It can be proved that $\mu$ is a cosheaf iff it satisfies the \emph{dual patching condition} for countable covers.

\smallskip
The cosheaves of more importance are the completely additive ones while the ones coming from concrete situations (e.g.~the cosheaves $E\mapto \Integrable_{1}(E, \mu)$ for a measure algebra $\pair{\Omega}{\mu}$) are $\sigma$-additive. The gap between $\sigma$-additivity and complete additivity\footnote{Completely additive functionals are treated in \cite[chapter 2, section 6]{fremlin:measure-theory32002}. They are the functionals of importance for the abstract Radon-Nikodym theorem.} is bridged by the same technical condition used in the ordinary measure-theoretic situation: the \emph{countable chain condition}.

\begin{definition} \label{definition:ccc}
A Boolean algebra $\Omega$ satisfies the \emph{countable chain condition} (or has ccc for short) if every partition $\mathcal{E}$ is countable\footnote{Partitions are anti-chains, so it would probably be more reasonable to call the condition in \ref{definition:ccc}, the \emph{countable antichain condition} or cac. But the terminology is entrenched and besides, it is not difficult to prove that if $\Omega$ is $\sigma$-complete then ccc is equivalent to the fact that there are no uncountable chains in $\Omega$.}.
\end{definition}

The ccc condition can be seen as the purely Boolean algebraic version of total $\sigma$-finiteness. A Boolean algebra that is $\sigma$-complete and has ccc is automatically order complete. If $\Omega$ is $\sigma$-complete and the measure $\mu$ is $\sigma$-additive and non-degenerate then $\Omega$ has ccc, so that the ccc condition is automatically satisfied in the most important class of measure algebras.

\smallskip
Now, let us consider Boolean algebras of the form $\subsets{X}$ for a set $X$. Since we want the ccc condition on $\subsets{X}$, $X$ must be a countable set. In this case, the Boolean algebra $\subsets{X}$ is complete, ccc and purely atomic. We have an obvious isomorphism,
\begin{equation*}
    \atoms(\subsets{X})\isomorphic X
\end{equation*}
identifying the atoms of $\subsets{X}$ with the elements of $X$. This implies that for every $E\subseteq X$ we have the equality
\begin{equation*}
    E= \sup\set{F\in \atoms(\subsets{X})\colon F\subseteq E}
\end{equation*}

The next theorem should now be clear.

\begin{theorem}[Radon-Nikodym] \label{theorem:discrete-categorified-rn}
The functor $\mu\mapto \parens{x\mapto \mu(x)}$ establishes an isomorphism
\begin{equation} \label{isomorphism:discrete-categorified-rn}
    \homset{\CoShv}{\subsets{X}}{\mathcal{A}} \isomorphic \homset{\BanCat}{X}{\mathcal{A}}
\end{equation}
\end{theorem}

Using the biadjunction \ref{equivalence:free-cocompletion-biadjunction}, we have that,
\begin{equation} \label{equivalence:discrete-categorified-rn}
    \homset{\CoShv}{\subsets{X}}{\mathcal{A}}\equivalent \homset{\BanModCat}{\PShv(X)}{\mathcal{A}}
\end{equation}
and taking into account theorem \ref{theorem:bundles-presheaves}, it follows that the universal cosheaf
\begin{equation*}
    \characteristic\colon \subsets{X}\to \PShv(X)
\end{equation*}
is given by $x\mapto \delta_{x}$ and thus by equivalence \eqref{equivalence:discrete-categorified-rn}, $\characteristic$ is just the Yoneda embedding. If $\mu\colon \subsets{X}\to \mathcal{A}$ is a cosheaf, then it factors via $\characteristic$ as $\int\differential{\mu}\colon \PShv(X)\to \mathcal{A}$ by taking the left Kan extension $\Lan_{\characteristic}\mu$. These constructions identify the category of measurable\footnote{Since every measurable is integrable in the categorified setting, we will use the two adjectives indistinguishably.} functors $\IBundle(\subsets{X})$ as the Banach $2$-space $\PShv(X)\isomorphic \Bundle(X)$. As expected, this Banach $2$-space does not depend on the cosheaf $\mu$, only on the set $X$. The coend formula \eqref{isomorphism:left-kan-extension-coend} for left Kan extensions implies that,
\begin{equation} \label{isomorphism:coend-formula-integral-functor}
    \integral{X}{\xi}{\mu}\isomorphic \sum_{x\in X}\xi_{x}\otimes \mu(x)
\end{equation}
which is what was to be expected from the fact that in atomic Boolean algebras, integrals are nothing but sums. In particular, the functor $\xi\mapto \integral{X}{\xi}{\mu}$ is cocontinuous.

\smallskip
Now that we have constructed the integral functor $\xi\mapto \integral{X}{\xi}{\mu}$ we can talk about the \emph{indefinite integral}. This is just the cosheaf
\begin{equation*}
    \subsets{X}\to \mathcal{A}\colon E\mapto \integral{E}{\xi}{\mu}\defequal \sum_{x\in E}\xi_{x}\otimes \mu(x)
\end{equation*}

We have constructed the integral of scalar (that is, $\Ban$-valued) bundles against vector measures. Following the path taken with measure algebras and the construction of vector integrals, the vector integral should be constructed via a bitensor product. Define the category $\ell(X, \mathcal{A})$ to be $\homset{\BanCat}{X}{\mathcal{A}}$. If $\mathcal{A}$ is the category $\Ban$ of Banach spaces this is just another notation for $\PShv(X)\isomorphic \Bundle(X)$.

\begin{theorem} \label{theorem:functor-ell-bitensor}
The bifunctor
\begin{equation}\label{functor:universal-bicocontinuous}
    \pair{\xi}{a}\mapto \xi\otimes a\colon x\mapto \xi_{x}\otimes a
\end{equation}
establishes a natural equivalence,
\begin{equation} \label{equivalence:functor-ell-bitensor}
    \ell(X, \mathcal{A})\equivalent \ell(X)\bitensor \mathcal{A}
\end{equation}
\end{theorem}
\begin{proof}
We merely sketch the proof. Start by noticing that since tensors are cocontinuous in each variable, \eqref{functor:universal-bicocontinuous} is bicocontinuous. Now, similarly to the case of $\ell(X)$, we introduce a subcategory of delta-like bundles that is dense in $\ell(X, \mathcal{A})$. Let $X\in X$ and $a\in \mathcal{A}$. Define:
\begin{equation} \label{equation-definition:delta-bundles-ell-functor}
    \delta_{x}\otimes a(y)\defequal
    \begin{cases}
        a &\text{if $y= x$,} \\
        \zero &\text{otherwise.}
    \end{cases}
\end{equation}

The canonical decomposition in $\ell(X, \mathcal{A})$ takes the form
\begin{equation} \label{isomorphism:canonical-decomposition-ell-functor}
    \xi\isomorphic \sum_{x\in X}\delta_{x}\otimes \xi_{x}
\end{equation}

Since $X$ is discrete, the end-formula for $\hom$-spaces of functor categories yields that $\homset{\ell}{\xi}{\zeta}\isomorphic \prod_{x\in X}\homset{\mathcal{A}}{\xi_{x}}{\zeta_{x}}$. These two facts put together, imply that the full subcategory of \eqref{equation-definition:delta-bundles-ell-functor} is dense in $\ell(X, \mathcal{A})$. Now, let $\varphi\colon \ell(X)\otimes \mathcal{A}\to \mathcal{B}$ be a bicocontinuous bifunctor. Define the functor $\universal{\varphi}\colon \ell(X, \mathcal{A})\to \mathcal{B}$ on elementary tensors $\delta_{x}\otimes a$ by $\varphi(\delta_{x}, a)$ and extend by density to the whole $\ell(X, \mathcal{A})$:
\begin{equation} \label{isomorphism:linearization-biconcotinuous}
    \universal{\varphi}(\xi)\isomorphic \sum_{x\in X}\varphi(\delta_{x}, \xi_{x})
\end{equation}

Straightforward colimit manipulations show that $\universal{\varphi}$ is cocontinuous. It is also clear that $\universal{\varphi}i\isomorphic \varphi$ and that any such functor would have to be isomorphic (up to a unique isomorphism) to $\universal{\varphi}$ by \eqref{isomorphism:linearization-biconcotinuous}.
\end{proof}

If $\mu\colon \subsets{X}\to \Ban$ is a cosheaf we define the integral functor $\ell(X, \mathcal{A})\to \mathcal{A}$ to be
\begin{equation} \label{isomorphism:vector-integral-functor}
    \integral{X}{\xi}{\mu}\defequal \sum_{x\in X}\mu(x)\otimes \xi_{x}
\end{equation}

We have now the $2$-diagram of figure \ref{2diagram:naturality-discrete-integral-functor} where $T\colon \mathcal{A}\to \mathcal{B}$ is a cocontinuous functor.
\begin{figure}[htbp]
    \begin{equation*}
    \xymatrix{
        \ell(X)\bitensor \mathcal{A} \ar[rrr]^{\id_{\ell(X)}\bitensor T} \ar[dd]_{\equivalent} \ar[dr]^{\int_{X}\differential{\mu}\bitensor \id_{\mathcal{A}}} &  &  & \ell(X)\bitensor
        \mathcal{B} \ar[dd]^{\equivalent} \ar[dl]_{\int_{X}\differential{\mu}\bitensor \id_{\mathcal{B}}} \\
          & \mathcal{A} \ar[r]^{T} & \mathcal{B} & \\
        \ell(X, \mathcal{A}) \ar[rrr]_{\covariant{T}} \ar[ur]^{\int_{X}\differential{\mu}} &  &  &
        \ell(X, \mathcal{B}) \ar[ul]_{\int_{X}\differential{\mu}}
    }
    \end{equation*}
\caption{Naturality of $\int_{X}\differential{\mu}$.}
\label{2diagram:naturality-discrete-integral-functor}
\end{figure}

Commutativity of \ref{2diagram:naturality-discrete-integral-functor} means that between every pair of parallel paths there is a canonical $2$-cell isomorphism between them and that \ref{2diagram:naturality-discrete-integral-functor} filled with these $2$-cells is commutative. The proof of all this is a relatively straightforward but somewhat tedious chore, and since we will (eventually) generalize everything to other Boolean algebras $\Omega$ we concentrate our attention in a small piece of \ref{2diagram:naturality-discrete-integral-functor} and leave the rest to be handled by the interested reader. For the most part, they are the direct categorification of the computations with simple functions needed to establish theorem \ref{theorem:l1-tensor-product}.

\smallskip
Let $\xi\in \ell(X, \mathcal{A})$ be an $\mathcal{A}$-valued bundle. Then by \eqref{isomorphism:canonical-decomposition-ell-functor}, $\xi\isomorphic \sum_{x\in X}\delta_{x}\otimes \xi_{x}$ and we have:
\begin{align*}
    T\parens{\integral{X}{\xi}{\mu}} &\isomorphic T\parens{\integral{X}{\parens{\sum_{x\in X}\delta_{x}\otimes \xi_{x}}}{\mu}} \displaybreak[0]\\
        &\isomorphic T\parens{\sum_{x\in X}\mu(x)\otimes \xi_{x}} \displaybreak[0]\\
        &\isomorphic \sum_{x\in X}\mu(x)\otimes T\xi_{x} \displaybreak[0]\\
        &\isomorphic \integral{X}{\parens{\sum_{x\in X}\delta_{x}\otimes T\xi_{x}}}{\mu} \displaybreak[0]\\
        &\isomorphic \integral{X}{\covariant{T}(\xi)}{\mu}
\end{align*}

This chain of isomorphisms provides the unique isometric isomorphism filling the lower trapezoid in \ref{2diagram:naturality-discrete-integral-functor} and shows that cocontinuous functors commute with integrals. In other words, the contortions of \cite{yetter:measurable-categories2005} to define the right notion of functor between categories of integrable functors are not needed.

\smallskip
Fubini's theorem now follows from the fact that $\ell$ is a left biadjoint to the unit ball $2$-functor $\BanCat\to \Cat$. The argument is just the categorification of the argument used above to establish the isometric isomorphism $\ell_{1}(X)\otimes \ell_{1}(Y)\isomorphic \ell_{1}(X\times Y)$. The isomorphism between iterated integrals follows likewise.

\section{Categorified measures and integrals}
\label{section:categorified-measures-integrals}

With the basic notions of Banach $2$-spaces in hand and categorified measure theory for the discrete case clarified, we justify in this final section some of the principles advanced in section \ref{section:towards-categorified-measure-theory} for categorified measure theory. Full details will be given in a future paper (\cite{rodrigues:categorified-measure-theory}). We will be concerned not so much with the direct categorification of the results of section \ref{section:measure-algebras-integrals}, but in the new phenomena that arise in the categorified setting. The former was already covered in subsection \ref{subsection:categorified-measure-theory-discrete-case} for discrete Boolean algebras $\subsets{X}$ and the details for the more general case do not differ by much once the basic constructions are understood. We will convey only their flavor and perforce many things will be left unsaid.

\subsection{Banach sheaves}
\label{subsection:banach-sheaves}

%Paragraph:
%- Grothendieck topologies in a Boolean algebra.

By principles \ref{principle:category-measurable-bundles} and \ref{principle:measurable-bundles-sheaves}, the objects that we put under the integral sign are sheaves for a suitable Grothendieck topology on $\Omega$. The general definition of Grothendieck topology can be seen in \cite[chapter III]{maclane-moerdijk:sheaves-geometry-logic1992} or in \cite[chapter 3, section 2]{borceux:handbook-categorical-algebra31994}. Some considerable simplifications are possible in the case of Boolean algebras (more generally, lattices). Let $\Omega$ be a Boolean algebra viewed as a category with objects the elements of $\Omega$ and an arrow $E\to F$ iff $E\subseteq F$. Since there is at most one arrow between any two objects, specifying a Grothendieck topology in $\Omega$ is easy as a \emph{sieve} in $\Omega$ is just a downward closed subset, that is, subsets $\mathcal{I}\subseteq \Omega$ such that if $E\in \mathcal{I}$ and $F\subseteq E$ then $F\in \mathcal{I}$. But when computing suprema of subsets in lattice, there is no loss of generality if we pass from a subset $U$ to the sieve $\sieve(U)$ generated by it, so that we can specify the covering sieves by simply saying what are the covering families.

\begin{definition} \label{definition:cover-boolean-algebras}
We say that a family $\mathcal{V}$ of elements of $\Omega$ \emph{covers} $E\in \Omega$ if
\begin{equation} \label{inequality:cover-boolean-algebras}
    E= \sup\mathcal{V}
\end{equation}

A family $\mathcal{V}$ that covers $E$ is, logically enough, called a \emph{cover} of $E$.
\end{definition}

For \eqref{inequality:cover-boolean-algebras} to make sense we have to assume that the supremum on the right-hand side exists. We could do without such completeness hypothesis if we were more careful in wording our definitions. There is little to be gained from such extra generality however, so here and elsewhere, whenever needed, we simply assume the existence of all the required suprema.

\smallskip
There are three Grothendieck topologies available, depending on the cardinality of the allowed covers.

\begin{definition} \label{definition:finite-countable-complete-topologies}
The \emph{finite Grothendieck topology} on a Boolean algebra is the topology generated by the finite covers. The \emph{countable or $\sigma$-Grothendieck topology} (on a $\sigma$-complete Boolean algebra) is the topology generated by the countable covers and the \emph{complete Grothendieck topology} (on a complete Boolean algebra) is the one generated by arbitrary cardinality covers.
\end{definition}

As the reader may suspect, for measure theory the topology of more interest on a Boolean algebra is the $\sigma$-topology. The complete topology is the best for categorial reasons that should be obvious; the category of sheaves for this topology is the category of sheaves for $\Omega$ as a locale. The gap between the countable and the complete topologies is bridged by the countable chain condition \ref{definition:ccc}. The finite topology, the topology of interest in this section, roughly corresponds to finite additivity and will be used to build a universal home for categorified measure theory just as $\Integrable_{\infty}(\Omega)$ is the universal home for all measures; its importance will become apparent by the end of the subsection.

%Paragraph: two aspects.
%- the local and global study of sheaves.
%- the fibration \Meas -> \Bool.

\smallskip
Before proceeding, two other aspects should be mentioned. We will study sheaves locally or one Boolean algebra at a time, but many of these sheaves (e.g.~$\Omega\mapto \Integrable_{\infty}(\Omega)$) are sheaves on the site\footnote{This site, as well as the sites of the slice categories $\BoolTop\slice X$, are all non-small, so strictly speaking we would have to trim $\BoolTop$ to a suitable small subcategory.} of the distributive category $\BoolTop$ of Boolean spaces with the standard, or open cover Grothendieck topology. This too, is a well known theme of algebraic geometry where it goes under the name of ``petit'' and ``gros'' topoi. On the other hand, a functor like the $\Integrable_{1}$-functor of section \ref{subsection:bochner-integral} is a cosheaf not on $\BoolTop$ but on the site $\Meas$ of measure algebras. We are being stingy on the details, in part because both these aspects are best understood in the context of fibrations (in the categorial sense), starting with the fibration $\Meas\to \Bool$ given by $\pair{\Omega}{\mu}\mapto \Omega$. Fibrations are a powerful tool to tackle these and other constructions like the several Banach $2$-space stacks floating around, but we have been dodging even mentioning them so as not to raise the categorial requirements bar for the paper. For the reader interested in the theory of fibrations, we refer him to \cite[chapter 8]{borceux:handbook-categorical-algebra21994} and \cite{streicher:fibred-categories-benabou}.

\smallskip
Please keep in mind that functor, presheaf, natural transformation, colimit, etc.~is \emph{always} understood in the $\Banc$-enriched sense. Since we have a topology on $\Omega$ we can define sheaves with values in any category $\mathcal{A}$. More specifically, we make the following definition.

\begin{definition} \label{definition:banach-sheaf-ba}
Let $\mathcal{A}$ be a finitely complete Banach $2$-space. A presheaf $\xi\colon \Omega\to \mathcal{A}$ is a \emph{Banach sheaf} if it is a sheaf for the finite topology.
\end{definition}

Sheaves for the other two topologies can also be defined, but we must require more completeness from the codomain $\mathcal{A}$. The next theorem works for any of the three topologies of definition \ref{definition:finite-countable-complete-topologies}. The reader can fish out the details of the proof from \cite[section 10]{rezk:fibrations-homotopy-colimits-simplicial-sheaves1998}.

\begin{theorem} [Sheaf condition]\label{theorem:sheaf-condition}
A presheaf $\xi\colon \Omega\to \mathcal{A}$ is a sheaf iff for every partition $\mathcal{E}$ of $E\in \Omega$, the canonical map filling diagram \ref{diagram:sheaf-condition} is an isometric isomorphism, in other words, the cone $\parens{p_{E, F}}$ with $F\in \mathcal{E}$, is a product.
\begin{figure}[htbp]
    \begin{equation*}
    \xymatrix{
        & \prod_{F\in \mathcal{E}}\xi(F) \ar[d]^{p_{\xi(F)}} \\
        \xi(E) \ar[r]_{p_{E, F}} \ar@{-->}[ur] & \xi(F)
    }
    \end{equation*}
\caption{The sheaf condition.}
\label{diagram:sheaf-condition}
\end{figure}
\end{theorem}

The Stone equivalence between the category of Boolean algebras and the dual of the category of Boolean spaces was proven very useful in subsection \ref{subsection:stone-space-ba-weak-integrals}, so it is natural to try to find the topological equivalent of sheaves as defined in \ref{definition:banach-sheaf-ba}. Let us investigate this matter more closely.

%Paragraph:
%- sheaf condition = bounded patching condition.

\smallskip
Sheaves in a topological space $X$ with values in a category $\mathcal{A}$ such as a Banach $2$-space are defined by simply copying the usual definition (see \cite[chapter II, section 1]{maclane-moerdijk:sheaves-geometry-logic1992}): they are presheaves $\xi\colon \open(X)\to \mathcal{A}$ satisfying the \emph{patching condition}, that is, for any open cover $\cover{U}{i}$ of $U$ diagram \ref{diagram:patching-condition-equalizer} is an equalizer. The left arrow is the map $s\mapto \parens{p_{U, U_{i}}(s)}$ and the parallel pair on the right is $\family{s}{i}\mapto \parens{p_{U_{i}, U_{i}\cap U_{j}}(s_{i})}$ and $\family{s}{i}\mapto \parens{p_{U_{j}, U_{i}\cap U_{j}}(s_{j})}$.
\begin{figure}[htbp]
    \begin{equation*}
    \xymatrix{
        \xi(U) \ar[r] & \prod_{i}\xi(U_{i}) \ar[r] \ar@<1ex>[r] & \prod_{i, j}\xi(U_{i}\cap U_{j})
    }
    \end{equation*}
\label{diagram:patching-condition-equalizer}
\caption{Equalizer patching condition.}
\end{figure}

The patching condition for sheaves of Banach spaces on $X$ translates into the following two requirements for every open cover $\cover{U}{i}$ of $U\subseteq X$:
\begin{enumerate}
  \item \label{enum:uniqueness-condition}
  \textbf{Uniqueness of patching:} For every $s\in \xi(U)$ we have the equality,
  \begin{equation*}
    \norm{s}= \sup\set{\norm{s_{i}}}
  \end{equation*}
  with $s_{i}= p_{U, U_{i}}(s)$.

  \item \label{enum:patching-condition-bounded-families}
  \textbf{Existence of patching for bounded families:} For every family $\family{s}{i}, i\in I$ with $s_{i}\in \xi(U_{i})$ such that,
  \begin{equation*}
    p_{U_{j}, U_{j}\cap U_{k}}(s_{j})= p_{U_{k}, U_{j}\cap U_{k}}(s_{k})
  \end{equation*}
  and $\sup\set{\norm{s_{i}}\colon i\in I}< \infty$, then there is $s\in \xi(U)$ such that $p_{U, U_{i}}(s)= s_{i}$.
\end{enumerate}

The conditions \eqref{enum:uniqueness-condition} and \eqref{enum:patching-condition-bounded-families} for a sheaf of Banach spaces are similar to the patching conditions for sheaves of sets and such algebraic objects as groups or linear spaces. We must not be fooled by the similarities however, because the presence of the boundedness hypothesis in \eqref{enum:patching-condition-bounded-families} implies for example, that if we apply the underlying functor $\Ban\to \Vect$ we do \emph{not} obtain a sheaf of linear spaces because the underlying functor is not continuous. More directly relevant to our purposes, is the fact that if we analyze the construction of the \'{e}tale space of a sheaf, the \emph{etalification}, and rerun it in the case of Banach-space valued presheaves, then it will \emph{not} produce the associated sheaf or \emph{sheafification} functor. This has led N.~Auspitz in \cite{auspitz:qsheaves-banach-spaces1975} (see also \cite{banaschewski:sheaves-banach-spaces1977}) to strengthen the patching condition to the so-called \emph{approximation condition}.

%Paragraph:
%- etalification (not exactly).

\smallskip
If $\xi$ is a presheaf on $X$ then the \emph{stalk} of $\xi$ at $x\in X$ is the Banach space,
\begin{equation} \label{def-equation:stalk-presheaf}
    \stalk{\xi}{x}\defequal \Colim_{U\in \mathcal{F}_{x}} \xi(U)
\end{equation}
where $\mathcal{F}_{x}$ is the open neighborhood filter base of $x\in X$. If $\iota_{U, x}\colon \xi(U)\to \stalk{\xi}{x}$ is the universal cone and $s\in \xi(U)$ then we will denote by $s_{x}$, the \emph{germ of $s$ at $x$}, the element $\iota_{U, x}(s)$. We construct a presheaf $\etalification{\xi}$ by putting,
\begin{equation} \label{functor:etalification}
    U\mapto \etalification{\xi}(U)\defequal \prod_{x\in U}\stalk{\xi}{x}
\end{equation}
with the obvious restriction maps. There is a natural presheaf map $\eta_{\xi}\colon \xi\to \etalification{\xi}$ given by $s\in \xi(U)\mapto \etalification{s}= \family{s}{x}$.

\smallskip
Note that $\prod_{x\in U}\stalk{\xi}{x}$ is the bounded section space of the bundle $\coprod_{x\in U}\stalk{\xi}{x}\to U$ and in fact, we will see below that it almost (but not quite) gives the etalification of a presheaf of Banach spaces.

%Paragraph:
%- definition of Banach sheaf.

\begin{definition} \label{definition:banach-sheaf-top}
Let $X$ be a topological space which we will assume compact Hausdorff. A presheaf $\xi$ of Banach spaces is a \emph{Banach sheaf} if it satisfies the following two conditions for every open cover $\cover{U}{i}$ of $U\subseteq X$:
\begin{enumerate}
  \item \label{def-enum:banach-sheaf-uniqueness}
  \textbf{Uniqueness of patching:} For every $s\in \xi(U)$ we have the equality,
  \begin{equation*}
    \norm{s}= \sup\set{\norm{s_{i}}}
  \end{equation*}
  with $s_{i}= p_{U, U_{i}}(s)$.

  \item \label{def-enum:banach-sheaf-approximation}
  \textbf{Approximation condition:} If $t\in \etalification{\xi}(U)$ is such that for every $\epsilon > 0$ and every $x\in U$ there is an open neighborhood $V\subseteq U$ of $x$ and an $s\in \xi(V)$ such that
  \begin{equation*}
    \norm{p_{U, V}(t) - \etalification{s}}<\epsilon
  \end{equation*}
  then $t= \etalification{\sigma}$ for some $\sigma\in \xi(U)$.
\end{enumerate}
\end{definition}

If $\family{s}{i}$ is a family satisfying the hypothesis of the patching condition \eqref{enum:patching-condition-bounded-families} and we put $\epsilon = 1/n$ in the approximation condition \eqref{def-enum:banach-sheaf-approximation}, then taking the limit $n\converges \infty$ we see that a Banach sheaf satisfies the bounded patching condition \eqref{enum:patching-condition-bounded-families}. The approximation condition is a completeness (in the Cauchy metric space sense) requirement, stating that if a family $\family{s}{i}$ can be \emph{locally approximated} than it can be patched up. In \cite{auspitz:qsheaves-banach-spaces1975} it is shown that the presheaf of bounded analytic functions on the unit disk of the complex plane satisfies the uniqueness and bounded patching conditions but not the approximation condition.

\smallskip
For the proof of the next theorem we refer the reader to \cite{banaschewski:injective-banach-sheaves1977}.

\begin{theorem} \label{theorem:almost-etalification}
Let $X$ be a compact Hausdorff space and $\family{B}{x}$ a family of Banach spaces. Then the presheaf
\begin{equation}\label{functor:family-spaces-banach-sheaf}
    U\mapto \prod_{x\in U}B_{x}
\end{equation}
is a Banach sheaf.
\end{theorem}

%Paragraph:
%- the etalification functor.

Theorem \ref{theorem:almost-etalification} implies that the presheaf \eqref{def-equation:stalk-presheaf} is a Banach sheaf. It \emph{almost} gives the correct sheaf of sections (and thus the correct notion of ``\'{e}tale bundle of Banach spaces''). In order to get it, we have to cut down the monad $\xi\mapto \etalification{\xi}$ on the category of presheaves to its idempotent part by taking the equalizer of the parallel pair
\begin{equation*}
    \eta_{\etalification{\xi}}, \etalification{\eta_{\xi}}\colon \etalification{\xi}\to \etalification{\etalification{\xi}}
\end{equation*}

The existence of the functor $\xi\mapto \etalification{\xi}$ implies that the category $\Shv(X)$ of Banach sheaves on $X$ is reflective in $\PShv(X)$ and thus a complete Banach $2$-space. Obviously enough, the reflector $\PShv(X)\to \Shv(X)$ will be called the sheafification functor. The category of Banach sheaves has been given several alternative descriptions. Starting with the etalification \ref{functor:etalification}, it can be shown that it is equivalent to a certain category of bundles of Banach spaces. Another important description takes as starting point the following theorem (also lifted from \cite{auspitz:qsheaves-banach-spaces1975}), that describes the sheafification of constant presheaves.

\begin{theorem} \label{theorem:sheafification-constant-presheaves}
Let $B$ be a Banach space. The sheafification of the constant presheaf $U\mapto B$ is the presheaf,
\begin{equation} \label{functor:sheafification-constant-presheaves}
    U\mapto \CBounded(U, B)
\end{equation}
where $\CBounded(U, B)$ is the Banach space of bounded continuous functions $U\to B$.
\end{theorem}

Note that for a compact Hausdorff space $X$, the global section space of \eqref{functor:sheafification-constant-presheaves} is just $\Continuous(X, B)$. This latter space is isometrically isomorphic to the \emph{injective tensor product}\footnote{See \cite[chapter 3]{ryan:introduction-tensor-product-banach-spaces2002} for the definition and basic properties of the injective tensor product of Banach spaces.}:
\begin{equation}\label{isomorphism:sheafification-injective-tensor-product}
    \Continuous(X, B)\isomorphic \Continuous(X)\injotimes B
\end{equation}

The category $\Shv(X)$ is a symmetric monoidal closed category. The tensor product is given by the sheafification of the pointwise tensor product structure in $\PShv(X)$. This means that the sheafification is a monoidal functor and thus takes monoids into monoids. Since the presheaf $U\mapto \Real$ is a monoid in $\PShv(X)$, and the initial monoid at that, it follows that its sheafification $U\mapto \CBounded(U)$ is a monoid in $\Shv(X)$ that acts in a canonical way on every sheaf. This leads to the description of Banach sheaves as sheaf modules over this sheaf algebra satisfying a \emph{local convexity condition}.

\smallskip
These two equivalences (and more) can be found in \cite{hoffman-keimel:sheaf-theoretical-concepts-analysis1979}. However, the most striking result is the equivalence of $\Shv(X)$ with the category of \emph{internal Banach spaces in the topos of (set-valued) sheaves on $X$} for a suitable, intuitionistically valid definition of Banach space (see \cite{burden-mulvey:banach-spaces-categories-sheaves1979}). This is not the place to dwell on these matters; the next theorem (\cite[proposition 7]{banaschewski:sheaves-banach-spaces1977} in new garbs) should suffice to underline the importance of these remarks.

\begin{theorem} \label{theorem:banach-sheaves-finite-topology}
Let $\Omega$ be a Boolean algebra and $X= \Stone(\Omega)$ its Stone space. If $\xi$ is a presheaf in $\Stone(\Omega)$ then
\begin{equation}\label{functor:finite-topology-sheaf}
    \invimage{\eta}(\xi)\colon E\mapto \xi(\eta(E))
\end{equation}
is a presheaf on $\Omega$. The functor $\xi\mapto \invimage{\eta}(\xi)$ establishes an equivalence between the category $\Shv(X)$ of Banach sheaves on $\Stone(\Omega)$ and the category $\Shv(\Omega)$ of sheaves on $\Omega$ for the finite topology.
\end{theorem}
\begin{proof}
Since $\eta$ is a Boolean algebra isomorphism it is clear that \eqref{functor:finite-topology-sheaf} is a sheaf for the finite topology. We briefly describe the weak inverse: if $\xi$ is a sheaf on $\Omega$ for the finite topology, since $\eta$ is a Boolean algebra isomorphism $\Omega\isomorphic \clopen(X)$, we have a presheaf $\clopen(X)\to \Ban$ given by
\begin{equation*}
    E\mapto \xi(\inverse{\eta}(E))
\end{equation*}

This presheaf satisfies the finite patching condition. To see that it satisfies the bounded patching condition note that if $\set{\eta(E_{i})}$ is an open cover of $\eta(E)$, by compactness it has a finite open subcover $\set{\eta(E_{i_{j}})}$. The fact that it is an open subcover implies that the ideal base $\set{\eta(E_{i_{j}})}$ is cofinal in the ideal base $\set{\eta(E_{i})}$ and thus
\begin{equation*}
    \xi(\inverse{\eta}(E))\isomorphic \Lim_{j}\xi(\inverse{\eta}(E_{i_{j}}))\isomorphic \Lim_{i}(\inverse{\eta}(E_{i}))
\end{equation*}

In fact, the same compactness argument also yields that $E\mapto \xi(\inverse{\eta}(E))$ satisfies the approximation condition \eqref{def-enum:banach-sheaf-approximation}. By the well-known result that a sheaf is uniquely determined by its values on a base for the topology (see \cite[chapter 2,, section 1, theorem 3]{maclane-moerdijk:sheaves-geometry-logic1992} -- the result readily extends to Banach sheaves), we obtain a Banach sheaf on $\Stone(\Omega)$.
\end{proof}

Theorem \ref{theorem:banach-sheaves-finite-topology} together with the remarks preceding it, places the sheaf-half of categorified measure theory within well-trod territory: internal Banach space theory in sheaf topoi. We have available powerful results such as Gelfand-Naimark duality (see \cite{banaschewski-mulvey:globalization-gelfand-duality-theorem2006}), the beginnings of measure theory (see the delightful \cite{jackson:sheaf-theoretic-approach-measure-theory2006}), etc.~Things get even more interesting if we realize that if $\Omega$ is order complete (for example, if it is a $\sigma$-measure algebra satisfying the ccc condition) then the topos $\Shv(\Omega)$ satisfies the internal version of the axiom of choice (see \cite[chapter 7, theorem 7.3]{barr-wells:toposes-triples-theories1983}). By a theorem of Diaconescu, the topos $\Shv(\Omega)$ is Boolean and therefore \emph{classical analysis is valid} in $\Shv(\Omega)$.

\smallskip
The internal language of a topos, and in general the logical aspects of topos theory, is fearsome black wizardry that I am not competent to deal with, so in the rest of the paper we will have to proceed using cruder, more primitive methods.

\smallskip
To finish off this subsection, here is an example of a Banach sheaf that should give warm fuzzy feelings to all Banach-space theorists out there. Recall that the forgetful functor $\CHaus\to \Set$ that to a compact Hausdorff space associates the underlying set has a left adjoint $\Free$ given on $X\in \Set$ by the Stone-Cech compactification of $X$ endowed with the discrete topology. Its universal property is depicted in diagram \ref{diagram:free-compact-hausdorff} for the case of bounded functions $f\colon X\to \field$.
\begin{figure}[htbp]
    \begin{equation*}
    \xymatrix{
        \Free(X) \ar@{-->}[r]^{\universal{f}} & \field \\
        X \ar[u] \ar[ur]_{f} &
    }
    \end{equation*}
\caption{The free compact Hausdorff space $\Free(X)$.}
\label{diagram:free-compact-hausdorff}
\end{figure}

On the other hand there is a map
\begin{equation*}
    X\to \Stone(\subsets{X})
\end{equation*}
that to $x\in X$ associates the principal filter generated by the singleton atom $\set{x}\in \subsets{X}$. By discreteness of $X$, this map is continuous and thus there is a unique map closing the triangle \ref{diagram:homeomorphism-stone-cech-stone-space-atomic-ba}. It can be shown that the dashed map is bijective and thus an homeomorphism.
\begin{figure}[htbp]
    \begin{equation*}
    \xymatrix{
        \Free(X) \ar@{-->}[r] & \Stone(\subsets{X}) \\
        X \ar[u] \ar[ur] &
    }
    \end{equation*}
\caption{The homeomorphism $\Free(X)\isomorphic \Stone(\subsets{X})$.}
\label{diagram:homeomorphism-stone-cech-stone-space-atomic-ba}
\end{figure}

Since the dashed map of \ref{diagram:homeomorphism-stone-cech-stone-space-atomic-ba} is an homeomorphism, the map $X\to \Stone(\subsets{X})$ satisfies the same universal property of diagram \ref{diagram:free-compact-hausdorff}. Coupling theorem \ref{theorem:ba-mr-stone-space} with the Riesz representation theorem, we have the isometric isomorphism
\begin{equation*}
    \badditive(\subsets{X})\isomorphic \dual{\Continuous(\Stone(X))}
\end{equation*}

The right-hand side of this isomorphism, tells us that we can obtain a large collection of measures on $\subsets{X}$ by taking point mass Dirac measures $\delta_{\widehat{x}}$ for $\widehat{x}$ a point of $\Stone(\subsets{X})$, or an ultrafilter of $\subsets{X}$. By definition of the induced map on Stone-Cech compactifications, we have,
\begin{equation*}
    \delta_{\widehat{x}}\universal{f}= \universal{f}(\widehat{x})= \lim_{\widehat{x}}f
\end{equation*}
where the right-hand side is the limit of $f$ with respect to the ultrafilter $\widehat{x}$. Since $\delta_{\widehat{x}}$ is a measure we can speak of a.e.~equality of functions and from,
\begin{equation*}
    \integral{\Free(X)}{\abs{\universal{f}}}{\delta_{\widehat{x}}}= \delta_{\widehat{x}}\abs{\universal{f}}= \lim_{\widehat{x}}\abs{f}
\end{equation*}
we have that $\universal{f}\aeequal 0$ iff $\lim_{\widehat{x}}\abs{f}= 0$.

\smallskip
Let us generalize. Since $X$ with the discrete topology is completely regular, it follows that the the inclusion $X\to \Free(X)$ is an open map with dense range and thus it establishes an equivalence
\begin{equation}\label{equivalence:sheaves-stone-cech-bundles-discrete}
    \Shv(\Stone(\subsets{X}))\equivalent \Shv(X)
\end{equation}

But by discreteness of $X$ the Banach $2$-space on the right-hand side is isomorphic to $\Bundle(X)$. This means that a Banach sheaf $\xi$ on $\Stone(\subsets{X})$ amounts to an $X$-family $\family{B}{x}$ of Banach spaces. Given a Banach sheaf $\xi= \family{B}{x}$ on $X$, the stalk $\stalk{\xi}{\widehat{x}}$ at an ultrafilter point $\widehat{x}\in \Stone(\subsets{X})$ is the colimit $\Colim_{U\in \mathcal{U}_{\widehat{x}}}\prod_{x\in U\cap X}B_{x}$ but this colimit is precisely the subspace of sections $f\in \prod_{x\in X}B_{x}$ of the global section space such that
\begin{equation*}
    \lim_{\widehat{x}}\norm{f}= 0
\end{equation*}

Fixing an ultrafilter $\widehat{x}\in \Stone(\subsets{X})$, the \emph{ultraproduct $\prod_{\widehat{x}}B_{x}$ of the family $\xi= \family{B}{x}$} is the quotient
\begin{equation} \label{equation-definition:ultraproducts}
    \prod_{\widehat{x}}B_{x}\defequal \parens{\prod_{x\in X}B_{x}}\quotient \stalk{\xi}{\widehat{x}}
\end{equation}

The ultraproduct \eqref{equation-definition:ultraproducts} could be further described in terms of direct images, but the preceding comments should indicate that this important tool of Banach space theory is but a special case of a sheaf-theoretic construction.

\subsection{Cosheaves and inverters}
\label{subsection:cosheaves-inverters}

Cosheaves on Boolean algebras were defined in subsection \ref{subsection:categorified-measure-theory-discrete-case}. Cosheaves have been studied in connection with Borel-Moore homology (\cite[chapter V]{bredon:sheaf-theory1997}). The link with distributions on topoi in the sense of F.~W.~Lawvere was uncovered in \cite{pitts:product-change-base-toposes1985}. I first stumbled upon this type of distributions when plodding through the first chapter of \cite{bunge-funk:singular-coverings-toposes2006} and refer the reader to it for the references to the original F.~W.~Lawvere's papers. I was not able to access them, but F.~W.~Lawvere's ideas on space and quantity are mentioned briefly in \cite[section 7]{lawvere:taking-categories-seriously2005}. This paper, in conjunction with \cite{lawvere:metric-spaces-generalized-logic-closed-categories2002}, is a veritable gold mine of ideas ready for the plunder and a highly recommended reading\footnote{It would not be too far off the mark, if I had titled this paper ``Taking seriously ``Taking categories seriously''''.}.

\smallskip
Recall theorem \ref{theorem:universal-property-L1} expressing the universal property of $\Integrable_{1}(\Omega, \mu)$. For every Lipschitz finitely additive map $\nu\colon \Omega\to B$ there is a unique bounded linear map $\Integrable_{1}(\Omega, \mu)\to B$ such that triangle \ref{diagram:universal-property-L1II} is commutative.
\begin{figure}[htbp]
    \begin{equation*}
    \xymatrix{
        \Integrable_{1}(\Omega, \mu) \ar@{-->}[r]^-{\int\differential\nu} & B \\
        \Omega \ar[u]^-{\characteristic} \ar[ur]_{\nu} &
    }
    \end{equation*}
\caption{Universal property of $\Integrable_{1}(\Omega, \mu)$.}
\label{diagram:universal-property-L1II}
\end{figure}

Categorifying, we want a Banach $2$-space $\Integrable(\Omega)$ such that the functor $\nu\mapto \int\differential\nu$ establishes an equivalence
\begin{equation}\label{equivalence:cosheaves-integrable-sheaves}
    \CoShv(\Omega, \mathcal{A})\equivalent \homset{\BanModCat}{\Integrable(\Omega)}{\mathcal{A}}
\end{equation}

In order to construct such a gadget, let us go back to the construction of the space $\simple(\Omega)$ of simple functions on $\Omega$ depicted diagrammatically in \ref{diagram:construction-simple-functions}.
\begin{figure}[htbp]
    \begin{equation*}
    \xymatrix{
        \Free(\Omega) \ar@{-->}[rr] \ar[dr]^{\pi} & & A \\
          & \simple(\Omega) \ar@{-->}[ur] & \\
          & \Omega \ar[luu]^{\iota} \ar[u]_{\characteristic} \ar@/_/[ruu]_{\mu} &
    }
    \end{equation*}
\caption{Construction of $\simple(\Omega)$.}
\label{diagram:construction-simple-functions}
\end{figure}

Diagram \ref{diagram:construction-simple-functions} tells us that $\simple(\Omega)$ is obtained by taking the free linear space $\Free(\Omega)$ and then the \emph{coequalizer} that kills the subspace $\iota(E\cup F) - \iota(E) - \iota(F)$ for all pairwise disjoint $E, F\in \Omega$. An appropriate modification of this construction inside the category of Banach spaces produces $\Integrable_{\infty}(\Omega)$ directly.

\smallskip
Now categorify. Let $\mu\colon \Omega\to \mathcal{A}$ be a cosheaf (to fix ideas, for the finite topology). As seen in subsection \ref{subsection:presheaf-categories}, the free Banach $2$-space on $\Omega$ is $\PShv(\Omega)$ and the induced map corresponding to the top right arrow of \ref{diagram:construction-simple-functions} is the left Kan extension $\Lan_{\Yoneda}\mu$. The category of integrable bundles $\IBundle(\Omega)$ will be a suitable quotient of $\PShv(\Omega)$ that when composed with Yoneda yields the universal cosheaf $\characteristic$.

\smallskip
In order to guess what this quotient is let $\sequence{E}{n}$ be a partition of $E\in \Omega$. The Yoneda functor would be a cosheaf if the cone $\Yoneda(E_{n})\to \Yoneda(E)$ was a coproduct, that is, $\Yoneda(E)\isomorphic \sum_{n}\Yoneda(E_{n})$. But on this hypothesis, by Yoneda lemma we have for every presheaf $\xi$
\begin{align*}
    \xi(E) &\isomorphic \homset{\PShv}{\Yoneda(E)}{\xi} \\
        &\isomorphic \homset{\PShv}{\sum_{n}\Yoneda(E_{n})}{\xi} \\
        &\isomorphic \prod_{n}\homset{\PShv}{\Yoneda(E_{n})}{\xi} \\
        &\isomorphic \prod_{n}\xi(E_{n})
\end{align*}

By theorem \ref{theorem:sheaf-condition} this means that $\xi$ is a sheaf so that our quotient is the \emph{sheafification functor}. Another way to look at the sheafification and that shows more clearly in what sense it is a categorification of the coequalizer $\Free(\Omega)\to \simple(\Omega)$ is to note that there is parallel pair of functors and a $2$-cell that we depict as in diagram \ref{2diagram:2cell-partitions-sup}.
\begin{figure}[htbp]
    \begin{equation*}
    \xymatrix{
        \exponential{\mathcal{A}}{\opposite{\Omega}} \rtwocell_{\Phi}^{\Psi}{\eta} & \exponential{\mathcal{A}}{\opposite{\partitions(\Omega)}}
    }
    \end{equation*}
\caption{The $2$-cell $\eta_{\mathcal{E}}\colon \xi(\sup\mathcal{E})\to \prod_{F\in \mathcal{E}}\xi(E)$.}
\label{2diagram:2cell-partitions-sup}
\end{figure}

The functor $\Psi$ sends the presheaf $\xi\colon \opposite{\Omega}\to \mathcal{A}$ to the functor $\mathcal{E}\mapto \xi(\sup\mathcal{E})$ and the functor $\Phi$ assigns to $\xi$ the functor $\mathcal{E}\mapto \prod_{F\in \mathcal{E}}\xi(E)$. The inner $2$-cell $\eta$ is the map closing triangle \ref{diagram:natural-eta-sheaves}.
\begin{figure}[htbp]
    \begin{equation*}
    \xymatrix{
          & \prod_{n}\xi(E_{n}) \ar[d] \\
        \xi(E) \ar[r] \ar@{-->}[ur]^{\eta_{\xi}} & \xi(E_{n})
    }
    \end{equation*}
\caption{The natural map $\eta_{\mathcal{E}}$.}
\label{diagram:natural-eta-sheaves}
\end{figure}

To force every presheaf $\xi$ to become a sheaf we have to make the natural map $\eta_{\xi}$ an isomorphism, or in other words, we have to \emph{invert the $2$-cell $\eta$}. Note the categorification pattern: instead of imposing an equality (take the coequalizer), we invert a $2$-cell. Quotients of categories that invert given $2$-cells are known as \emph{coinverters} and are one of a special class of finite colimits peculiar to $2$-categories. We simply state the definition next and refer the reader to \cite[section 4]{kelly:elementary-observations-2categorical-limits1989} for more information on this type of $2$-categorial colimits.

%ToDo:
%- lack the specification of appropriate categories.

\begin{definition} \label{definition:coinverter}
Let $\mathcal{A}$ be a $2$-category and $\eta\colon f\to g\colon a\to b$ a $2$-cell. The coinverter of $\eta$ is a $1$-cell $\pi\colon b\to c$ such that $p\eta$ is an isomorphism and $\pi$ is universal among such $1$-cells. That is, every $1$-cell $h\colon b\to d$ such that $h\eta$ is an isomorphism factors uniquely through $\pi$ as in diagram \ref{2diagram:universal-property-coinverter},
\begin{figure}[htbp]
    \begin{equation*}
    \xymatrix{
        a \rtwocell_{g}^{f}{\eta} & b \ar[r]^{\pi} \ar[dr]_{h} & c \ar@{-->}[d]^{\universal{h}} \\
          &  & d
    }
    \end{equation*}
\caption{Universal property of the coinverter $\pi$.}
\label{2diagram:universal-property-coinverter}
\end{figure}
and the map $h\mapto \universal{h}$ establishes an \emph{isomorphism between the categories of such $1$-cells}.
\end{definition}

Coinverters or \emph{localizations}\footnote{Localization usually means something stronger: a \emph{finitely continuous} coinverter. Contrary to the category of sets or categories of algebraic objects like $\Vect$, finite limits do not commute with filtered colimits in $\Banc$. What does happen is that $\Banc$ is $\omega_{1}$-locally presentable with $\omega_{1}$ the first uncountable ordinal (see \cite[chapter 5]{borceux:handbook-categorical-algebra21994}, especially example 5.2.2.e). It is also true that filtered colimits commute with finite products, a result that is useful for Banach sheaves by theorem \ref{theorem:banach-sheaves-finite-topology}. This is proposition 5 of \cite{borceux:higher-order-sheaves-banach-modules1982}.} are hard to construct\footnote{The most famous example of a coinverter may very well be the coinversion of the homotopy functor by the weak equivalences in homotopy theory (see \cite{gabriel-zisman:calculus-fractions-homotopy-theory1967}). In a sense, the whole technology ranging from the calculus of fractions to Quillen model categories, has come to life precisely to handle coinverter constructions.}, but it is well known that localizations of presheaf categories $\PShv(\mathcal{A})$ correspond to Grothendieck topologies on $\mathcal{A}$, so that sheaf categories should be the solution for the birepresentation \eqref{equivalence:cosheaves-integrable-sheaves}.

\smallskip
Denote by $r$ the sheafification reflection $\PShv(\Omega)\to \Shv(\Omega)$. Precomposing with Yoneda we obtain a functor $\characteristic\colon \Omega\to \Shv(\Omega)$.

\begin{theorem} \label{theorem:integrables-sheaves}
Precomposition with $\characteristic\colon \Omega\to \Shv(\Omega)$ establishes a natural equivalence
\begin{equation}\label{equivalence:integrables-sheaves}
    \homset{\BanModCat}{\Shv(\Omega)}{\mathcal{A}}\equivalent \CoShv(\Omega, \mathcal{A})
\end{equation}
\end{theorem}
\begin{proof}
The crucial result is that $\mu\colon \Omega\to \mathcal{A}$ is a cosheaf iff the left Kan extension $\Lan_{\Yoneda}\mu$ factors uniquely (up to inner isomorphism $2$-cells) through the sheafification $r$ as in diagram \ref{diagram:factorization-cosheaves}.
\begin{figure}[htbp]
    \begin{equation*}
    \xymatrix{
        \PShv(\Omega) \ar[rr]^{\Lan_{\Yoneda}\mu} \ar[dr]^{r} & & \mathcal{A} \\
          & \Shv(\Omega) \ar@{-->}[ur] & \\
          & \Omega \ar[luu]^{\Yoneda} \ar[u]_{\characteristic} \ar@/_/[ruu]_{\mu} &
    }
    \end{equation*}
\caption{Factorization of cosheaves via $\Shv(\Omega)$.}
\label{diagram:factorization-cosheaves}
\end{figure}
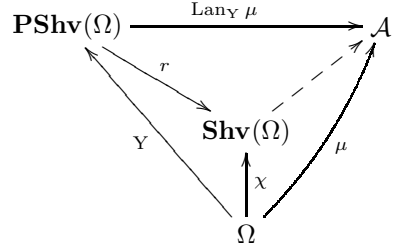

Once that is done, equivalence \ref{equivalence:integrables-sheaves} follows from the results of subsection \ref{subsection:presheaf-categories}, theorem \ref{theorem:free-cocompletion} in particular. The remarks above already hinted at how to construct the factorization; for more details, the reader is advised the consult the refences given at the beginning of the subsection.
\end{proof}

Theorem proves that $\characteristic\colon \Omega\to \Shv(\Omega)$ is the universal cosheaf: every other cosheaf factors uniquely through it. Although $\characteristic$ is a cosheaf, for each $E\in \Omega$, $\characteristic(E)$ is a sheaf. Though limits $\Shv(\Omega)$ are formed pointwise, colimits in $\Shv(\Omega)$ are \emph{not}. Instead, one forms the colimit in $\PShv(\Omega)$ and sheafifies it. Because of this, the sheaves $\characteristic(E)$ are \emph{not} projective. On the other hand, since sheafification is a left adjoint, it is cocontinuous and it is well-known that the composition of a dense functor with a left adjoint is dense. The following coend calculation proves this:
\begin{align*}
    \xi &\isomorphic ri(\xi) \\
        &\isomorphic r\parens{\int^{E\in \Omega}\homset{\PShv}{\Yoneda(E)}{i\xi}\otimes \Yoneda(E)} \\
        &\isomorphic r\parens{\int^{E\in \Omega}\homset{\Shv}{\characteristic(E)}{\xi}\otimes \Yoneda(E)} \\
        &\isomorphic \int^{E\in \Omega}\homset{\Shv}{\characteristic(E)}{\xi}\otimes \characteristic(E)
\end{align*}

Since the last term is precisely the left Kan extension $\Lan_{\characteristic}\characteristic$ by the coend formula, density of $\characteristic$ follows from the already cited theorem \cite[chapter 5, theorem 5.1]{kelly:enriched-category-theory2005}.

%Paragraph:
%- \Shv as the categorification of \Integrable_{\infty}.

\smallskip
Given the resemblance of the universal properties of $\Integrable_{\infty}(\Omega)$ and $\Integrable_{1}(\Omega, \mu)$ one could ask what exactly have we categorified. In the ordinary decategorified setting, the difference between $\Integrable_{\infty}(\Omega)$ and $\Integrable_{1}(\Omega, \mu)$ arises from the fact that boundedness and summability are non-vacuous concepts. These distinctions disappear in the categorified setting since \emph{all} small colimits exist, but there is still one substantial difference between $\Integrable_{\infty}(\Omega)$ and $\Integrable_{1}(\Omega, \mu)$ worth noticing. The former is a Banach algebra and the corresponding universal measure $\characteristic$ is spectral. We will now see that the universal cosheaf $\characteristic$ provided by theorem \ref{theorem:integrables-sheaves} is spectral.

\smallskip
In the first place, intersection $\cap$ makes of $\Omega$ a monoidal category\footnote{Even more, since $\Omega$ is a Boolean algebra, it is cartesian closed with products given by $\cap$. The cartesian closed structure of $\Omega$ is also good in the sense that it makes of $\Omega$ a \emph{compact closed} category -- for example, see \cite{freyd-yetter:braided-compact-closed-categories-applications-low-dimensional-topology1989}. This good duality structure is another thing that gets lost in the passage from finite to infinite-dimensional spaces.}. The presheaf category $\PShv(\Omega)$ is monoidal with the pointwise tensor product\footnote{For the knowledgeable reader, we note that the Day convolution on $\PShv(\Omega)$ reduces to the pointwise tensor product. One less monoidal structure to worry about, which is always a good thing. This follows from the fact that $\Omega$ is a lattice with only one arrow between any two objects.} and it is easy to see that $\Yoneda(E\cap F)\isomorphic \Yoneda(E)\otimes \Yoneda(F)$. Since the monoidal structure in $\Shv(\Omega)$ is the sheafification of the pointwise tensor product it follows that:
\begin{equation} \label{isomorphism:characteristic-monoidal}
    \characteristic(E\cap F)\isomorphic \characteristic(E)\otimes \characteristic(F)
\end{equation}

Theorem \ref{theorem:banach-sheaves-finite-topology} can now be seen as the categorification of theorem \ref{theorem:stone-measurable-continuous-map}. The proof of \ref{theorem:banach-sheaves-finite-topology} does not look like a categorification of the proof of \ref{theorem:stone-measurable-continuous-map}, but we \emph{can} categorify this proof! Let $X$ be the Stone space of $\Omega$ and for a clopen $\eta(E)$ of $X$ consider the presheaf on $\open(X)$ given by
\begin{equation} \label{functor:characteristic-sheaves-stone-space}
    U\mapto \prod_{x\in \eta(E)\cap U}\field
\end{equation}

By theorem \ref{theorem:almost-etalification} it is a Banach sheaf. More importantly, we can readily recognize it as the sheafification of the representable of $E\in \clopen(X)\subseteq \open(X)$ and isomorphic to $U\mapto \CBounded(U\cap E)$. Since the representables are dense and the clopens form a basis of $X$, the sheaves \eqref{functor:characteristic-sheaves-stone-space} form a dense subcategory of $\Shv(\Stone(\Omega))$. This is the categorification of the Stone-Weierstrass theorem. Theorem \ref{theorem:banach-sheaves-finite-topology} now drops out by noting that $E\mapto \parens{U\mapto \CBounded(U\cap E)}$ is a spectral cosheaf, taking its direct integral via theorem \ref{theorem:integrables-sheaves}, and then using coend calculus to prove that it is fully-faithful and density of the $U\mapto \CBounded(U\cap E)$ to prove essential surjectiveness.

\smallskip
By theorem \ref{theorem:integrables-sheaves} a cosheaf $\mu$ will factor uniquely via the full subcategory of the characteristic sheaves $\characteristic(E)$.

\begin{definition} \label{definition:simple-sheaves-subcategory}
Define the category $\simple(\Omega)$ of \emph{simple sheaves} to be the full subcategory of $\characteristic(E)$ with $E\in \Omega$.
\end{definition}

The universal cosheaf $\characteristic\colon \Omega\to \Shv(\Omega)$ factors uniquely through the inclusion $\simple(\Omega)\to \Shv(\Omega)$, therefore no confusion should arise if we denote the inclusion $\Omega\to \simple(\Omega)$ also by $\characteristic$. Density of $\simple(\Omega)$ together with the unique factorization of \ref{theorem:integrables-sheaves} yields the following important theorem.

\begin{theorem} \label{theorem:sheaves-presheaves-simples}
Consider the left Kan extensions $\Lan_{\characteristic}\mu$ as in diagram \ref{diagram:sheaves-presheaves-simples}.
\begin{figure}[htbp]
    \begin{equation*}
    \xymatrix{
        \Shv(\Omega) \ar[r]^-{\int\differential\mu} & \mathcal{A} \\
        \simple(\Omega) \ar@/^/[u]^{i} \ar@{-->}[ur]^{\Lan_{\characteristic}\mu} & \\
        \Omega \ar[u]^{\characteristic} \ar[uur]_{\mu} &
    }
    \end{equation*}
\caption{The factorization $\Lan_{\characteristic}\mu$.}
\label{diagram:sheaves-presheaves-simples}
\end{figure}

The functor $\mu\mapto \Lan_{\characteristic}\mu$ establishes an equivalence between Banach $2$-spaces
\begin{equation}\label{equivalence:sheaves-presheaves-simples}
    \homset{\BanCat}{\simple(\Omega)}{\mathcal{A}}\equivalent \CoShv(\Omega, \mathcal{A})
\end{equation}
\end{theorem}

In particular, $\CoShv(\Omega, \mathcal{A})$ is complete if $\mathcal{A}$ is. In view of theorem \ref{theorem:sheaves-presheaves-simples} it pays off to have an idea of what sort of category is $\simple(\Omega)$. The first thing to notice is that $\simple(\Omega)$ is self-dual. To see this, we first need to recall the adjunction between presheaves and precosheaves. If $\mathcal{A}$ is a small Banach category, there is a functor $\LAdjoint\colon \PShv(\mathcal{A})\to \opposite{\PCoShv(\mathcal{A})}$ given on objects by
\begin{equation}\label{map:ladjoint-isbell-conjugation}
    \xi\mapto \parens{a\mapto \homset{\PShv}{\xi}{\Yoneda^{a}}}
\end{equation}

Similarly, there is a functor $\RAdjoint\colon \opposite{\PCoShv(\mathcal{A})}\to \PShv(\mathcal{A})$ given on objects by
\begin{equation}\label{map:radjoint-isbell-conjugation}
    \mu\mapto \parens{a\mapto \homset{\PCoShv}{\mu}{\Yoneda_{a}}}
\end{equation}

The next theorem shows that this pair of functors is part of an adjunction. This adjunction is called Isbell conjugation in \cite[section 7]{lawvere:taking-categories-seriously2005}. Since the paper does not supply a proof we do.

\begin{theorem} [Isbell conjugation]\label{theorem:isbell-conjugation}
Let $\mathcal{A}$ be a small category. There is a natural isomorphism:
\begin{equation} \label{isomorphism:isbell-conjugation}
    \homset{\opposite{\PCoShv}}{\LAdjoint\xi}{\mu}\isomorphic \homset{\PShv}{\xi}{\RAdjoint\mu}
\end{equation}
\end{theorem}
\begin{proof}
The calculus of ends affords a nifty proof of \eqref{isomorphism:isbell-conjugation}. Starting with the right-hand side, using the end formula \ref{isomorphism:end-formula-hom-functor-categories}, cocontinuity, interchange of ends and symmetry of $\otimes$ we have:
\begin{align*}
    \homset{\PShv}{\xi}{\RAdjoint\mu} &\isomorphic \int_{a\in \mathcal{A}}\homset{\Ban}{\xi(a)}{\RAdjoint\mu(a)} \displaybreak[0]\\
        &\isomorphic \int_{a\in \mathcal{A}}\homset{\Ban}{\xi(a)}{\homset{\PCoShv}{\mu}{\Yoneda_{a}}} \displaybreak[0]\\
        &\isomorphic \int_{a\in \mathcal{A}}\homset{\Ban}{\xi(a)}{\int_{b\in \mathcal{A}}\homset{\Ban}{\mu(b)}{\homset{\mathcal{A}}{a}{b}}} \displaybreak[0]\\
        &\isomorphic \int_{a\in \mathcal{A}}\int_{b\in \mathcal{A}}\homset{\Ban}{\xi(a)}{\homset{\Ban}{\mu(b)}{\homset{\mathcal{A}}{a}{b}}} \displaybreak[0]\\
        &\isomorphic \int_{a\in \mathcal{A}}\int_{b\in \mathcal{A}}\homset{\Ban}{\xi(a)\otimes \mu(b)}{\homset{\mathcal{A}}{a}{b}} \displaybreak[0]\\
        &\isomorphic \int_{a\in \mathcal{A}}\int_{b\in \mathcal{A}}\homset{\Ban}{\mu(b)\otimes \xi(a)}{\homset{\mathcal{A}}{a}{b}} \displaybreak[0]\\
        &\isomorphic \int_{a\in \mathcal{A}}\int_{b\in \mathcal{A}}\homset{\Ban}{\mu(b)}{\homset{\Ban}{\xi(a)}{\homset{\mathcal{A}}{a}{b}}} \displaybreak[0]\\
        &\isomorphic \int_{b\in \mathcal{A}}\int_{a\in \mathcal{A}}\homset{\Ban}{\mu(b)}{\homset{\Ban}{\xi(a)}{\homset{\mathcal{A}}{a}{b}}} \displaybreak[0]\\
        &\isomorphic \int_{b\in \mathcal{A}}\homset{\Ban}{\mu(b)}{\int_{a\in \mathcal{A}}\homset{\Ban}{\xi(a)}{\homset{\mathcal{A}}{a}{b}}} \displaybreak[0]\\
        &\isomorphic \int_{b\in \mathcal{A}}\homset{\Ban}{\mu(b)}{\homset{\PShv}{\xi}{\Yoneda^{b}}} \displaybreak[0]\\
        &\isomorphic \int_{b\in \mathcal{A}}\homset{\Ban}{\mu(b)}{\LAdjoint\xi(b)} \displaybreak[0]\\
        &\isomorphic \homset{\PCoShv}{\mu}{\LAdjoint\xi}
\end{align*}
\end{proof}

From Isbell conjugacy, the next theorem follows.

\begin{theorem} \label{theorem:simples-dual-simples}
The Isbell adjunction \eqref{isomorphism:isbell-conjugation} descends to an equivalence
\begin{equation}\label{equivalence:}
    \simple(\dual{\Omega})\equivalent \dual{\simple(\Omega)}
\end{equation}
\end{theorem}

The final ingredient we need is that the complementation map $\icomplement\colon \Omega\to \opposite{\Omega}$ is a Boolean algebra isomorphism between $\Omega$ and its dual Boolean algebra $\opposite{\Omega}$. By taking the left Kan extension as in square \ref{diagram:ba-dual-kan-extension-equivalence} we obtain an equivalence $\PShv(\Omega)\equivalent \PCoShv(\Omega)$.
\begin{figure}[htbp]
    \begin{equation*}
    \xymatrix{
        \PShv(\Omega) \ar@{-->}[rr]^{\Lan_{\Yoneda_{\Omega}}\Yoneda_{\opposite{\Omega}}\icomplement} & & \PCoShv(\Omega) \\
        \Omega \ar[u]^{\Yoneda_{\Omega}} \ar[rr]_{\icomplement} & & \dual{\Omega} \ar[u]_{\Yoneda_{\opposite{\Omega}}}
    }
    \end{equation*}
\caption{The equivalence $\PShv(\Omega)\equivalent \PCoShv(\Omega)$.}
\label{diagram:ba-dual-kan-extension-equivalence}
\end{figure}

It is now easy to see that this equivalence descends to an equivalence
\begin{equation}\label{equivalence:simples-self-dual}
    \simple(\Omega)\equivalent \simple(\opposite{\Omega})
\end{equation}

Next, we describe the sheaves $\characteristic(E)$. Since $\characteristic(E)$ is the sheafification of $\Yoneda(E)$ and by definition we have the isomorphism\footnote{Recall our notational practice of identifying ordinary categories with their free Banach categories.},
\begin{equation*}
    \Omega(F, E)\isomorphic
    \begin{cases}
        \field & \text{if $F\subseteq E$,} \\
        \zero & \text{otherwise.}
    \end{cases}
\end{equation*}
the representable $\Yoneda(E)$ is isomorphic the sheafification\footnote{A word of caution is necessary here. By theorem \ref{theorem:sheafification-constant-presheaves} and the isometric isomorphism \eqref{isomorphism:sheafification-injective-tensor-product} the sheafification of the constant presheaf $E\mapto B$ is $E\mapto \Integrable_{\infty}(E)\injotimes B$. This Banach space is \emph{strictly smaller} than the Banach space $\Integrable_{\infty}(E, B)$ considered in subsection \ref{subsection:failure-radon-nikodym} if $B$ is infinite-dimensional, since the functions in $\Integrable_{\infty}(E)\injotimes B$ have \emph{compact range}. A typical example of a strongly measurable function not in $\Integrable_{\infty}(E)\injotimes B$ is built by picking a bounded sequence with no convergent subsequences. A more sophisticated one goes like this: pick a separable reflexive space $B$. Than its unit ball $\ball(B)$ is compact for the weak topology; take the corresponding Borel structure. Since the inclusion $\ball(B)\to B$ is weakly continuous, by the Pettis theorem \ref{theorem:pettis} it is strongly measurable. But of course, its range is very far from being compact.} of $E\mapto \field$ precomposed with the Boolean algebra map $F\mapto E\cap F$.

\begin{theorem} \label{theorem:sheafification-representables}
For each $E\in \Omega$, the sheaf $\characteristic(E)$ is isomorphic to
\begin{equation*}
    F\mapto \Integrable_{\infty}(E\cap F)
\end{equation*}
\end{theorem}
\begin{proof}
Combine the preceding observations with theorems \ref{theorem:stone-measurable-continuous-map}, \ref{theorem:sheafification-constant-presheaves} and \ref{theorem:banach-sheaves-finite-topology}.
\end{proof}

The next step is the computation of morphism spaces $\homset{\Shv}{\xi}{\zeta}$ where $\xi$ and $\zeta$ are constant sheaves. This is perfectly doable, but the more important case is contained in the next theorem and its proof already contains the basic ingredients of the more general case.

\begin{theorem} \label{theorem:morphism-spaces-constant-line-sheaves}
If $E, F\in \Omega$ then the sheaf $\homset{\Shv}{\characteristic(E)}{\characteristic(F)}$ is isomorphic to
\begin{equation*}
    G\mapto \Integrable_{\infty}(E\cap F\cap G)
\end{equation*}
\end{theorem}
\begin{proof}
The proof proceeds by passing to the Stone space $\Stone(\Omega)$ and then using (a variation of) the Banach-Stone theorem (see \cite[chapter 6, section 2]{conway:course-functional-analysis1990}).
\end{proof}

Theorems \ref{theorem:sheafification-representables} and \ref{theorem:morphism-spaces-constant-line-sheaves} give us a fairly accurate picture of $\simple(\Omega)$ as little more than a categorial disguise of the Banach algebra $\Integrable_{\infty}(\Omega)$. The morphism spaces $\homset{\simple}{E}{F}$ are the sheaves $\Integrable_{\infty}(E\cap F)$ and composition is given by multiplication in the Banach algebra. Note that theorem \ref{theorem:morphism-spaces-constant-line-sheaves} implies the symmetry of $\homset{\simple}{E}{F}$ in its arguments.

\smallskip
Theorem \ref{theorem:sheaves-presheaves-simples} tells us that $\mu\colon \Omega\to \mathcal{A}$ is a cosheaf iff it factors through $\characteristic\colon \Omega\to \simple(\Omega)$. This factorization has a striking description in terms of integration against a spectral measure -- this will be described in subsection \ref{subsection:spectral-measure-cosheaf}. For now, we just note that theorem \ref{theorem:sheaves-presheaves-simples} lands us in a situation analogous to that observed at the end of subsection \ref{subsection:free-banach-2spaces-sets} with sets $X$ replaced by the Banach categories $\simple(\Omega)$. In particular, we are able to prove a Fubini theorem, characterize the cocontinuous functors
\begin{equation*}
    \Integrable(\Omega)\to \Integrable(\Sigma)
\end{equation*}
either as distributors $\opposite{\simple(\Omega)}\otimes \simple(\Sigma)\to \Ban$ or as cosheaves on $\Omega\otimes \Sigma$, etc. It also allows us to explain why D.~Yetter's definition of measurable functors in \cite{yetter:measurable-categories2005} as a certain measurable bundle over the product $X\times Y$ is essentially correct. This is done in full in \cite{rodrigues:categorified-measure-theory}, but one last theorem is worth mentioning. Recall that a category is \emph{Cauchy complete} if every idempotent splits.

\begin{theorem} \label{theorem:simple-cauchy-complete}
The category $\simple(\Omega)$ is Cauchy complete.
\end{theorem}
\begin{proof}
Once again this is provable by passage to the Stone space $\Stone(\Omega)$.  Retracts of Banach spaces are complemented subspaces and complemented subspaces of $\Integrable_{\infty}(\Omega)$ are of the form $\Integrable_{\infty}(E)$ for $E\subseteq \Omega$.
\end{proof}

Cauchy completeness is just one strand in a web of connections involving distributors, Morita equivalence and Tannaka-Krein reconstruction via the center of categories\footnote{The \emph{center} of a category $\mathcal{A}$ is the commutative monoid of natural transformations $\id_{\mathcal{A}}\to \id_{\mathcal{A}}$. Since $\characteristic$ is dense, the center of $\Shv(\Omega)$ is isomorphic to the center of $\simple(\Omega)$ and by theorem \ref{theorem:morphism-spaces-constant-line-sheaves}, isomorphic to $\Integrable_{\infty}(\Omega)$.}. We refer the reader to \cite[chapter 5, section 8]{kelly:enriched-category-theory2005}, \cite[chapter 7, sections 8 and 9]{borceux:handbook-categorical-algebra11994} and above all, to \cite{lindner:morita-equivalences-enriched-categories1974} that contains many important results in the context of enriched categories that are directly useful for categorified measure theory. For Tannaka-Krein reconstruction, see \cite{joyal-street:introduction-tannaka-duality-quantum-groups1991}.

\subsection{The cosheafification functor}
\label{subsection:cosheafification-functor}

It is easy to see that $\CoShv(\Omega, \mathcal{A})$ is a Banach $2$-space with colimits computed pointwise, that is, the inclusion
\begin{equation} \label{map:cosheaf-inclusion-precosheaf}
    i\colon \CoShv(\Omega, \mathcal{A})\to \power{\mathcal{A}}{\Omega}
\end{equation}
is cocontinuous. In the case where $\mathcal{A}$ is the base category $\Ban$, $\CoShv(\Omega)$ has a small dense subcategory and by the adjoint functor theorem the inclusion has a right adjoint called the \emph{cosheafification} functor. In this section, we construct the cosheafification more directly, by taking appropriate limits of Banach spaces\footnote{On the assumption that $\mathcal{A}$ is complete, the limit construction of the cosheafification also yields a right adjoint to \eqref{map:cosheaf-inclusion-precosheaf}.}.

\smallskip
Recall that the \emph{variation} of an additive map $\mu$ is given by the (possibly infinite) quantity
\begin{equation} \label{equation:variation-additive-mapping}
    \variation{\mu}(E)= \sup\set{\sum_{F\in\mathcal{E}}\norm{\nu(F)}\colon \mathcal{E}\in \partitions(E)}
\end{equation}

The cosheafification functor can be seen a categorification of \eqref{equation:variation-additive-mapping}. More precisely, for every precosheaf $\mu$, define $\cosheafification{\mu}(E)$ as the limit of the functor $\partitions(E)\to \mathcal{A}$, that is
\begin{equation} \label{map:cosheafification}
    E\mapto \Lim_{\mathcal{E}\in \partitions(E)}\sum_{F\in \mathcal{E}}\mu(F)
\end{equation}

In other words, an element $\theta\in \cosheafification{\mu}(E)$ is a \emph{cogerm}, or a compatible choice of elements $\theta_{\mathcal{E}}\in \sum_{F\in \mathcal{E}}\mu(F)$.

\begin{theorem} \label{theorem:cosheafification}
The functor \eqref{map:cosheafification} is a right adjoint to the inclusion functor.
\end{theorem}
\begin{proof}
The fact that \eqref{map:cosheafification} indeed gives the cosheafification functor, is a straightforward dualization of the fact that
\begin{equation*}
    E\mapto \Colim_{\mathcal{E}\in \partitions(E)}\prod_{F\in \mathcal{E}}\mu(F)
\end{equation*}
yields the sheafification functor. For the latter, the details can be glimpsed from the already cited \cite[section 10]{rezk:fibrations-homotopy-colimits-simplicial-sheaves1998}.
\end{proof}

The universal property of the cosheafification $\cosheafification{\mu}$ is depicted in diagram \ref{diagram:universal-property-cosheafification}. Equivalently, we have the natural isomorphism,
\begin{equation} \label{adjunction:cosheafification-inclusion}
    \homset{\PCoShv}{i(\nu)}{\mu}\isomorphic \homset{\CoShv}{\nu}{\cosheafification{\mu}}
\end{equation}

\begin{figure}[htbp]
    \begin{equation*}
    \xymatrix{
            & \cosheafification{\mu} \ar[d]^{\varepsilon} \\
        \nu \ar[r]_{\tau} \ar@{-->}[ur] & \mu
    }
    \end{equation*}
\caption{Universal property of the cosheafification.}
\label{diagram:universal-property-cosheafification}
\end{figure}

Since $\Omega$ has a top element, also denoted by $\Omega$, it follows that the colimit functor
\begin{equation*}
    \Colim_{E\in \Omega}\mu(E)\isomorphic \mu(\Omega)
\end{equation*}
is just the evaluation of $\mu$ at $\Omega$. By definition of colimit, this yields the adjunction,
\begin{equation} \label{adjunction:evaluation-constant-presheaf}
    \homset{\Ban}{\mu(\Omega)}{B}\isomorphic \homset{\PCoShv}{\mu}{\diagonal B}
\end{equation}
where $\diagonal$ is the diagonal functor associating to each Banach space $B$ the corresponding constant precosheaf $E\mapto B$. Composing the adjunctions \eqref{adjunction:evaluation-constant-presheaf} and \eqref{adjunction:cosheafification-inclusion} we obtain the important result that evaluation is left adjoint to the constant cosheaf:
\begin{equation} \label{adjunction:evaluation-constant-cosheaf}
    \homset{\Ban}{\mu(\Omega)}{B}\isomorphic \homset{\CoShv}{\mu}{\cosheafification{B}}
\end{equation}

If $\mu$ is the cosheaf and we take the cosheaf
\begin{equation*}
    E\mapto \parens{\int\xi\differential{\mu}}E\defequal \int_{E}\xi\differential{\mu}\isomorphic \int_{\Omega}\characteristic(E)\otimes \xi\differential{\mu}
\end{equation*}
given by the indefinite integral of a sheaf, then its evaluation at $\Omega$ is simply its total direct integral and adjunction \eqref{adjunction:evaluation-constant-cosheaf} reads as:
\begin{equation} \label{adjunction:direct-integral-constant-cosheaf}
    \homset{\Ban}{\int_{\Omega}\xi\differential{\mu}}{B}\isomorphic \homset{\CoShv}{\int\xi\differential{\mu}}{\cosheafification{B}}
\end{equation}

Given these adjunctions and the fact that for a measure $\mu$ and a constant sheaf $\characteristic(E)$, $\int_{\Omega}\xi\differential{\mu}$ is the space $\Integrable_{1}(E, \mu)$, it is important to identify the class of constant cosheaves. This is done in the next subsection.

\subsection{The spectral measure of a cosheaf}
\label{subsection:spectral-measure-cosheaf}

In this subsection, we describe the fundamental construction of the \emph{spectral measure of a cosheaf}. The construction depends strongly on the fact that Banach $2$-spaces as we have defined them have linear spaces of morphisms and the arguments with direct sums in additive categories go through to some extent. We will restrict ourselves to finitely additive precosheaves, that is, cosheaves for the finite Grothendieck topology. All the essential ideas are already contained in this case including the description of the factorization $\simple(\Omega)\to \mathcal{A}$ as the integral against the spectral measure. The $\sigma$-additive case needs more analysis and measure theory and will be treated in the forthcoming \cite{rodrigues:categorified-measure-theory}.

\smallskip
Let $\mu$ be a cosheaf $\Omega\to \mathcal{A}$ and $F\subseteq E$ an element of $\Omega$. Since $\set{F, E\relcomplement F}$ is a partition of $E$, the cone $\parens{\mu_{F, E}, \mu_{E\relcomplement F, E}}$ is a coproduct. We can now define the \emph{cosheaf projections} as the unique arrows filling diagram \ref{diagram:projection-cosheaf}.
\begin{figure}[htbp]
    \begin{equation*}
    \xymatrix{
          & \mu(E\relcomplement F) & \\
        \mu(F) \ar[r]_{\mu_{F, E}} \ar[ur]^{0} \ar[dr]_{\id_{\mu(F)}} & \mu(E) \ar@{-->}[u]_-{p_{E, E\relcomplement F}} \ar@{-->}[d]_{p_{E, F}} & \mu(E\relcomplement F) \ar[l]^{\mu_{E\relcomplement F , E}} \ar[dl]^{0} \ar[ul]_{\id_{\mu(E\relcomplement F)}} \\
          & \mu(F) &
    }
    \end{equation*}
\caption{Cosheaf projections.}
\label{diagram:projection-cosheaf}
\end{figure}

It can be proved that $p_{E, F}$ is a presheaf on $\Omega$. In fact, putting (note the capitalization),
\begin{equation*}
    P_{E}\defequal \mu_{E, 1}p_{1, E}
\end{equation*}
we have:

\begin{theorem} \label{theorem:cosheaf-associated-spectral-measure}
The map $E\mapto P_{E}$ is a spectral measure on $\mu(\Omega)$. The induced Banach algebra morphism,
\begin{equation*}
    \Integrable_{\infty}(\Omega)\to \homset{\Ban}{\mu(\Omega)}{\mu(\Omega)}
\end{equation*}
provided by theorem \ref{theorem:universal-property-Linfty-algebra} is an isometry.
\end{theorem}
\begin{proof}
The proof is a series of routine diagram manipulations with coproducts in Banach $2$-spaces.
\end{proof}

Since $\mathcal{A}$ is cocomplete it has tensors, and the Banach algebra morphism induced by the spectral measure is equivalent to an arrow
\begin{equation} \label{map:spectral-action}
    \mu\colon \Integrable_{\infty}(\Omega)\otimes \mu(\Omega)\to \mu(\Omega)
\end{equation}
satisfying the axioms of a \emph{monoid action}, that is, we have the commutative diagrams \ref{diagram:unital-action} and \ref{diagram:associative-action} for the action map \eqref{map:spectral-action}, where the symbol $\cdot$ denotes the multiplication map in $\Integrable_{\infty}(\Omega)$.
\begin{figure}[htbp]
    \begin{equation*}
    \xymatrix{
        \field\otimes \mu(\Omega) \ar[r]^{\isomorphic} \ar[d]_{\id\otimes \id_{\mu(\Omega)}} & \mu(\Omega) \ar[d]^{\id_{\mu(\Omega)}} \\
        \Integrable_{\infty}(\Omega)\otimes \mu(\Omega) \ar[r]_-{\mu} & \mu(\Omega)
    }
    \end{equation*}
\caption{Unital law for the action $\mu\colon\Integrable_{\infty}(\Omega)\otimes \mu(\Omega)\to \mu(\Omega)$.}
\label{diagram:unital-action}
\end{figure}

\begin{figure}[htbp]
    \begin{equation*}
    \xymatrix{
        (\Integrable_{\infty}(\Omega)\projotimes
        \Integrable_{\infty}(\Omega))\otimes \mu(\Omega) \ar[rr]^{\isomorphic} \ar[d]_{\cdot\otimes 1_{\mu(\Omega)}} &  &
        \Integrable_{\infty}(\Omega)\otimes (\Integrable_{\infty}(\Omega)\otimes
        \mu(\Omega)) \ar[d]^{1_{\Integrable_{\infty}(\Omega)}\otimes \mu} \\
        \Integrable_{\infty}(\Omega)\otimes \mu(\Omega) \ar[dr]_{\mu} &  &
        \Integrable_{\infty}(\Omega)\otimes \mu(\Omega) \ar[dl]^{\mu} \\
          & \mu(\Omega) &
    }
    \end{equation*}
\caption{Associative law for the action $\mu\colon\Integrable_{\infty}(\Omega)\otimes \mu(\Omega)\to \mu(\Omega)$.}
\label{diagram:associative-action}
\end{figure}

Theorem \ref{theorem:cosheaf-associated-spectral-measure} places very strong constraints on the $\Integrable_{\infty}(\Omega)$-actions coming from cosheaves. This is not the place to provide a detailed exposition, suffice to say that the precise characterization of these actions is tied to a notion dual to that of \emph{local convexity}, a notion briefly mentioned in subsection \ref{subsection:banach-sheaves} in connection with the characterization of Banach sheaves as certain $\Continuous(X)$-Banach modules. It can be seen as a generalization to $\Continuous(X)$-Banach modules of the Banach lattice conditions involved in the Riesz-Kakutani duality. Riesz-Kakutani duality is discussed on any book on Banach lattices. Possible references are \cite{lacey:isometric-theory-classical-banach-spaces1974}, \cite[chapter 5]{fremlin:measure-theory32002}, \cite{luxemburg-zaanen:riesz-spacesI1971} and \cite{zaanen:riesz-spacesII1983}.

\smallskip
We need a more conceptual way to look at the spectral measure. By theorem \ref{theorem:cosheaf-associated-spectral-measure} it corresponds to an algebra morphism
\begin{equation} \label{map:spectral-measure-monoid-map}
    \Integrable_{\infty}(\Omega)\to \homset{\Ban}{\mu(\Omega)}{\mu(\Omega)}
\end{equation}

The sheaf $E\mapto \Integrable_{\infty}(E)$ is a monoid in $\Shv(\Omega)$. The presence of $\Omega$ on both the domain and the codomain of the spectral measure map, leads us to suspect that the right-hand side of \eqref{map:spectral-measure-monoid-map} can be made into a sheaf in such a way that \eqref{map:spectral-measure-monoid-map} is a monoid sheaf map. Let $\mu$ and $\nu$ be two cosheaves and $\tau\colon \mu\to \nu$ a cosheaf map. If $E\in \Omega$ and $\mathcal{E}$ is a partition of $E$ we have the commutative diagram \ref{diagram:induced-cosheaf-maps} for each $E_{n}\in \mathcal{E}$.
\begin{figure}[htbp]
    \begin{equation*}
    \xymatrix{
        \mu_{\mathcal{E}} \ar[rrr]^{\tau_{\mathcal{E}}} \ar[dr]_{\varepsilon_{\mathcal{E}}} &  &  & \nu_{\mathcal{E}} \ar[dl]^{\varepsilon_{\mathcal{E}}} \\
          & \mu(E) \ar[r]^{\tau_{E}} & \nu(E) & \\
        \mu(E_{n}) \ar[rrr]_{\tau_{E_{n}}} \ar[uu]^{i_{\mu(E_{n})}} \ar[ur]_{\mu_{E_{n}, E}} &  &  & \nu(E_{n}) \ar[uu]_{i_{\nu(E_{n})}} \ar[ul]^{\nu_{E_{n}, E}}
    }
    \end{equation*}
\caption{Induced cosheaf maps.}
\label{diagram:induced-cosheaf-maps}
\end{figure}

Since the maps $\varepsilon_{\mathcal{E}}$ are isomorphisms, it follows that,
\begin{equation*}
    \norm{\tau_{E}}= \max\set{\norm{\tau_{E_{n}}}\colon E_{n}\in \mathcal{E}}
\end{equation*}
and thus, the map $\tau_{E}\mapto \parens{\tau_{E_{n}}}$ is an isometric isomorphism. In other words, diagram \ref{diagram:induced-cosheaf-maps} provides us with projections,
\begin{equation*}
    \homset{\Ban}{\mu(E)}{\nu(E)}\to \homset{\Ban}{\mu(E_{n})}{\nu(E_{n})}
\end{equation*}
giving a sheaf $E\mapto \homset{\Ban}{\mu(E)}{\nu(E)}$. A more abstract way of seeing this is by using the end formula \eqref{isomorphism:end-formula-hom-functor-categories} and noting that
\begin{align*}
    \homset{\CoShv}{\mu}{\nu} &\isomorphic \homset{\PCoShv}{i\mu}{i\nu} \\
        &\isomorphic \int_{E\in \Omega}\Ban(i\mu(E), i\nu(E))
\end{align*}

By the preceding argument, this presheaf is actually a sheaf (for the appropriate Grothendieck topology) and since the composition map
\begin{equation*}
    \homset{\CoShv}{\mu}{\nu}\otimes \homset{\CoShv}{\nu}{\lambda}\to \homset{\CoShv}{\mu}{\lambda}
\end{equation*}
is a presheaf map, we have:

\begin{theorem} \label{theorem:cosheaves-enriched-sheaves}
Let $\mathcal{A}$ be a Banach $2$-space. The category $\CoShv(\Omega, \mathcal{A})$ is a $\Shv(\Omega)$-enriched category.
\end{theorem}

Since $\Integrable_{\infty}$ is the sheafification of the constant presheaf $E\mapto \field$, it follows that $\Integrable_{\infty}$ is initial in the category of monoids in $\Shv(\Omega)$. By theorem \ref{theorem:cosheaves-enriched-sheaves}, for each cosheaf $\mu$, the sheaf $\homset{\CoShv}{\mu}{\mu}$ is a monoid in $\Shv(\Omega)$ and thus there is a unique monoid map,
\begin{equation*}
    \Integrable_{\infty}\to \homset{\CoShv}{\mu}{\mu}
\end{equation*}
which of course, is just the spectral measure map of theorem \ref{theorem:cosheaf-associated-spectral-measure}.

\smallskip
Whenever we have a category $\mathcal{A}$ enriched in another category $\mathcal{B}$, the first question to answer is if $\mathcal{A}$ has $\mathcal{B}$-tensors. The next theorem could be called the fundamental theorem of categorified measure theory and says that $\CoShv(\Omega, \mathcal{A})$ has all $\Shv(\Omega)$-tensors (and thus it is cocomplete as a $\Shv(\Omega)$-enriched category) and that these are just the indefinite integral cosheaves. It provides the correct formulation of the abstract Radon-Nikodym property characterizing direct integrals via a universal property.

\begin{theorem} \label{theorem:cosheaves-sheaf-tensors-integral}
There is a natural isometric isomorphism
\begin{equation} \label{isomorphism:cosheaves-sheaf-tensors-integral}
    \homset{\CoShv}{\int\xi\differential{\mu}}{\nu}\isomorphic \homset{\Shv}{\xi}{\homset{\CoShv}{\mu}{\nu}}
\end{equation}
\end{theorem}

We end this subsection with two basic applications of the spectral measure construction: the description of the factorization $\simple(\Omega)\to \mathcal{A}$ for a cosheaf $\mu$ and the characterization of the constant cosheaves.

\smallskip
Let $\mu$ be a cosheaf of Banach spaces. We want to define a functor on the category  $\simple(\Omega)$ and with values on $\Ban$ that on objects is given by $\characteristic(E)\mapto \mu(E)$\footnote{The construction can be generalized in a straightforward fashion to a cosheaf $\Omega\to \mathcal{A}$.}. In order to understand what is $\mu(f)$ for a function $f\in \Integrable_{\infty}(E\cap F)\colon E\to F$ let us start by taking $\mu$ to be the $\Integrable_{1}(E, \mu)$-cosheaf induced by a finite, positive measure $\mu$. Thus, $\mu(f)$ is now a map
\begin{equation*}
    \mu(f)\colon \Integrable_{1}(E)\to \Integrable_{1}(F)
\end{equation*}

We can now hazard a guess for $\mu(f)$ as a composite of the form displayed in diagram \ref{diagram:mu-map}.
\begin{figure}[htbp]
    \begin{equation*}
    \xymatrix{
        \Integrable_{1}(E) \ar@{-->}[r] \ar[d]_{p_{E, E\cap F}} & \Integrable_{1}(F) \\
        \Integrable_{1}(E\cap F) \ar[r] & \Integrable_{1}(E\cap F) \ar[u]_{\mu_{E\cap F, F}}
    }
    \end{equation*}
\caption{The map $\mu(f)$.}
\label{diagram:mu-map}
\end{figure}

The outer terms just fiddle with the domain of the functions of $\Integrable_{1}(E)$ in order to make them land in $\Integrable_{1}(F)$. The inner term is given by composing $g\in \Integrable_{1}(E\cap F)$ with a map $\field\to \field$ obtained by some sort of integral $\int_{E\cap F}f$. Let $f$ be a simple function $E\cap F\to \field$ of the form,
\begin{equation*}
    f= \sum_{n}\characteristic(E_{n})k_{n}
\end{equation*}
with $\family{E}{n}$ a finite partition of $E\cap F$ and $k_{n}$ base field elements. Let $g\in \Integrable_{1}(E\cap F)$ be a simple function of the form $\sum_{m}\characteristic(F_{m})l_{m}$ with $\family{F}{m}$ a finite partition of $E\cap F$ and $l_{m}\in \field$. Since we expect that $\int f(g)= \parens{\int f}g$ we have
\begin{equation*}
\begin{split}
    \parens{\int f}g &= \int f(g) \\
        &= \sum_{n, m}\characteristic(E_{n}\cap F_{m})k_{n}(l_{m}) \\
        &= \sum_{n, m} P_{E_{n}}(\characteristic(F_{m}))k_{n}(l_{m}) \\
        &= \parens{\sum_{n}P_{E_{n}}k_{n}}\parens{\sum_{m}\characteristic(F_{m})l_{m}}
\end{split}
\end{equation*}

But the term $\sum_{n}P_{E_{n}}k_{n}$ is just the integral of $f$ against the spectral cosheaf measure. Plugging in the outer morphisms of diagram \ref{diagram:mu-map}, we arrive at the following definition.

\begin{definition} \label{definition:integral-cosheaf-map}
Let $\mu$ be a cosheaf of Banach spaces and $f\colon E\cap F\to \field$ a simple function of the form $\sum_{n}\characteristic(E_{n})k_{n}$ (or a \emph{simple morphism} $\characteristic(E)\to \characteristic(F)$). Then the induced map $\int f\differential{\mu}$ is
\begin{equation} \label{equation-definition:integral-cosheaf-map}
    \int f\differential{\mu}\defequal \sum_{n}k_{n}\parens{\mu_{E_{n}, F}p_{E, E_{n}}}
\end{equation}
\end{definition}

That definition \ref{definition:integral-cosheaf-map} is indeed functorial is the essential content of the next theorem.

\begin{theorem} \label{theorem:integral-cosheaf-map}
Let $\mu$ be a cosheaf of Banach spaces and $f\colon \characteristic(E)\to \characteristic(F)$ a simple morphism of characteristic sheaves. If $\int f\differential{\mu}$ is defined as in \eqref{equation-definition:integral-cosheaf-map} then:
\begin{equation*}
    \norm{\int f\differential{\mu}}= \supnorm{f}
\end{equation*}

Furthermore, if $g$ is a simple morphism $\characteristic(F)\to \characteristic(G)$ then:
\begin{align*}
    \int gf\differential{\mu} &= \int g\differential{\mu}\int f\differential{\mu} \\
    \int 1_{\characteristic(E)}\differential{\mu} &= \id_{\mu(E)}
\end{align*}
\end{theorem}
\begin{proof}
The proof is a series of routine manipulations with the associated spectral measure of a cosheaf. Full details in \cite{rodrigues:categorified-measure-theory}.
\end{proof}

Since the simple functions are dense in $\Integrable_{\infty}$, the functor $\simple(\Omega)\to \Ban$ can now be constructed in the obvious manner.

\smallskip
On to the task of characterizing the constant cosheaves. The poset of partitions $\partitions(E)$ is filtered and by construction of limits of Banach spaces, an element of $\Lim_{\mathcal{E}\in \partitions(E)}\sum_{F\in \mathcal{E}}\mu(F)$ is a compatible family $\lambda_{\mathcal{E}}\in \sum_{F\in \mathcal{E}}B$ with:
\begin{equation*}
    \sup\set{\sum_{F\in \mathcal{E}}\norm{\lambda_{F}}\colon \mathcal{E}\in \partitions(E)}< \infty
\end{equation*}

This suggests that the cosheafification of $B$ is the space of measures of \emph{bounded variation}. It is indeed, but let us define the term ``bounded variation'' first.

\begin{definition} \label{definition:bounded-variation}
The map $\nu$ has \emph{bounded variation} if for every $E\in \Omega$, its variation $\variation{\nu}(E)$ as defined in \eqref{equation:variation-additive-mapping} is finite.
\end{definition}

If $\nu$ has bounded variation, then its variation $\variation{\nu}(E)$ is a positive, \emph{finitely additive} measure on $\Omega$ that dominates $\nu$ with Lipschitz norm $1$. Denote by $\bvadditive(\Omega, B)$ the space of additive maps $\Omega\to B$ of bounded variation with the \emph{total variation} norm $\variation{\nu}(\Omega)$.

\begin{theorem} \label{theorem:bvadditive-complete}
The space $\bvadditive(\Omega, B)$ is complete.
\end{theorem}
\begin{proof}
Straightforward, since only finite additivity is involved.
\end{proof}

To understand functoriality of $\bvadditive$ on $\Omega$ start by identifying each $E\in \Omega$ with the principal ideal $\ideal(E)$ it generates. Each $\ideal(E)$ is a Boolean algebra with unit $E$ and the inclusion $\ideal(E)\to \Omega$ is a ring morphism that does \emph{not} preserve the unit (and thus, it is not a Boolean algebra morphism). If $E\subseteq F$, there is a Boolean algebra projection,
\begin{equation*}
    p_{F, E}\colon \ideal(F)\to \ideal(E)
\end{equation*}
given by intersection: $G\subseteq F\mapto G\cap E$. The presheaf $E\mapto \ideal(E)$ is a sheaf of Boolean algebras. The map
\begin{equation*}
    \bvadditive(E, B)\to \bvadditive(F, B)
\end{equation*}
is now given by pullback, that is, $\nu\mapto \parens{\pullback{p_{F, E}}\nu\colon G\mapto \nu(E\cap G)}$.

\begin{theorem} \label{theorem:cosheaf-bounded-variation}
Let $\Omega$ be a Boolean algebra and $B$ a Banach space. Then the precosheaf $E\mapto \bvadditive(E, B)$ is a cosheaf for the finite Grothendieck topology.
\end{theorem}
\begin{proof}
Fix a finite partition $\family{E}{n}$ of $E\in \Omega$. The inclusions $E_{n}\subseteq E$ induce inclusion maps $\additive(E_{n}, B)\to \additive(E, B)$. If $\nu\colon E\to B$ is an additive map and $F\subseteq E$, then the pullback of $\nu$ via the projection is,
\begin{equation*}
    \nu_{n}(F)= \nu(F\cap E_{n})
\end{equation*}
and we see that $\nu(F)= \sum_{n}\nu_{n}(F)$ so that the induced map
\begin{equation*}
    \bigoplus_{n}\additive(E_{n}, B)\to \additive(E, B)
\end{equation*}
is a linear isomorphism. To prove that this map descends to an isometric isomorphism on the spaces of bounded variation maps, let $\mathcal{F}_{n}$ be a finite partition of $E_{n}$. Then $\mathcal{F}= \bigcup_{n}\mathcal{F}_{n}$ is a finite partition of $E$ and
\begin{equation*}
    \sum_{n}\sum_{F\in \mathcal{F}_{n}}\norm{\nu_{n}(F)}= \sum_{F\in \mathcal{F}}\norm{\nu(F)}
\end{equation*}
so that $\sum_{n}\variation{\nu_{n}}(E_{n})\leq \variation{\nu}(E)$. On the other hand, if $\mathcal{F}$ is a partition of $E$ that refines $\mathcal{E}$, putting $\mathcal{F}_{n}= \set{F\in \mathcal{F}\colon F\subseteq E_{n}}$, then $\mathcal{F}_{n}$ is a partition of $E_{n}$ and
\begin{equation*}
    \sum_{F\in \mathcal{F}}\norm{\nu(F)}= \sum_{n}\sum_{F\in \mathcal{F}_{n}}\norm{\nu_{n}(F)}\leq \sum_{n}\variation{\nu_{n}}(E_{n})
\end{equation*}

Since the set of finite partitions refining $\mathcal{E}$ is cofinal in $\partitions(E)$, we have that $\variation{\nu}(E)\leq \sum_{n}\variation{\nu_{n}}(E_{n})$, so that the induced map
\begin{equation*}
    \sum_{n}\bvadditive(E_{n}, B)\to \bvadditive(E, B)
\end{equation*}
is an isometric isomorphism and $E\mapto \bvadditive(E, B)$ is an additive precosheaf on $\Omega$.
\end{proof}

Theorem \ref{theorem:cosheaf-bounded-variation} implies that not every cosheaf is of the form $\Integrable_{1}(\Omega, \mu, B)$ for a measure algebra $(\Omega, \mu)$ and a Banach space $B$. Just take $\bvadditive(\Omega, B)$ with $B$ a Banach space \emph{without} the Radon-Nikodym property (e.g.~$\Integrable_{1}(\Omega, \mu)$ for $(\Omega, \mu)$ an atomless $\sigma$-additive measure algebra).

\smallskip
Finally, we can state the characterization of the constant cosheaves as the spaces of bounded variation measures. In essence, it is an application of the spectral measure construction.

\begin{theorem} \label{theorem:constant-cosheaf-bvadditive}
Let $\theta$ be a cosheaf and $\tau\colon \theta\to B$ be a map between $\theta$ and the constant precosheaf $E\mapto B$. Then there is a unique map closing the triangle \ref{diagram:universal-property-bvadditive}.
\begin{figure}[htbp]
    \begin{equation*}
    \xymatrix{
            & \bvadditive(E, B) \ar[d]^-{\evaluation_{E}} \\
        \theta(E) \ar[r]_{\tau_{E}} \ar@{-->}[ur] & B
    }
    \end{equation*}
\caption{Universal property of $\bvadditive(\Omega, B)$.}
\label{diagram:universal-property-bvadditive}
\end{figure}
\end{theorem}
\begin{proof}
Let us sketch the construction of the associated map. Since $\theta$ is a cosheaf, by theorem \ref{theorem:cosheaf-associated-spectral-measure} there is an associated spectral measure $E\mapto P_{E}$. Let $\mu\in \theta(E)$. Then we have a finitely additive map given by:
\begin{equation} \label{map:induced-constant-cosheaf-map}
    F\mapto \tau_{F}(P_{F}\mu)
\end{equation}

The map that to $\mu$ associates \eqref{map:induced-constant-cosheaf-map} is easily seen to be a precosheaf contractive map. It lands on the space of bounded variation maps because the spectral measure has \emph{pointwise bounded variation},
\begin{equation*}
    \norm{\mu}= \norm{\sum_{F\in \mathcal{E}}P_{F}(\mu)}= \sum_{F\in \mathcal{E}}\norm{P_{F}(\mu)}
\end{equation*}
where the norm equalities come from the fact that $\theta(E)\isomorphic \sum_{F\in \mathcal{E}}\theta(F)$ and the fact that the inclusion cosheaf maps $\mu_{F, E}$ are \emph{isometries}.
\end{proof}

%The bibliography.
\bibliographystyle{amsalpha}
\bibliography{bibliography}
\end{document}